\documentclass[10pt]{amsart} 
\usepackage{amscd,amssymb,graphicx,colortbl}

\usepackage{amsfonts}
\usepackage{amsmath}
\usepackage{amsxtra}
\usepackage{latexsym}
\usepackage[mathcal]{eucal}  
\usepackage{enumitem} 
\usepackage{tikz}

\input xy
\xyoption{all}

\usepackage[pdftex,bookmarks,colorlinks,breaklinks]{hyperref}

\newtheorem{thm}{Theorem}[section]
\newtheorem{f}[thm]{Fact} 
\newtheorem{cor}[thm]{Corollary}
\newtheorem{lem}[thm]{Lemma}

\newtheorem{prop}[thm]{Proposition}

\theoremstyle{definition}
\newtheorem{defin}[thm]{Definition}

\theoremstyle{remark}
\newtheorem{remark}[thm]{Remark}
\newtheorem{remarks}[thm]{Remarks}

\newtheorem{ex}[thm]{Example}

\numberwithin{equation}{section}

\newcommand{\delete}[1]{} 
\newcommand{\nt}{\noindent}

\def\eps{{\varepsilon}}

\newcommand{\ben}{\begin{enumerate}}

\newcommand{\een}{\end{enumerate}}
\newcommand{\bit}{\begin{itemize}}

\newcommand{\eit}{\end{itemize}}

\def\R {{\mathbb R}}
\def\N {{\mathbb N}}
\def\Z {{\mathbb Z}}
\def\Q {{\mathbb Q}}

\def\T {{\mathbb T}}
\newcommand{\cU}{\mathfrak{U}}

\def\Aut{{\mathrm Aut}\,}
\def\St{{\mathrm St}\,}

\def\diam{{\mathrm{diam}}}

\def\B{{\mathcal{B}}}

\def\Xcal{\mathcal{X}} 
\def\Rcal{\mathcal{R}}

\def\Ucal{\mathcal{U}}

\def\Rcal {{\mathcal{R}}}

\def\LOTS{\operatorname{LOTS}}
\def\COTS{\operatorname{COTS}}
\def\GLOTS{\operatorname{GLOTS}}
\def\GCOTS{\operatorname{GCOTS}}

\def\a{\alpha}

\def\t{\tau}
\def\s{\sigma}
\def\g{\gamma}

\def\wrt{with respect to }

\newcommand{\cls}{{\rm{cl\,}}}

\newcommand{\al}{\alpha}

\newcommand{\medcirc}{\scalebox{0.5}[0.5]{\ensuremath{\bigcirc}}}

\hypersetup{
	linkcolor=blue,
	citecolor=red,
	urlcolor=red
}

\begin{document} 

\title[]
{Circular orders: topology and continuous actions} 

\author[]{Michael Megrelishvili}
\address{Department of Mathematics,
Bar-Ilan University, 52900 Ramat-Gan, Israel}
\email{megereli@math.biu.ac.il}
\urladdr{http://www.math.biu.ac.il/$^\sim$megereli}

\dedicatory{Dedicated to Eli Glasner on the occasion of his 80th birthday}
\thanks{Supported by the Gelbart Research Institute at the Department of Mathematics, Bar-Ilan  University}



\date{November, 2026} 

\begin{abstract}
We study the topology of circularly ordered sets. While the algebraic notion is classical, the general \emph{topological} theory has received comparatively little attention. In this work we provide a self-contained topological exposition and present several new directions and results:
\begin{itemize}
\item Initiate a systematic study of generalized circularly ordered topological spaces and of continuous group actions on them. 
\item Provide a convex uniform structure description of circularly ordered compactifications. This yields a topological and uniform analysis of Novak's algebraic regular completion. 
\item Demonstrate that this uniform-structure approach yields several new results in the theory of  $G$-compactifications for topological group actions on abstract ordered spaces.  
\item 
Reexamine functions of bounded variation on circularly ordered sets and prove generalizations of Helly's selection theorem (for circular and linear orders). 
\end{itemize}  
These developments and the systematic analysis of circular order topologies are motivated by recent applications in topological dynamics, particularly in joint works with Eli Glasner.  
\end{abstract}

\maketitle

\setcounter{tocdepth}{1}
\tableofcontents

\section{Introduction} \label{s:intro}

The concept of a circular (or cyclic) order has deep historical roots, with its axiomatic foundations established by E. V. Huntington \cite{Hunt16, Hunt24} and independently introduced by E. \v{C}ech \cite{Cech}.
There are several geometric roots of circular orders \cite{Struve}. 

By contrast, the \textit{topological} theory of abstract circularly ordered sets equipped with a ternary relation (Definition~\ref{newC}) has remained comparatively underdeveloped; an important early contribution is H. Kok’s 1973 monograph \cite{Kok}.
 For many years, topological and dynamical applications focused primarily on the standard circle $\T$ and its orientation-preserving homeomorphisms. 
This gap has become an obstacle in some recent works in topological dynamics \cite{GM-c, GM-tLN, GM-UltraHom21, GM-TC, Me-OrdSem}, where circular order topologies arise naturally. In joint works with Eli Glasner, circularly ordered compact $G$-systems (Definition \ref{d:c-ordG-sp}) were shown to have low dynamical complexity. Namely they are tame (see Section \ref{s:tame} for the definitions and properties).  
These applications motivate a self-contained  treatment of circularly ordered topological spaces (COTS) and their generalizations. 
We also believe that the topological theory of circular orders has its own significance in many other research directions.  

There is a natural ``bridge" connecting circularly  ordered topological spaces ($\COTS$) to the 
well-studied theory of linearly ordered topological spaces ($\LOTS$) \cite{Nach, Eng89, CR-b, Kok}. The main mechanism for this bridge is the \textbf{split space} construction, which is detailed systematically in Section \ref{s:split}.

However, one of the justifications for a separate $\COTS$ theory lies precisely where this analogy breaks down. In topological dynamics, the dynamical picture is more complex (see Remark \ref{r:c-Helly}). For compact \textbf{minimal} $G$-systems, linear orderability is a profound restriction, forcing the system to be trivial  \cite{Ellis, GM-c}.  
In contrast, circularly ordered  (minimal) compact $G$-spaces form a rich class. 
 Important examples include the Sturmian-type minimal $\Z^k$-systems, which are familiar in symbolic dynamics and are naturally circularly ordered \cite{GM-c, GM-tLN}. 
 This sharp contrast shows that a complete $\COTS$ theory is not merely a generalization of $\LOTS$ theory but a necessary and distinct field.  
\begin{itemize}[topsep=0pt] 
	\item Sections \ref{s:CircOrd} and \ref{s:CircTop} study the properties of the \textit{circular interval topology} $\lambda_{\medcirc}$ (Proposition \ref{Hausdorff}, Lemma \ref{l:c-is-Open}) and the precise relationship between $\lambda_{\medcirc}$ and the linear interval topology $\lambda_{\leq}$. 
	 \item Recall that the classical Helly space  $M_+([0,1],[0,1]) \subset [0,1]^{[0,1]}$ is first countable. In Theorem \ref{t:lin-Helly} we show that the \textit{generalized Helly space} $M_+(X,Y)$ of all LOP maps is first countable in the pointwise topology for compact metrizable LOTS $X$ and $Y$. In contrast, the circular analogue fails: $M_+(\T,\T)$ is not first countable (Example \ref{ex:BigHelly}). 
	This demonstrates once more 
	the relative complexity of circular orders 
	when compared with linear orders. 	
	\item Section \ref{s:compINV}  deals with special  compactifications of ordered spaces using inverse limits. Section \ref{s:split} develops the ``split space" bridge (Lemma \ref{c-doubling}), where we study the topological properties of COTS.  	
	\item Another key direction is the systematic study of \textit{Generalized Circularly Ordered Topological Spaces} ($\GCOTS$) in Section \ref{s:GCOTS}, a concept introduced in \cite{Me-OrdSem}. We establish its fundamental relationship to $\GLOTS$ (Lemma \ref{l:easyPropGCO}) and examine their equivariant compactifications.  
	\item Section \ref{s:Novak} contains a full topological analysis of Nov\'{a}k's regular completion \cite{Novak-cuts}. Using the  ``complete=compact" result \cite{Me-OrdSem}, we prove that this completion is a proper COTS compactification (Theorem \ref{t:CompletionEmbeddingAndDensity}) and that it is the \textit{minimal circularly ordered compactification}  (Theorem \ref{t:NovakMinimality}). This establishes a strong analogy to the Dedekind-MacNeille completion for $\LOTS$ 
	(Remark \ref{r:bingo}).  
	
	We study the major role of \textit{convex uniform structures} and show in Theorem \ref{t:UniformRep} that there is an order isomorphism
	$
	\operatorname{Comp}_{\mathrm{COP}}(X) \longleftrightarrow  \operatorname{Unif}_{\mathrm{GCO}}(X) 
	$
	between circular order preserving (COP) proper  compactifications and precompact convex uniform structures. 
	For every circular order we define the \textit{intrinsic} uniform structure (see \textit{interval uniformity} $\mu_{\medcirc}$ in Theorem \ref{t:NovakUniformityInternal}) whose uniform completion gives exactly Nov\'ak's completion.
	
	\item For every topological group $G$,
	Theorem~\ref{t:myCOTS-G-comp} shows that every $\COTS$
	$G$-space $X$ with a COP action admits a proper COP
	$G$-compactification, namely its minimal COP
	(Nov\'ak) compactification.
	By Corollary~\ref{c:GCOTSareG-compactifiable},
	the same existence result holds for every $\GCOTS$
	$G$-space with a COP action.
	
	If, in addition, $G$ is either locally compact or
	connected, then Theorem~\ref{t:G-bounded-GCOTS}
	implies that, for every compatible precompact convex
	uniformity $\Ucal$ on $X$ with $\Ucal$-uniformly
	continuous $g$-translations, the canonically extended
	$G$-action on the completion of $(X,\Ucal)$ is
	jointly continuous.
	For such groups $G$, the order isomorphism 
	$\operatorname{Comp}_{\mathrm{COP}}(X)
	\longleftrightarrow 
	\operatorname{Unif}_{\mathrm{GCO}}(X)$ 
	identifies the $G$-saturated precompact convex
	uniformities with the proper COP $G$-compactifications.
	
	\item 
	The final sections open some analytical directions.
	Section~\ref{s:fr} on fragmentability and
	Section~\ref{s:BV} on functions of bounded variation
	develop these concepts for abstract ordered spaces. 
	 We provide a generalization of the classical fact that BV (e.g., monotone) functions are Baire 1 by proving that order preserving maps into precompact convex uniform spaces are fragmented, yielding Baire class 1 functions on Polish domains (Theorem \ref{monot}). This framework leads to new generalizations of \textit{Helly's selection theorem}, proving that $BV_r$ spaces on abstract ordered sets are sequentially compact  under natural assumptions (Theorem \ref{corOfGenHelly2}, Theorem \ref{GenHellyThm}).  
	  \item 
	  Section~\ref{s:tame} applies these analytical tools to
	  dynamics, proving that all compact circularly ordered
	  $G$-spaces admit representations on Rosenthal Banach
	  spaces and are therefore dynamically tame
	  (Theorem~\ref{t:CoisWRN}). 
\end{itemize}

This paper provides a self-contained topological treatment of circularly ordered spaces. 
 Several of the presented results appeared earlier, occasionally without detailed proofs, in \cite{GM-tame,GM-c,GM-UltraHom21,GM-TC,Me-Helly,Me-medianBV,Me-OrdSem,Me-MaxEqComp}. We also provide new contributions, including the results numbered: 
\ref{p:Prop-c-ord}, \ref{t:M_Closed},  
\ref{c-doubling}, 
\ref{t:COTS_joint_cont}, 
 \ref{t:CompletionEmbeddingAndDensity}, 
\ref{t:NovakMinimality}, \ref{p:CompToCircle},
\ref{t:NovakUniformityInternal},  
 \ref{t:UniformRep}, 
\ref{t:myCOTS-G-comp}, 
\ref{c:GCOTSareG-compactifiable},  
\ref{t:G-bounded-GCOTS}, 
\ref{monot}, \ref{GenHellyThm}, 
\ref{t:BVisFr}, \ref{p:CorGCOTSisWRN}.   
 
\section{Circular and linear orders: basic properties} 
\label{s:CircOrd} 

The axiomatic approach (using ternary relations) to circular order goes back to E. V. Huntington \cite{Hunt16,Hunt24}. A geometric formulation using separation was later given by H.S.M. Coxeter. 
A cyclic ordering relation was independently introduced by E. Čech \cite{Cech}, which later inspired subsequent works of L. Rieger. 
A standard example of a circularly ordered space is the circle $\T$. An abstract circular (some authors prefer the term  \textit{cyclic}) order $R$ on a set $X$ 
can be defined as a certain ternary relation. 
Intuitively, a circular order can be seen as a linear order ``wrapped around" into a circle.  
We recall one of the main definitions of circular orders. 

\begin{defin} \label{newC} \cite{Hunt16,Hunt24,Cech} 
	Let $X$ be a set. A ternary relation $R \subset X^3$ on $X$ is said to be a {\it circular} (sometimes also called \emph{cyclic}) order (or, in short: \textit{c-order})   
	if the following four conditions are satisfied. It is convenient sometimes to write shortly $[a,b,c]$ (or, even simply $abc$) instead of $(a,b,c) \in R$. 
	\ben
	\item Cyclicity: 
	$[a,b,c] \Rightarrow [b,c,a]$;  
	
	\item Asymmetry: 
	$[a,b,c] \Rightarrow \neg [c,b,a]$ (or $(c,b,a) \notin R$);
	
	\item Transitivity:    
	$
	\begin{cases}
		[a,b,c] \\
		[a,c,d]
	\end{cases}
	$ 
	$\Rightarrow [a,b,d]$; 
	\item Totality: 
	if $a, b, c \in X$ are distinct, then $[a, b, c]$ 
	or $[a, c, b]$.  
	\een
	Then $(X,R)$ is a \textit{circularly ordered set}. 
	
	If $R$ satisfies the first three conditions (1), (2), (3) then $R$ is said to be a \textit{partial} circular order.   
\end{defin}

Observe that by (1) and (2) $[a,b,c]$ implies that $a,b,c$ are distinct. 

\begin{lem} \label{l:property}
	$
	\begin{cases}
		[c,a,x] \\
		[c,x,b]
	\end{cases}
	$ 
	$\Rightarrow [a,x,b]$. 
\end{lem} 
\begin{proof}
	By cyclicity, $[x,b,c]$ and $[x,c,a]$. Transitivity gives $[x,b,a]$, hence $[a,x,b]$ by cyclicity.
\end{proof}

Some alternative axiomatizations can be found in \cite{BORT}, \cite{Calegari04}, \cite{McMullen07}, \cite{BS}.  

For \emph{distinct} $a,b\in X$, define the
(oriented) \emph{intervals} by
$(a,b)_{\circ}:=\{x\in X:[a,x,b]\}$,
$[a,b]_{\circ}:=(a,b)_{\circ}\cup\{a,b\}$,
$[a,b)_{\circ}:=(a,b)_{\circ}\cup\{a\}$, and
$(a,b]_{\circ}:=(a,b)_{\circ}\cup\{b\}$.
For every $a\in X$, we also put
$(a,a)_{\circ}:=\varnothing$ and
$[a,a]_{\circ}:=\{a\}$.
The half-open circular interval notation
$[a,b)_{\circ}$ and $(a,b]_{\circ}$ is used only
when $a\ne b$.
Sometimes we drop the subscript when the context
is clear, or use $R$ instead of $\circ$.
For distinct $a,b\in X$, we have
$X\setminus[a,b]_R=(b,a)_R$.
 
For a circular order $R$ on $X$, the intersection of two open circular intervals is either empty, a single open circular interval, or a disjoint union of two open circular intervals.  
 
 \subsection*{Partial and linear orders} 
 
 By a \textit{partial order}, we mean a reflexive, antisymmetric, and transitive relation $\leq$. 
 A \textit{linear order}, as usual, is a partial order which is totally ordered, meaning that for distinct $a,b \in X$ 
 we have exactly one of the alternatives: $a<b$ or $b<a$. As usual, $a < b$  means that $a \leq b$ and $a\neq b$.  Sometimes we write just $(X,\leq)$, or even simply $X$, where no ambiguity can occur.  
 For every linearly ordered set $(X,\leq)$ define the rays 
 $$(a,\to):=\{x \in X: a < x\}, \ \ \ \ (\leftarrow,b):=\{x \in X: x <b\}$$ 
 with $a,b \in X$. 
 All such rays form a prebase (also called subbase) for the \textit{interval topology} $\lambda_{\leq}$. A topological space is said to be \emph{Linearly Ordered Topological Space} ($\mathrm{LOTS}$) if its topology is $\lambda_{\leq}$ for some linear order $\leq$. Recall the following well known fact.   

\begin{lem} \label{ordHausd}
	Let $(X,\le)$ be a $\LOTS$. For $u_1<u_2$ there are disjoint $\lambda_{\le}$-open neighborhoods $O_1,O_2$ with $O_1<O_2$ (i.e. $x<y$ for all $(x,y)\in O_1\times O_2$). In particular $(X,\lambda_{\le})$ is Hausdorff and the graph of $\le$ is closed in $X\times X$.
\end{lem}
\begin{proof}
	If $(u_1,u_2)=\emptyset$, take $O_1=(\leftarrow,u_2)$ and $O_2=(u_1,\to)$. Otherwise choose $t\in(u_1,u_2)$ and set $O_1=(\leftarrow,t)$, $O_2=(t,\to)$. For closedness: if 
	\(x\nleq y\), then $y<x$. If $(y,x) \neq \emptyset$ choose $t\in(y,x)$ and take neighborhoods $U=(t,\to)$ of $x$ and $V=(\leftarrow,t)$ of $y$. For the case $(y,x) = \emptyset$ take $U:=(y,\to), V:=(\leftarrow,x)$. 
	Then $U\times V$ is a neighborhood of $(x,y)$ disjoint from the graph of $\le$.
\end{proof}

 \begin{f} \label{f:propLOTS} 
 	Let $(X,\leq,\lambda_{\leq})$ be a $\LOTS$. 
 	\begin{enumerate}
 		\item \cite{HLZ} \ $X$ is monotonically normal (hence also, hereditarily collectionwise normal; in particular, hereditarily normal). 
 		\item The interval topology $\lambda_{\le}$ on $X$ is compact iff every subset of $X$ has a supremum (with $\sup(\emptyset)=\min X$) equivalently iff every subset has an infimum (with $\inf(\emptyset)=\max X$). 
 		\item \cite{LB,Nagata} $X$ is metrizable if (and only if) the diagonal $\Delta_X$ of $X^2$ is a $G_{\delta}$ subset. 
 	\end{enumerate}	
 \end{f} 
 
 \begin{defin} \label{d:ord} (Nachbin \cite{Nach})
 	Let $(X,\tau)$ be a topological space and $\leq$ 
 	a partial order on $X$. The triple 
 	$(X,\tau,\leq)$ is said to be a 
 	\emph{partially ordered space} (POTS) if the graph of the relation $\leq$ is $\tau$-closed in $X \times X$. 
 	We say that it is compact (separable, etc.) if 
 	$(X,\tau)$ is a compact (separable, etc.) space.  
 \end{defin}
 
 Every $(X,\tau,\leq)$ from POTS is Hausdorff because the diagonal $\Delta_X=\{(x,x)\}$ being intersection of two closed relations $\le$ and $\ge$, is closed in $X^2$. 
   The class POTS is preserved under the subspaces and the products $\prod_{i \in I} X_i$ (using the product topology and the natural coordinate-wise partial order:  $(a_i) \leq (b_i)$ iff $a_i \leq b_i$ for every $i$). 
   
   Any \emph{compact} $\LOTS$ is a
   compact $\rm{POTS}$ (in the sense of Definition \ref{d:ord}) by Lemma \ref{ordHausd}. Conversely, 
   for every compact POTS $(X,\tau,\leq)$, where $\leq$ is a linear order, $\tau=\lambda_{\le}$. 
  
  \begin{lem} \label{l:DenseLin} 
  	Let $(Y,\t,\leq)$ be a \rm{POTS}.  
  	Suppose that $X$ is a dense subset of $Y$ such that the restricted partial order $\leq_X$ on $X$ is a linear order. Then $\leq$ is a linear order on $Y$.  
  \end{lem} 
  \begin{proof}
  	The relation $R:=\leq$ is a closed  subset of $Y \times Y$. The union $R \cup R^{-1}$ is closed in $Y \times Y$. The subset $X \times X$ is dense in $Y \times Y$ and is contained in $R \cup R^{-1}$. Hence, 
  	$R \cup R^{-1} = Y \times Y$.  	
  \end{proof}
  
 A map $f\colon  (X,\leq) \to (Y,\leq)$ between two (partially) ordered sets is said to be \emph{order preserving} (OP) or  \textit{increasing} 
 if $x \leq x'$ implies $f(x) \leq f(x')$ for every $x,x' \in X$. Let $(X,\leq)$ and $(Y,\leq)$ be partially ordered sets.  
 Denote by $M_+(X,Y)$ the set of all order preserving maps $X \to Y$. For $Y=\R$ we use the symbol 
 $M_+(X,\leq)$ or $M_+(X)$. 
 Since the order of $\R$ is closed in $\R^2$, we have $\cls(M_+(X)) = M_+(X)$. That is, $M_+(X)$ is pointwise closed in $\R^X$. 
 If $(Y,\tau,\leq)$ is a compact partially ordered space then 
 $M_+(X,Y)$ is pointwise closed in $Y^X$. 
 For compact partially ordered spaces $X,Y$ we define also $C_+(X,Y)$ the set 
 of all continuous and OP maps $X \to Y$. 
 
 \begin{lem}\label{l:Nachbin}\emph{(Nachbin \cite[Section 3, Thm. 6]{Nach})}
 	Let $(X,\tau,\le)$ be a compact POTS. Then $C_+(X,[0,1])$ separates points of $X$. Moreover, if $A\subseteq X$ is closed and $f\colon A\to \R$ is continuous and order preserving, there exists a continuous order-preserving $F\colon X\to \R$ with $F|_A=f$.
 \end{lem}
 
\begin{remark} \label{r:chech} \ 
	\ben 
	\item 
	Every linear order $\leq$ on $X$ defines a \emph{standard circular order} $R_{\leq}$ on $X$ as follows: 
	$[x,y,z]$ iff one of the following conditions is satisfied:
	$x < y < z, \ y < z < x, \  z < x < y.$	 
	
	\item  \cite[p.~35]{Cech} (standard cuts) Let $(X,R)$ be a c-ordered set and let $z\in X$.
	Define a binary relation $\le_z$ on $X$ by declaring that $z\le_z x$ for every $x\in X$, and
	for $a,b\in X\setminus\{z\}$ put
	\[
	a<_z b \ \Longleftrightarrow\ [z,a,b].
	\]
	Then $\le_z$ is a linear order on $X$ with least element $z$. The associated circular
	order $R_{\le_z}$ coincides with $R$.

	\item For every distinct $a,b$ in a circularly ordered set $(X,R)$, the interval
	$
	[a,b]_R
	$
	is linearly ordered by the restriction of the standard cut $\le_a$. Equivalently,
	$
	[a,b]_R=[a,b]_{\le_a}.
	$
	In particular, for $u,v\in (a,b)_R$ one has
	$
	u<_a v \ \Longleftrightarrow\ [a,u,v].
	$ 
	\een
\end{remark}

\section{Topology of circular orders} \label{s:CircTop} 

The following natural topologization of circular orders mimics interval topology of linear orders.  

 \begin{prop}[Interval topology] \label{Hausdorff} \cite{GM-c,GM-UltraHom21}
 	\ben  
 	\item 
 	For every circular order $R$ on $X$ the family of subsets
 	$${\mathcal B}_1:=\{X \setminus [a,b]_R : \ a,b \in X\} \cup \{X\}$$  
 	forms a base for a topology $\lambda_R$ on $X$ which we call the \emph{interval topology} of $R$. 
 	\item If $X$ contains at least three elements, the (smaller) family of intervals  
 	$${\mathcal B}_2:=\{(a,b)_R : \  a,b \in X, a \neq b\}$$ 
 	forms a base for the same topology $\lambda_R$ on $X$. 
 	\item The interval topology $\lambda_R$ of every circular order $R$ is Hausdorff.  
 	\een 
 	\textbf{ Notation}: Sometimes we use the notation $\lambda_{\medcirc}$ instead of $\lambda_R$, when the context is clear. 
 \end{prop} 
 \begin{proof} 
 	(1) and (2) If $|X|\le 2$, then the claim is immediate: $\mathcal B_1$ is a base for the discrete
 	topology, while (2) is vacuous. So assume that $|X|\ge 3$.
 	
 	For distinct $a\neq b$ we have
 	$
 	X\setminus [a,b]_R=(b,a)_R,
 	$
 	whereas for $a=b$,
 	$
 	X\setminus [a,a]_R=X\setminus\{a\}.
 	$
 	Hence
 	$
 	\mathcal B_1=\mathcal B_2\cup\{X\setminus\{a\}:a\in X\}\cup\{X\}.
 	$ In particular, $\mathcal B_2 \subseteq \mathcal B_1$.  
 	We claim that every set $X\setminus\{a\}$ is open in the topology generated by $\mathcal B_2$.
 	Indeed, let $b\in X\setminus\{a\}$. Choose $c\in X\setminus\{a,b\}$. By totality, either $[a,b,c]$
 	or $[a,c,b]$. In the first case $b\in (a,c)_R\subset X\setminus\{a\}$, and in the second case
 	$b\in (c,a)_R\subset X\setminus\{a\}$. Thus every point of $X\setminus\{a\}$ is contained in some
 	member of $\mathcal B_2$ lying in $X\setminus\{a\}$, so
 	$
 	X\setminus\{a\}=\bigcup\{U\in\mathcal B_2:U\subset X\setminus\{a\}\}.
 	$
 	Therefore $\mathcal B_1$ and $\mathcal B_2$ generate the same topology.

 	Clearly, $\mathcal B_2$ covers $X$: for every $x\in X$
 	choose distinct $a,b\in X\setminus\{x\}$ and interchange
 	$a,b$ if necessary so that $[a,x,b]$. 
 	It remains to show that $\mathcal B_2$ is a base. Let
 	$
 	A_1=(a_1,b_1)_R, \ A_2=(a_2,b_2)_R,
 	$
 	and $x\in A_1\cap A_2$. As observed after Lemma \ref{l:property}, the intersection of two open circular
 	intervals is either empty, a single open circular interval, or a disjoint union of two open circular
 	intervals. Hence there exists $U\in\mathcal B_2$ such that
 	$x\in U\subset A_1\cap A_2.$
 	
 	(3) Let $R$ be the circular order on $X$ and $\lambda_{\medcirc}$ the corresponding topology.
 	Let $a\neq b$. We have to find two disjoint neighborhoods $U$ and $V$ of $a$ and $b$ respectively.
 	If $X=\{a,b\}$ then the proof is trivial because $\{a\}=X\setminus\{b\}\in\mathcal{B}_1$ and
 	$\{b\}=X\setminus\{a\}\in\mathcal{B}_1$. Assume that $|X|\ge 3$ and choose $d\in X\setminus\{a,b\}$.
 	By the totality axiom, either $[b,d,a]$ or $[a,d,b]$. We consider only the first case; the second one is similar.
 	There are two subcases:
 	
 	1) $(a,b)\neq\emptyset$. Choose $c\in(a,b)$ and define $U:=(d,c)$, $V:=(c,d)$.
 	Then $[a,c,b]$ and $[b,d,a]$ yield $[b,a,c]$ and $[a,b,d]$ by cyclicity; hence
 	Lemma \ref{l:property} implies $[d,a,c]$ and $[c,b,d]$.
 	Therefore $a\in(d,c)=U$ and $b\in(c,d)=V$.
 	Moreover $U\cap V=\emptyset$ since $x\in(d,c)\cap(c,d)$ would give both $[d,x,c]$ and $[c,x,d]$,
 	contradicting asymmetry axiom.
 	
 	2) $(a,b)=\emptyset$.
 	Define $U:=(d,b)$, $V:=(a,d)$.
 	Since $(a,b)=\emptyset$ and $d\neq a,b$, totality yields $[a,b,d]$ (so $[d,a,b]$ by cyclicity),
 	hence $a\in(d,b)=U$ and $b\in(a,d)=V$.
 	If $x\in U\cap V$, then $[d,x,b]$ and $[a,x,d]$, and also $[a,b,x]$ (because $(a,b)=\emptyset$).
 	Applying Lemma \ref{l:property} to $[a,b,x]$ and $[a,x,d]$ we get $[b,x,d]$, i.e. $[d,b,x]$ by cyclicity,
 	a contradiction to $[d,x,b]$.
 	Thus $U\cap V=\emptyset$. 
 \end{proof}
 
 In contrast to the interval topology of linear orders, the topology of circular orders has received comparatively little attention in the literature, with the notable exception of H. Kok's monograph \cite{Kok}. 
More precisely, we record here a text from \cite[page 6]{Kok}: ``a topological space $(X,\tau)$ is said to be \textit{strictly cyclically orderable} 
if there exists a cyclic ordering $R$ on $X$ such that the intervals $(a,b)_R$ (with $a,b \in X$) form a base for the topology $\tau$. In the weaker case that all intervals are $\tau$-open, $X$ is called \textit{cyclically orderable}.  This is similar (but not equivalent) to the definition of generalized circularly ordered space below)." Note that we do not use these definitions below. 

  If $X$ contains only two or one element then every $(a,b)_R$ is empty. In this case ${\mathcal B}_2$ is not a topological base at all. 
 As we have seen in Proposition \ref{Hausdorff}, the family 
 ${\mathcal B}_2$ generates a topology on $X$ under a (minor) assumption that $X$ contains at least 3 elements. 
 
 In every circularly ordered set $(X,\circ)$ we have 
 $X \setminus (a,b)_{\circ}=[b,a]_{\circ}$ for every distinct $a \neq b$. So the ``circular closed interval" $[b,a]_{\circ}$ is always closed in the interval topology for all, not necessarily distinct, $a ,b$ (observe that the singleton $\{a\}$ is closed by Proposition \ref{Hausdorff}.3).    

 \begin{lem} \label{l:c-is-Open} \cite{GM-c}
 	Let $R$ be a circular order on $X$ and $\lambda_{\medcirc}$ the induced interval  topology. Then 
 	\begin{enumerate}
 		\item $R$ is an open subset of $\widetilde{X^3}$, where 
 		$\widetilde{X^3}:= \{(a,b,c) \in X^3: a,b,c \ \text{are pairwise distinct}\}.$  
 		More precisely, 
 		 for every  $[a,b,c]$ there exist disjoint  neighborhoods $U_1,U_2,U_3$ of $a,b,c$ respectively such that $[U_1,U_2,U_3]$ meaning that $[a',b',c']$ for every $(a',b',c') \in U_1 \times U_2 \times U_3$. One may generalize this to the case of any finite cycle $[c_1,c_2, \cdots, c_m]$.   
 		\item $R$ is a clopen subset of $\widetilde{X^3}$. In fact, $\widetilde{X^3}=R \cup R^*$ is the topological sum, where $R^*$ is the ``opposite circular relation".   
 	\end{enumerate} 
 \end{lem} 
 \begin{proof} (1) 
 	We have four (up to equivalence; from eight) cases:
 	
 	(a) 	
 	$
 	\begin{cases}
 		\exists x \in (a,b) \neq \emptyset \\
 		\exists y \in (b,c) \neq \emptyset \\
 		\exists z \in (c,a) \neq \emptyset
 	\end{cases}
 	$ 
 	
 	Then $a \in U_1:=(z,x), b \in U_2:=(x,y), c \in U_3:=(y,z)$ are the desired neighborhoods. 
 	
 	(b) 
 	$
 	\begin{cases}
 		(a,b) = \emptyset \\
 		\exists y \in (b,c) \neq \emptyset \\
 		\exists z \in (c,a) \neq \emptyset
 	\end{cases}
 	$ 
 	
 	Then choose $a \in U_1:=(z,b), b \in U_2:=(a,y), c \in U_3:=(y,z)$. 
 	
 	(c) 
 	$
 	\begin{cases}
 		(a,b) = \emptyset \\
 		(b,c) = \emptyset \\
 		\exists z \in (c,a) \neq \emptyset
 	\end{cases}
 	$ 
 	
 	Then choose $a \in U_1:=(z,b), U_2:=(a,c)=\{b\}, c \in U_3:=(b,z)$. 
 	
 	(d) 
 	$
 	\begin{cases}
 		(a,b) = \emptyset \\
 		(b,c) = \emptyset \\
 		(c,a) = \emptyset
 	\end{cases}
 	$  
 	
 	Then simply choose $\{a\} = U_1:=(c,b), \{b\}=U_2:=(a,c), \{c\}=U_3:=(b,a)$.  
 	
 	(2) It is a corollary of (1) using the totality axiom. 
 \end{proof}
 
 \textbf{Notation}.  
 Denote by COTS the class of all topological spaces $(X,\tau)$ for which $\tau=\lambda_R$ for some circular order $R$.  
 By \rm{comp-LOTS} (resp. \rm{comp-COTS}) we mean the subcollection of compact members of \textsc{LOTS} (resp. \textsc{COTS}).
 
  
  \begin{prop} \label{inclusion}  
  	\rm{comp-LOTS} $\subset$ \rm{comp-COTS}.  
  	More precisely: every compact linearly ordered space 
  	is a circularly ordered space  
  	\wrt the canonically associated circular order. 
  \end{prop}
  \begin{proof}
  	Let $\leq$ be a linear order on $X$ such that the interval topology $\lambda_{\leq}$ is compact. Then $\lambda_{\medcirc} \subseteq \lambda_{\leq}$ by Lemma \ref{l:cWEAKER}.1 below. 
  	On the other hand, $\lambda_{\medcirc}$ is Hausdorff by Proposition \ref{Hausdorff}. Since $\lambda_{\leq}$ is compact, we get 
  	$\lambda_{\leq} = \lambda_{\medcirc}.$ 
  \end{proof}
  
  The compactness of $\lambda_{\leq}$ is essential. Indeed, the induced circular order topology of $[0,1)$ is naturally homeomorphic to the circle. 
  This gives a justification of the standard identification of the \emph{sets} $\T$ and (c-ordered) $[0,1)$. 
  So, it is not true that $\mathrm{LOTS} \subset \mathrm{COTS}$ even on the topology level. It turns out that the end-points control the situation. 
 
 	\begin{lem} \label{l:cWEAKER} Let $\leq$ be a linear order on $X$ and $R:=R_{\leq}$ is the corresponding circular order. Then
 	\begin{enumerate}
 		\item Always $\lambda_R \subseteq \lambda_{\leq}$.
 		\item $\lambda_R = \lambda_{\leq}$ if and only if one of the following conditions is satisfied:
 		\begin{itemize}
 			\item $(X,\leq)=[u,v]$ has the minimum $u$ and the maximum $v$; 
 			\item $(X,\leq)$ has neither a minimum nor a maximum;
 		\end{itemize}
 	\item $\lambda_R \neq \lambda_{\leq}$	if and only if one of the following conditions is satisfied: 
 	\begin{itemize}
 		\item $(X,\leq)=[u,\to)$ has the minimum $u$ but not the maximum;
 		\item $(X,\leq)=(\leftarrow,v]$ has the maximum $v$ but not the minimum.
 	\end{itemize} 	
 	\end{enumerate}	  
 	\end{lem}
 	\begin{proof}
 		(1) We show that every basic $\lambda_R$-open interval is $\lambda_{\le}$-open.
 		For distinct $a,b\in X$ we have
 		\[
 		(a,b)_R=
 		\begin{cases}
 			(a,b)_{\le}, & \text{if } a<b,\\[2mm]
 			(\leftarrow,b)_{\le}\cup (a,\to)_{\le}, & \text{if } b<a.
 		\end{cases}
 		\]
 		Hence $(a,b)_R\in \lambda_{\le}$ for all $a\neq b$, and therefore
 		$
 		\lambda_R\subseteq \lambda_{\le}.
 		$
 		
 		(2) First assume that $(X,\le)=[u,v]$ has both a minimum $u$ and a maximum $v$.
 		Then for every $x\in X$,
 		$
 		(\leftarrow,x)_{\le}=(v,x)_R 
 		\ \text{and} \ 
 		(x,\to)_{\le}=(x,u)_R.
 		$
 		Thus every subbasic $\lambda_{\le}$-open ray is $\lambda_R$-open, so
 		$
 		\lambda_{\le}\subseteq \lambda_R.
 		$
 		Together with (1), this yields $\lambda_R=\lambda_{\le}$.
 		
 		Now assume that $(X,\le)$ has neither a minimum nor a maximum.
 		For every $a\in X$,
 		\[
 		(a,\to)_{\le}=\bigcup_{x>a}(a,x)_{\le}
 		=\bigcup_{x>a}(a,x)_R,
 		\]
 		because $a<x$ implies $(a,x)_R=(a,x)_{\le}$.
 		Similarly,
 		\[
 		(\leftarrow,a)_{\le}=\bigcup_{x<a}(x,a)_{\le}
 		=\bigcup_{x<a}(x,a)_R.
 		\]
 		So every subbasic $\lambda_{\le}$-open ray is $\lambda_R$-open, hence
 		$
 		\lambda_{\le}\subseteq \lambda_R.
 		$
 		Again, by (1), $\lambda_R=\lambda_{\le}$.
 		
 		(3) Assume that $(X,\le)=[u,\to)$ has a minimum $u$ but no maximum.
 		Fix $a>u$ and consider the $\lambda_{\le}$-open neighborhood
 		$
 		U=[u,a)_{\le}
 		$
 		of $u$. We claim that $U\notin \lambda_R$. 
 		Indeed, let $V=(b,c)_R$ be any basic $\lambda_R$-neighborhood of $u$.
 		Since $u\in V$, necessarily $c<b$, and therefore
 		$
 		V=(\leftarrow,c)_{\le}\cup (b,\to)_{\le}.
 		$
 		Because $u$ is the minimum, this is
 		$
 		V=[u,c)_{\le}\cup (b,\to)_{\le}.
 		$
 		Since $X$ has no maximum, the ray $(b,\to)_{\le}$ is nonempty; in fact, it contains points
 		larger than $a$. Hence $V\nsubseteq U$. Therefore no basic $\lambda_R$-neighborhood of $u$
 		is contained in $U$, so $U$ is not $\lambda_R$-open. Thus
 		$
 		\lambda_R\ne \lambda_{\le}.
 		$
 		
 		The case where $(X,\le)=(\leftarrow,v]$ has a maximum but no minimum is symmetric. 
 	\end{proof}
 
 The circle $\T$ is a COTS but not a LOTS. The space $[0,1)$ is a LOTS but not a COTS, \cite{Kok}. Hence, 
 after the definition of GLOTS and GCOTS (Section \ref{s:GCOTS}), these examples imply 
 $\LOTS$ $\subsetneq$ $\GLOTS$ and $\COTS$ $\subsetneq$  $\GCOTS$.

 	\begin{defin} \label{d:cutsNovak} 
 	Following Nov\'ak \cite{Novak-cuts}, we define cuts in c-ordered sets. 
 	Let $(X,\circ)$ be a c-ordered set. A linear order $\leq$ on the set $X$ is said to be a {\it cut} on $(X,\circ)$ if 
 	$$
 	a < b < c \ \text{in} \ (X, \leq) \ \text{implies that} \  [a,b,c] \ \text{in} \ (X,\circ). 
 	$$
 	\end{defin}
 	
 		For every linearly ordered set $(X,\leq)$ and its associated circular order $R_{\leq}$  (Remark \ref{r:chech}(1)) the given linear order $\leq$ is a cut on $(X,R_{\leq})$. 
 	Every point $z$ of a circularly ordered set $(X,R)$ defines a standard pointed cut $\leq_z$ on $X$ as it was defined in Remark \ref{r:chech}(2). 
 
 \begin{lem} \label{l:cut} 
 	Let $\leq$ be a cut on $(X,\circ)$.   
 	\begin{enumerate} 
 		\item \cite[Lemma 2.2]{Novak-cuts} If $[a,b,c]$ then either $a<b<c$ or $b<c<a$ or $c < a<b$. 
 		\item Let $a<b$. Then $(a,b)_{\leq}=(a,b)_{\circ}$, $[a,b]_{\leq}=[a,b]_{\circ}$,  $[a,b)_{\leq}=[a,b)_{\circ}$, $(a,b]_{\leq}=(a,b]_{\circ}$. 
 		\item Let $a<b$. For every interval
 		$
 		Y\in\{[a,b]_{\circ},[a,b)_{\circ},
 		(a,b)_{\circ},(a,b]_{\circ}\},
 		$
 		the subspace topology $Y|_{\lambda_{\medcirc}}$
 		inherited from $\lambda_{\medcirc}$ is the same as
 		the subspace topology $Y|_{\lambda_{\leq}}$
 		inherited from $\lambda_{\leq}$. 
 	\end{enumerate}
 \end{lem}
 
 \begin{proof}
 	(2) It is enough to show that $(a,b)_{\leq}=(a,b)_{\circ}$. 
 	
 	Let $x \in (a,b)_{\leq}$. Then by the definition of a compatible cut we have $[a,x,b]$. Hence, $x \in (a,b)_{\circ}$. 
 	
 	Conversely, let $x \in (a,b)_{\circ}$. Then by (1), either $a<x<b$ or $x<b<a$ or $b < a<x$. By our assumption, $a<b$. So, we necessarily have $a<x<b$. Hence, $x \in (a,b)_{\leq}$. 
 	
 	(3) 
 	We explain only the case of $Y:=[a,b]_{\circ}$, where $a<b$.  
 	First, we show that $Y|_{\lambda_{\leq}} \subseteq Y|_{\lambda_{\medcirc}}$ by verifying that the intersections of $Y$ with the subbasic open rays of $\lambda_{\leq}$ are open in $Y|_{\lambda_{\medcirc}}$. Note that for every $x \in [a,b]_{\circ}$:
 	$$(\leftarrow,x)_{\leq} \cap [a,b]_{\circ} = [a,x)_{\leq} = (b,x)_{\circ} \cap [a,b]_{\circ}$$ 
 	$$(x,\rightarrow)_{\leq} \cap [a,b]_{\circ} = (x,b]_{\leq} = (x,a)_{\circ} \cap [a,b]_{\circ}.$$
 	Since $(b,x)_\circ$ and $(x,a)_\circ$ are open in $\lambda_{\medcirc}$, these intersections are relatively open in $Y|_{\lambda_{\medcirc}}$. This proves the inclusion $Y|_{\lambda_{\leq}} \subseteq Y|_{\lambda_{\medcirc}}$. 
 	
 	In order to see the converse inclusion $Y|_{\lambda_{\medcirc}} \subseteq Y|_{\lambda_{\leq}}$, consider a basic open set $(c,d)_\circ$ of $\lambda_{\medcirc}$. We must show its intersection with $Y$ is relatively open in $Y|_{\lambda_{\leq}}$. 
 	If $c < d$, then by (2) we have $(c,d)_\circ = (c,d)_{\leq}$. Thus, 
 	$(c,d)_\circ \cap Y = (c,d)_{\leq} \cap Y,$ 
 	which is clearly open in $Y|_{\lambda_{\leq}}$. 
 	If $d < c$, then $(c,d)_\circ = (c,\rightarrow)_{\leq} \cup (\leftarrow,d)_{\leq}$. Consequently,
 	$$
 	(c,d)_\circ \cap Y = \big( (c,\rightarrow)_{\leq} \cap Y \big) \cup \big( (\leftarrow,d)_{\leq} \cap Y \big),
 	$$
 	which is a union of two relatively open sets in $Y|_{\lambda_{\leq}}$. 
 	
 	For other intervals $Y$, the proof is similar. 
 \end{proof}

 \subsection*{Cycles and order preserving maps}
 On the set $\{0, 1, \cdots, n-1\}$ consider (uniquely defined up to isomorphism) the standard c-order modulo $n$. Denote this c-ordered set simply by $C_n$. 
 
 \begin{defin} \label{d:cycl}  Let $(X,R)$ be a c-ordered set. We say that a vector
 	$(x_1,x_2, \cdots, x_n) \in X^n$ is a \textit{cycle} in $X$ if it satisfies the following two conditions:  
 	\begin{enumerate} 
 		\item For every $[i,j,k]$ in $C_n$ and \textit{distinct} 
 		$x_i, x_j, x_k$ we have $[x_i,x_j,x_k]$; 
 		\item $x_i=x_k \ \Rightarrow$  
 		$(x_i=x_{i+1}= \cdots =x_{k-1}=x_k) \ \vee \ (x_k=x_{k+1}=\cdots =x_{i-1}=x_i).$  
 	\end{enumerate}  
 	If otherwise is not stated, we consider only \textit{injective} cycles (all $x_i$ are distinct). 
 \end{defin}
 
 \begin{defin} \label{d:c-ordMaps} 
 	A function $f \colon X_1 \to X_2$ between c-ordered sets $(X_1,R_1)$ and $(X_2,R_2)$ is said to be {\it c-order preserving} (or, in short:  {\it COP}), 
 	if it satisfies the following two conditions:
 	
 	\begin{itemize} 
 		\item [(COP1)] 
 		 For every $[a,b,c]$ in $X_1$ and  \textit{pairwise distinct} $f(a), f(b), f(c)$ we have $[f(a), f(b), f(c)]$;
 	
 		\item [(COP2)] If $f(x)=f(y)$ then $f$ is constant on one of the closed intervals $[x,y], \ [y,x]$. 
 	\end{itemize}   
 \end{defin}
 
Every COP $f$ moves every cycle to a cycle. 
 
 For (total) c-ordered sets (COP1) is equivalent (by totality and asymmetry) to: 
 for every cycle $[f(a), f(b), f(c)]$ in $X_2$ we have $[a,b,c]$ in $X_1$; equivalently, 
 $(f^3)^{-1}(R_2) \subseteq R_1$ (where $f^3:=f \times f \times f$). 
 
 In general, condition (COP1) does not imply condition (COP2). Indeed, consider a 4-element cycle $X=Y=\{1,2,3,4\}$ and a selfmap $p: X \to X, p(1)=p(3)=1, p(2)=p(4)=2$. Then (COP1) is trivially satisfied but (COP2) fails. However in Proposition \ref{p:Prop-c-ord}.1 we show that (COP1) implies (COP2) for every $p(X)$ which 
 is not a two-point set. 
 
 For every linear order preserving  map $(X_1,\leq_1) \to (X_2,\leq_2)$ the map $(X_1,R_{\leq_1}) \to (X_2,R_{\leq_2})$ (between the corresponding c-ordered sets) is also c-order preserving.

 Let $M_+(X_1,X_2)$ be the collection of c-order preserving
 maps from $X_1$ into $X_2$.  A composition of c-order preserving maps is c-order preserving. Therefore, $M_+(X,X)$ is a semigroup under the composition (with the identity $id_X$) for every c-ordered $X$. 
 
 A COP function $f \colon X_1 \to X_2$ is an \textit{isomorphism} if, in addition, $f$ is a bijection (in this case, of course, only (COP1) is enough). It is necessarily a homeomorphism under the interval topologies. 
   Every finite c-ordered set with $n$ elements is isomorphic  to $C_n=\{0, 1, \cdots, n-1\}$ (mod $n$). 
   
 	Denote by $H_+(X)$ the group of all COP (similarly, for LOP) isomorphisms 
 	$X \to X$  
 which is a subgroup of the symmetric group $S(X)$ of all bijections $X \to X$ (in fact, homeomorphisms where $X$ carries the interval topology). 
In some cases we allow also the notation like $\Aut(X,\circ), \Aut(X,\leq)$. 

 For every circularly ordered set $(X,\circ)$ and every subgroup  $G \subset H_+(X)$, the corresponding action $G \times X \to X$ defines a circularly ordered $G$-set $X$.  
If, in addition, $(X,\lambda_{\circ})$ is compact, then $H_+(X)$ is a topological subgroup of $H(X)$ (of all homeomorphisms), usually equipped with the compact-open topology. Then the natural action $H_+(X) \times X \to X$ is continuous. 

A \textit{topological group}, as usual, will mean that the multiplication and the inversion are continuous. Below we consider only Hausdorff topological groups and spaces. 
\textit{Semitopological group} means that the multiplication is separately continuous.     
For every Hausdorff topological space $X$ the homeomorphism group $H(X)$ is a semitopological group  with respect to the pointwise topology $\s^p$ (inherited from the inclusion  $H(X) \subset (X,\t)^X$). 

 By a $G$-\textit{space}, we mean a topological space $X$ endowed with a continuous action 
 $\pi \colon G \times X \to X$ of a topological group $G$ on $X$. Often we write $gx$ or $g(x)$ instead of $\pi(g,x)$. 
If $G$ is a discrete group then the action is continuous if and only if every $g$-translation $\rho_g \colon X \to X, \rho_g(x)=gx$ is continuous; if and only if $\pi \colon G_{\mathrm{disc}} \times X \to X$ is continuous, where $G_{\mathrm{disc}}$ is the discrete copy of $G$.  
For every $x \in X$ we have the corresponding orbit map  $\zeta_{x} \colon G \to X$. 
A continuous map $f \colon X_1 \to X_2$ between two $G$-spaces is called a $G$-\textit{map} (or \textit{equivariant}) if $f(gx) = g f(x)$ for all $g \in G$ and $x \in X_1$.  

  \begin{lem} \label{l:DS-inclusion} 
	Every linearly ordered \textbf{compact} $G$-system is a circularly ordered $G$-system.  	
\end{lem}
\begin{proof}
	Apply Proposition \ref{inclusion} (which ensures the coincidence of two topologies $\lambda_{\medcirc} = \lambda_{\leq}$) and note that if a  $g$-translation $X \to X$ is $\leq$-linear order preserving then it is also circular $R_{\leq}$-order preserving. 
\end{proof}

\begin{defin} \label{d:c-ordG-sp}  
We say that a continuous action $G \times X \to X$ of a topological group $G$ on a circularly ordered (compact) space $(X,\lambda_{\medcirc})$ is a \textit{circularly ordered (dynamical) $G$-system} if all $g$-translations $\rho \colon X \to X$ ($g \in G$) are COP. The class of  \textit{linearly ordered} dynamical $G$-systems is defined similarly.  	
\end{defin}
 
A prototypical example of a circularly ordered system is the classical case of the  circle $\T$ equipped with the action of the group $H_+(\T)$ (or, some of its subgroup $G$) on $\T$. Sturmian $\Z$-systems are important examples of c-ordered symbolic systems. Geometrically Sturmian systems can be obtained from $\T$ by a splitting of a certain dense subset of $\T$ (see Remark \ref{r:add}). 
There exists a very rich literature about actions on $\T$. 
See for example, \cite{Ghys} and references therein.  
 Systematic investigation of \textit{abstract} c-ordered $G$-spaces was initiated in 
\cite{GM-c,GM-UltraHom21,GM-TC}.

In analogy with the notion of a \emph{partially ordered
	topological space} (POTS) from Definition~\ref{d:ord},
one may define a \emph{partially circularly ordered
	topological space} (PCOTS) as a topological space
$(X,\tau)$ endowed with a partial circular order $R$
whose graph is closed in the subspace $\widetilde X^3$
of pairwise distinct triples.

\section{Convexity and generalized Helly space} 

\subsection*{Convex subsets in circular orders} 
\label{s:convex} 

\begin{defin} \label{d:c-convex} 
	Let $(X,R)$ be a circularly ordered set. Let us say that a subset $Y \subseteq X$ is \textit{convex} in $X$ if for every $a,b \in Y$ at least one of the intervals $[a,b]$, $[b,a]$ is a subset of $Y$. 
\end{defin}

\begin{remark} \label{r:convex} \ 
	\begin{enumerate} 
		\item The sets $\varnothing$, $X$, $\{u\}$ and
		$X\setminus\{u\}$ are convex for every $u\in X$.
		For distinct $u,v\in X$, the intervals
		$(u,v)_R$, $[u,v]_R$, $(u,v]_R$ and $[u,v)_R$
		are convex as well.
	
		\item Now the second condition (COP2) in Definition \ref{d:c-ordMaps} can be reformulated as follows: the preimage $f^{-1}(c)$ of a singleton $\{c\}$ is a  convex subset in $X$. 
	\end{enumerate} 
\end{remark}

\begin{lem} \label{l:ConvexProprties} 
	Let $(X,R)$ be a circularly ordered set. 
	\begin{enumerate}
		\item The complement of any convex subset is convex (in contrast to linear orders). 
		\item 
		Let $C_1, C_2$ be convex subsets in $X$ such that $C_1 \cap C_2 \neq \emptyset$. Then the union $C_1 \cup C_2$ is also convex (in contrast to the intersection). 
		\item 
			Let $(X,R)$ be a circularly ordered set and let $A,B,C$ be a triple of nonempty disjoint convex subsets in $X$. Then either $[A,B,C]$ or $[A,C,B]$. 
	\end{enumerate} 	
\end{lem} 
	\begin{proof} (1) 
		Let $C \subseteq X$ be convex and take distinct
		$a,b \in X \setminus C$.
		Suppose that neither $[a,b]$ nor $[b,a]$ is contained
		in $X \setminus C$. Then there exist
		\[
		u \in C \cap (a,b), \qquad v \in C \cap (b,a).
		\]
		Thus $[a,u,b,v]$ is a cycle. In particular,
		$b \in (u,v)$ and $a \in (v,u)$.
		Hence neither $[u,v]$ nor $[v,u]$ is contained in $C$,
		contrary to the convexity of $C$.
		Therefore at least one of $[a,b]$, $[b,a]$ is contained
		in $X \setminus C$, proving that $X \setminus C$ is convex.
		
		(2) If $C_1\cup C_2=X$, the assertion is immediate.
		Otherwise choose $t\notin C_1\cup C_2$ and cut the circular
		order at $t$. Each $C_i$ is convex in the resulting linear
		order: if $u<_t v$ belong to $C_i$, then $t\in[v,u]$,
		so circular convexity forces $[u,v]\subseteq C_i$.
		Since $C_1\cap C_2\neq\emptyset$, their union is linearly
		convex, hence circularly convex.
		
		(3) Observe that a convex set $D$ avoiding two distinct
		points $u,v$ lies entirely in one of $(u,v)$ and $(v,u)$.
		Indeed, if $d\in D\cap(u,v)$ and $e\in D\cap(v,u)$, then
		$v\in(d,e)$ and $u\in(e,d)$, contradicting convexity.
		
		Choose $a_0\in A$, $b_0\in B$, and $c_0\in C$, and suppose
		that $[a_0,b_0,c_0]$. Applying the observation successively
		to $A$, $B$, and $C$ gives
		\[
		[a,b_0,c_0],\qquad [a,b,c_0],\qquad [a,b,c]
		\]
		for arbitrary $a\in A$, $b\in B$, and $c\in C$.
		Thus $[A,B,C]$. The opposite initial orientation gives
		$[A,C,B]$. 
	\end{proof}

\begin{prop} \label{p:Prop-c-ord}   
	Assume that $X,Y$ are circularly ordered sets. 
	\begin{enumerate}
		\item Let $p \colon X \to Y$ satisfy condition (COP1) in Definition \ref{d:c-ordMaps} and $p(X)$ is not a two-point set. Then $p$ is COP. 
	
		\item Let $p \colon X \to Y$ be a COP map. Then 
		\begin{enumerate} 
			\item If $p(a) \neq p(b)$ then $p[a,b] = [p(a),p(b)] \cap p(X)$. The image $p(C)$ is convex in $p(X)$ for every convex $C \subseteq X$.    
			\item The preimage $p^{-1}(I)$ is convex for every convex subset $I$ in $Y$.  
			\item For every $z \in X$  the restricted map 
			$$
			(X \setminus p^{-1}(p(z)),\leq_z) \to (Y,\leq_{p(z)})
			$$
			is linear order preserving. 	
		\end{enumerate}
	\end{enumerate} 	
\end{prop}
\begin{proof}    
	(1) We have to check that (COP2) is satisfied. Assuming the contrary, we have an injective cycle $[x_1,x_2,x_3,x_4]$ such that  $p(x_1)=p(x_3)=c \in p(X)$, $p(x_2) \neq c, p(x_4)  \neq c$. 
		Now observe that $p(x_2)=p(x_4)$. Indeed, if $p(x_2) \neq p(x_4)$, the elements $p(x_2), c, p(x_4)$ are pairwise distinct. Since $[x_2, x_3, x_4]$ holds in $X$, condition (COP1) implies $[p(x_2), c, p(x_4)]$ in $Y$. Similarly, $[x_4, x_1, x_2]$ implies $[p(x_4), c, p(x_2)]$. This contradicts the asymmetry axiom.
	  	
	  	Therefore, we have $t:=p(x_2)=p(x_4)$. We claim that $p(X)\subseteq \{c,t\}$.
	  	Indeed, let $u\in [x_1,x_3]$ and $v\in [x_3,x_1]$ be such that
	  	$p(u)\neq c$ and $p(v)\neq c$. We show that $p(u)=p(v)$.
	  	
	  	Assume the contrary. Then $p(u),c,p(v)$ are pairwise distinct. Since
	  	$[u,x_3,v]$ holds, condition {\rm(COP1)} implies $[p(u),c,p(v)]$.
	  	Also, since $[v,x_1,u]$ holds, condition {\rm(COP1)} implies
	  	$[p(v),c,p(u)]$, contradicting asymmetry. Hence $p(u)=p(v)$.
	  	
	  	Now if $u\in[x_1,x_3]$ and $p(u)\neq c$, take
	  	$v=x_4$ to obtain $p(u)=t$; similarly, if
	  	$v\in[x_3,x_1]$ and $p(v)\neq c$, take $u=x_2$.
	  	Since $X=[x_1,x_3]\cup[x_3,x_1]$, every value of
	  	$p$ different from $c$ is equal to $t$. Hence
	  	$p(X)\subseteq\{c,t\}$, and therefore
	  	$p(X)=\{c,t\}$, contradicting $|p(X)|\ge3$.
	
	(2a) Assume that $p(a)\neq p(b)$. We first show that 	
	\[
	p([a,b])=[p(a),p(b)]\cap p(X).
	\]
	
	If $x\in (a,b)$ and $p(x)\notin \{p(a),p(b)\}$, then $p(a),p(x),p(b)$ are pairwise distinct.
	Since $[a,x,b]$ holds, (COP1) implies $[p(a),p(x),p(b)]$. Hence $p(x)\in [p(a),p(b)]$.
	Thus $p([a,b])\subseteq [p(a),p(b)]\cap p(X)$.

	Conversely, let $y\in [p(a),p(b)]\cap p(X)$. Choose $c\in X$ such that $p(c)=y$.
	If $y=p(a)$ or $y=p(b)$, then clearly $y\in p([a,b])$. So assume that
	$y\notin \{p(a),p(b)\}$. Then $[p(a),p(c),p(b)]$, and the points
	$p(a),p(c),p(b)$ are pairwise distinct. Since $X$ is totally c-ordered, the equivalent
	form of {\rm(COP1)} yields $[a,c,b]$. Hence $c\in [a,b]$, and therefore
	$y=p(c)\in p([a,b])$. This proves  $p([a,b])=[p(a),p(b)]\cap p(X)$.

	Now let $C\subseteq X$ be convex. We show that $p(C)$ is convex in $p(X)$.
	Take $y_1,y_2\in p(C)$, and choose $x_1,x_2\in C$ such that $p(x_1)=y_1$ and $p(x_2)=y_2$.
	If $y_1=y_2$, there is nothing to prove. So assume that $y_1\neq y_2$. 
	Since $C$ is convex, at least one of the closed intervals $[x_1,x_2]$ or $[x_2,x_1]$
	is contained in $C$. If $[x_1,x_2]\subseteq C$, then by the equality already proved,
	$p([x_1,x_2])=[y_1,y_2]\cap p(X)\subseteq p(C)$. If $[x_2,x_1]\subseteq C$, then similarly
	$p([x_2,x_1])=[y_2,y_1]\cap p(X)\subseteq p(C)$. 
	Therefore $p(C)$ is convex in $p(X)$.

(2b) We have to check that $p^{-1}(I)$ is convex for every convex subset $I$ of $Y$. 
Let \(a,b\in f^{-1}(I)\). 
If \(p(a)=p(b)\), use (COP2). 

Assume that \(p(a)\ne p(b)\). 
By the definition of circular convexity, at least one of the following holds:
\[
[p(a),p(b)]_Y\subseteq I
\qquad\text{or}\qquad
[p(b),p(a)]_Y\subseteq I.
\] 
Since \(p(a)\ne p(b)\), Proposition \ref{p:Prop-c-ord}(2a)  gives
$
p([a,b]_X)
=
[p(a),p(b)]_Y\cap p(X).
$
In particular,
\[
p([a,b]_X)\subseteq[p(a),p(b)]_Y.
\]
Interchanging \(a\) and \(b\), we also obtain
\[
p([b,a]_X)\subseteq[p(b),p(a)]_Y.
\] 
First case: suppose \([p(a),p(b)]_Y\subseteq I\). Then
$
p([a,b]_X)\subseteq[p(a),p(b)]_Y\subseteq I.
$
Therefore every \(x\in[a,b]_X\) satisfies \(p(x)\in I\), which means
$
[a,b]_X\subseteq p^{-1}(I).
$

Second case: suppose \([p(b),p(a)]_Y\subseteq I\). The same argument with the endpoints reversed gives
$
p([b,a]_X)\subseteq[p(b),p(a)]_Y\subseteq I,
$
and hence
$
[b,a]_X\subseteq p^{-1}(I).
$
Thus, in either case,
\[
	[a,b]_X\subseteq p^{-1}(I)
	\quad\text{or}\quad
	[b,a]_X\subseteq p^{-1}(I).
\]

	(2c) Let $a \leq_z b$ where $a,b \in X \setminus p^{-1}(p(z))$ (so, $p(a) \neq p(z), p(b) \neq p(z)$). 
	In particular, $a \neq z, b \neq z$. 
	We have to show that 
	$p(a) \leq_{p(z)} p(b)$.  Without loss of generality, we can suppose that 
	$p(a) \neq p(b)$.  So, $a \leq_z b$ implies $zab$.  
	Since $p(a), p(z), p(b)$ are pairwise distinct, (COP1)  gives $p(z)p(a)p(b)$, i.e. $p(a)\le_{p(z)} p(b)$.
\end{proof}

\subsection*{Generalized Helly space} \label{s:GenHellySp} 

\begin{thm} \label{t:M_Closed}   
	$M_+(X,Y)$ is pointwise closed in $Y^{X}$ for every pair of c-ordered sets $X,Y$. 
\end{thm}
\begin{proof}
	Let $\{p_i: \ i \in I\}$ be a net in $M_+(X,Y)$ which pointwise converges to some function $p \colon X \to Y$. 
	We have to show that $p$ is also c-order preserving. Assume it is not. Clearly, if $p$ is a  constant map then $p$ is COP. So, we have two cases:  
	
	\medskip 
	(1) $|p(X)| \geq 3$. 
	Then by Proposition \ref{p:Prop-c-ord}.1 there exists a triple $x_1,x_2,x_3$ (of distinct points) such that $x_1x_2x_3$ but $p(x_1)p(x_3)p(x_2)$. Since the c-order is open (Lemma \ref{l:c-is-Open}), choose open pairwise disjoint neighborhoods $O_1,O_2,O_3$ of $p(x_1),p(x_2),p(x_3)$, respectively,  such that $[y_1',y_3',y_2']$ for every $y'_i \in O_i$. This leads to a contradiction because  $p_{i_0}(x_1)p_{i_0}(x_2)p_{i_0}(x_3)$ and at the same time $p_{i_0}(x_1)p_{i_0}(x_3)p_{i_0}(x_2)$ for some $i_0 \in I$.   
	
	(2) Assume that $p(X)=\{u,v\}$, where $u\neq v$. Since $p$ is not COP, at least one of the
	fibers $p^{-1}(\{u\})$, $p^{-1}(\{v\})$ is not convex. Without loss of generality, $A:=p^{-1}(\{u\})$
	is not convex. Then there exist $x_1,x_3\in A$ such that neither $[x_1,x_3]\subseteq A$ nor
	$[x_3,x_1]\subseteq A$. Choose $x_2\in [x_1,x_3]\setminus A$ and $x_4\in [x_3,x_1]\setminus A$.
	Then $p(x_1)=p(x_3)=u$, $p(x_2)=p(x_4)=v$, and $[x_1,x_2,x_3,x_4]$ is an injective $4$-cycle in $X$.
	
	If $|Y|\ge 3$ we may choose
	disjoint open intervals $u\in U$ and $v\in V$ in $Y$. Since $p_i\to p$ pointwise, there exists
	$i_0$ such that
	\[
	p_{i_0}(x_1),\, p_{i_0}(x_3)\in U
	\quad\text{and}\quad
	p_{i_0}(x_2),\, p_{i_0}(x_4)\in V .
	\]
	Hence $x_1,x_3\in p_{i_0}^{-1}(U)$, while $x_2\in [x_1,x_3]\setminus p_{i_0}^{-1}(U)$ and
	$x_4\in [x_3,x_1]\setminus p_{i_0}^{-1}(U)$. So $p_{i_0}^{-1}(U)$ is not convex. Similarly,
	$p_{i_0}^{-1}(V)$ is not convex. This contradicts Proposition~\ref{p:Prop-c-ord}.(2b).
	
	If $|Y|=2$, then $U:=\{u\}=Y\setminus [v,v]$ and $V:=\{v\}=Y\setminus [u,u]$ are open in $Y$.
	The same argument shows that $p_{i_0}^{-1}(U)$ and $p_{i_0}^{-1}(V)$ are not convex, again
	contradicting Proposition~\ref{p:Prop-c-ord}.(2b). 
\end{proof} 

\begin{prop} \label{p:compCirc1} 
	Let $K$ be a c-ordered compact space. Then $H_+(K)$ is a closed subgroup of the homeomorphism group $H(K)$. 
	Therefore, $H_+(K)$ is complete  (in its two sided uniformity). 
\end{prop}
\begin{proof} It is well known that the topological group $H(K)$ is complete for every compact space $K$. Therefore, it is enough to see that the  subgroup $H_+(K)$ is closed. 
	Let $g_i$ be a net in $H_+(K)$ which converges in $H(K)$ to some $g  \in H(K)$ with respect to the compact open topology. Then this net converges also pointwise. Hence, $g \in M_+(K,K)$ by Theorem \ref{t:M_Closed}. Since $\rho_g$ is a homeomorphism of $K$ we obtain that $g \in H_+(K)$. 
\end{proof}

\begin{defin} \label{d:sing}  Let $(X,\leq)$ be a linearly ordered set. Let us say that $u \in X$ is a \textit{right-singular point} (resp. \textit{left-singular}) and write $u \in \rm{sing}^{-}(X)$ (resp. $u \in \rm{sing}^{+}(X)$) if $[u,\to)$ (resp. $(\leftarrow,u]$) is a clopen subset in $(X,\lambda_{\leq}$). Define also $\rm{sing}(X):=\rm{sing}^{-}(X) \cup  \rm{sing}^{+}(X)$  the set of all \textit{singular points}. 
\end{defin}  

\begin{lem} \label{l:sing}  Let $(X,\lambda_{\leq})$ be a linearly ordered compact metric space. 
	Then $\rm{sing}(X,\leq)$ is at most countable. 
\end{lem}
\begin{proof}
	The maps
	$u\mapsto [u,\to)$ on $\rm{sing}^-(X)$ and
	$u\mapsto(\leftarrow,u]$ on $\rm{sing}^+(X)$
	are injective into $\rm{clop}(X)$.
	Now $\rm{clop}(X)$ is countable for every compact metric
	space: if $\mathcal B$ is a countable base, every
	clopen set is a finite union of members of $\mathcal B$.
	Hence both $\rm{sing}^-(X)$ and $\rm{sing}^+(X)$ are
	countable, and so is $\rm{sing}(X)$.
\end{proof}

It is a classical fact in analysis that every linear order preserving bounded function $f \colon X \to \R$ on an interval $X \subseteq \R$ has one-sided limits. 
This can be extended to the case of linear order preserving maps 
$f \colon X \to Y$ such that $X$ is first countable and $Y$ is sequentially compact; see \cite{FS}. 

\begin{lem} \label{l:propSIDE} 
	Let $(X,\lambda_{\leq_X})$ and $(Y,\lambda_{\leq_Y})$ be $\LOTS$, where $Y$ is compact and metrizable. 
	Assume that $p \colon (X,\leq_X) \to (Y,\leq_Y)$ is a LOP map. Then 
	\begin{enumerate}
		\item $p$ has one-sided limits $L(a), R(a)$ at any $a \in X$ and $L(a)  \leq p(a) \leq R(a) $. 
		\item If $a \notin \rm{sing}(X)$, then $p$ is continuous at $a$ if and only if $L(a) = p(a) = R(a)$.  
		\item The set of discontinuity points of $p$ outside $\rm{sing}(X)$ is at most countable. 
		In particular, if $\rm{sing}(X)$ is countable (e.g. if $X$ is a compact metric space, by Lemma \ref{l:sing}),
		then $p$ has at most countably many discontinuity points.
	\end{enumerate} 
\end{lem} 
\begin{proof} 
	(1) Since $Y$ is a compact LOTS, it is order-complete. Thus, the one-sided limits 
	$L(a):=\sup \{p(x): x<a\}$ and $R(a):=\inf \{p(x): a<x\}$ 
	exist in $Y$ for every $a \in X$ (with the convention $L(\min X) = \min Y$ and $R(\max X) = \max Y$). Since $p$ is LOP, $x < a$ implies $p(x) \leq p(a)$, so taking the supremum yields $L(a) \leq p(a)$. Similarly, $p(a) \leq R(a)$, giving $L(a) \leq p(a) \leq R(a)$.
	
	Assume now that $a < b$ and $a,b \notin \rm{sing}(X)$. The interval $(a,b)$ is nonempty (otherwise either $[b,\to)$ or $(\leftarrow,a]$ would be clopen, making them singular). For any $x \in (a,b)$, the LOP property implies $R(a) \leq p(x) \leq L(b)$. Hence, $R(a) \leq L(b)$, which yields:
	\begin{equation} \label{eq:empty} 
		a < b,\ a,b \notin \rm{sing}(X) \Longrightarrow (L(a), R(a)) \cap (L(b), R(b)) = \emptyset.
	\end{equation} 
	
	(2) Let $a \notin \rm{sing}(X)$. 
	If $p$ is continuous at $a$, the preimage of any neighborhood of $p(a)$ must intersect both $(\leftarrow,a)$ and $(a,\to)$ (since $a$ is not singular). This forces $\sup_{x < a} p(x) = p(a)$ and $\inf_{x > a} p(x) = p(a)$. Thus, $L(a) = p(a) = R(a)$.

Conversely, assume that $L(a)=p(a)=R(a)$. Let $O$ be any neighborhood of $p(a)$ in $Y$.
Then $O$ contains a basic convex open neighborhood $V$ of $p(a)$ of one of the forms
$(u,v)$, $(\leftarrow,v)$, or $(u,\to)$.

If $V=(u,v)$, choose $x_1<a<x_2$ such that $p(x_1)>u$ and $p(x_2)<v$.
Then for every $x\in (x_1,x_2)$ we have
$u<p(x_1)\le p(x)\le p(x_2)<v$, hence $p(x)\in V\subseteq O$.

If $V=(\leftarrow,v)$, choose $x_2>a$ such that $p(x_2)<v$.
Then for every $x\in [a,x_2)$ we have $p(x)\le p(x_2)<v$, hence $p(x)\in V\subseteq O$.

If $V=(u,\to)$, choose $x_1<a$ such that $p(x_1)>u$.
Then for every $x\in (x_1,a]$ we have $u<p(x_1)\le p(x)$, hence $p(x)\in V\subseteq O$. 
Thus $p$ is continuous at $a$.  
	
	(3) Let $D_0 := \{a \in X \setminus \rm{sing}(X) : p \text{ is discontinuous at } a\}$. By (2), for every $a \in D_0$, we must have $L(a) < R(a)$. We will map $D_0$ injectively into a countable set. 
	
	Since $Y$ is a linearly ordered compact metric space, it has a countable dense subset $C$, and by Lemma \ref{l:sing}, $\rm{sing}(Y)$ is countable. Define $\phi: D_0 \to C \cup \rm{sing}(Y)$ as follows:
	\begin{itemize}
		\item If $(L(a),R(a)) \neq \emptyset$, choose $\phi(a) \in C \cap (L(a),R(a))$.
		\item If $(L(a),R(a)) = \emptyset$, set $\phi(a) := R(a)$. Since $a \in D_0$, $L(a) < R(a)$. Thus $R(a)$ is the immediate successor of $L(a)$, making the ray $[R(a), \to)$ open, so $R(a) \in \rm{sing}(Y)$.
	\end{itemize}
	Let $a < b$ in $D_0$. From \eqref{eq:empty} we have $R(a) \leq L(b)$. By the definition of $\phi$, we have $\phi(a) \leq R(a)$ and $L(b) < \phi(b)$ in both subcases. Therefore, $\phi(a) \leq R(a) \leq L(b) < \phi(b)$, which implies $\phi(a) < \phi(b)$. 
	Thus, $\phi$ is strictly increasing, hence injective, making $D_0$ countable.
	
	Finally, the set of all discontinuity points of $p$ is contained in $D_0 \cup \rm{sing}(X)$, which is at most countable if $\rm{sing}(X)$ is countable.
\end{proof}

The following useful result is a generalization of \cite[Proposition 2]{Bour}. 

\begin{f} \label{countable} \cite{GM-TC} 
	Let $X$ be a set, $(Y,d)$  a metric space, 
	and $E \subseteq Y^X$ a compact subspace in the pointwise topology. 
	The following conditions are equivalent:
	\begin{enumerate}
		\item a point $p \in E$ admits a countable local basis in $E$;  
		\item there is a countable set $C \subset X$ which determines $p$ in $E$ (meaning that for 
		any given $q \in E$, the condition $q(c) =p(c)$ for all $c \in C$ implies that $q(x)=p(x)$ for every $x \in X$).
	\end{enumerate} 
\end{f}

Recall that the Helly space $M_+([0,1],[0,1])$ is first countable (see, for example, \cite[page 127]{SS}).  
The following theorem is inspired by this classical fact. 

\begin{thm}[Generalized Helly space] \label{t:lin-Helly}
	Let $X$ and $Y$ be linearly ordered sets with their interval topologies such that 
	$X$ and $Y$ are compact metric spaces. Then the set $M_+(X,Y)$ in the pointwise topology of all LOP maps is first countable (and compact by Theorem \ref{t:M_Closed}).    
\end{thm}
\begin{proof}
	Let $p\in M_+(X,Y)$. By Lemmas \ref{l:propSIDE}.3 and \ref{l:sing}, the set $\rm{disc}(p)$ is countable. Also,
	$\rm{sing}(X)$ is countable by Lemma \ref{l:sing}. Since $X$ is compact and metrizable, there exists a
	countable dense subset $A\subseteq X$. Without loss of generality, we may assume that $A$
	contains the endpoints of $X$, if they exist. Put
	$C:=\rm{disc}(p)\cup \rm{sing}(X)\cup A$.
	
	It is enough to show that $C$ satisfies the condition of Fact \ref{countable} for
	$E:=M_+(X,Y)$. Let $q\in M_+(X,Y)$ be such that $q(c)=p(c)$ for every $c\in C$.
	We prove that $q(x)=p(x)$ for every $x\in X$.
	
	Assume the contrary, and choose $x_0\in X$ such that $q(x_0)\neq p(x_0)$. Then
	$x_0\notin C$. Hence $x_0\notin \rm{disc}(p)$, so $p$ is continuous at $x_0$; also
	$x_0\notin \rm{sing}(X)$, and $x_0$ is not an endpoint of $X$.
	
	Choose an open convex neighborhood $U$ of $p(x_0)$ in $Y$ such that $q(x_0)\notin U$.
	Since $p$ is continuous at $x_0$ and $x_0$ is not an endpoint, there exist
	$x_1,x_2\in X$ such that $x_0\in (x_1,x_2)$ and $p((x_1,x_2))\subseteq U$.
	Since $x_0\notin \rm{sing}(X)$, the intervals $(x_1,x_0)$ and $(x_0,x_2)$ are nonempty.
	As $A$ is dense in $X$, there exist $a_1\in A\cap (x_1,x_0)$ and $a_2\in A\cap (x_0,x_2)$.
	Then $a_1<x_0<a_2$ and $p(a_1),p(a_2)\in U$.
	
	Since $U$ is convex, we have $[p(a_1),p(a_2)]\subseteq U$. On the other hand, $q$ is LOP
	and $q(a_i)=p(a_i)$ for $i=1,2$, so
	$q(a_1)\le q(x_0)\le q(a_2)$. Therefore
	$q(x_0)\in [p(a_1),p(a_2)]\subseteq U$, contradicting the choice of $U$. 
	This contradiction shows that $q=p$. Hence $C$ determines $p$ in $M_+(X,Y)$.
	By Fact \ref{countable}, the point $p$ admits a countable local basis in $M_+(X,Y)$.
	Since $p$ was arbitrary, $M_+(X,Y)$ is first countable. Compactness follows from
	Theorem \ref{t:M_Closed}.  
\end{proof}

\begin{remark} \label{r:c-Helly-Note}
For \textbf{circular} order preserving maps, the last step of the previous proof is not always true. Namely, by Definition \ref{d:c-ordMaps}, it is possible that $q(a_1)=q(a_2)=q(x)$ for every $x \in [a_2,a_1]_{\circ}$ but $q(x_0) \notin [q(a_1),q(a_2)]_{\circ}$.
\end{remark}

The circular version of Theorem \ref{t:lin-Helly} is not true.

\begin{ex} \label{ex:BigHelly}
	The space $M_+(\T,\T)$ is not first countable. 	
\end{ex}	
\begin{proof} 
	Let $p_a \colon \T \to \T$ be the constant map $p_a(x)=a$. For every $b \neq a$ in $\T$ define the map 
	$$p_a^{b} \colon \T \to \T, \ \ p_a^{b}(x):=a \ \forall x \neq b, \ p_a^{b}(b):=b.$$
	Then $p_a^{b}, p_a \in M_+(\T,\T)$. Note that $p_a^b$ is COP because the fibers are $\T \setminus\{b\}$ and $\{b\}$, both of which are convex in a circular order, and (COP1) is vacuous when images are not all distinct.
	Fact \ref{countable}  shows that $p_a$ does not admit a countable local basis; cf. Proposition \ref{ex:T}.    
	Indeed, for every countable $C\subset\T$ choose
	$b\in\T\setminus(C\cup\{a\})$; then
	$p_b^a|_C=p_a|_C$ while $p_b^a\neq p_a$.  
\end{proof}
  
 \section{Completeness, compactifications and inverse limits} 
 \label{s:compINV} 
 
 A {\it gap} on $(X,\circ)$ is a cut $\leq$ on $X$ that has neither a least nor a greatest element. This continues the idea of classical Dedekind cuts for linear orders. 
 A c-ordered set 
 $(X,\circ)$ is \textit{complete} in the sense of Nov\'{a}k \cite{Novak-cuts} if it has no gaps. 
 Every circularly ordered set admits a \textit{completion} \cite{Novak-cuts}. 
 
 \subsection*{Compactness of the circular topology} 
 
 The following result gives a natural circular version of the compactness criterion for linear orders (compare Fact \ref{f:propLOTS}.2).    
 
 \begin{thm}[complete=compact] \label{t:c-comp} \cite{Me-OrdSem}  
 	Let $\lambda_{\circ}$ be the circular topology of a circular order on a set $X$. The following conditions are equivalent:
 	\begin{enumerate} 
 		\item $\lambda_{\circ}$ is a compact topology. 
 		\item $(X,\circ)$ is complete.
 		\item $[a,b]_{\circ}$ is compact in the subspace topology of $(X,\lambda_{\medcirc})$ for every $a,b \in X$.
 	\end{enumerate} 
 \end{thm}
 
 \begin{proof}
 	(1) $\Rightarrow$ (2) 
 	Let $(X,\leq)$ be a gap. Then 
 	$\cup \{(a,b)_{\leq}: a < b\}=X$ is a cover. This cover is open (since $(a,b)_{\leq}=(a,b)_{\circ}$ by Lemma \ref{l:cut}.2). Because the cut $\le$ has no extrema, no finite subcollection can cover $X$.
 	A finite union of ``bounded open" intervals in a linear order without endpoints cannot cover the whole set.
 	
 	\medskip 
 	(2) $\Rightarrow$ (3) 
 	Assume that $[a,b]_{\circ}$ is not compact in its subspace topology. Then the linearly ordered space $([a,b]_{\leq_a},\leq_a)$ 
 	has a subset $Y$ such that $\sup(Y)$ does not exist (use Fact \ref{f:propLOTS}). Consequently, $([a,b]_{\leq_a}, \leq_a)$ has a ``linear gap", i.e., a subset $Y \subset [a,b]_{\leq_a}$ without supremum. Then $Y$ does not have a supremum as a subset of $(X,\leq_a)$ either. Indeed, 
 	if a supremum existed in $X$, it would still have to lie in $[a,b]_{\leq_a}$, because $b$ is already an upper bound of $Y$.   
 	
 	Let $X_{1}:=\{x \in X: \forall y \in Y \ \ y <_a x\}$ and $X_2:=X \setminus X_1$. 
 	Define now a linear order $\leq_Y$ on $X$ by the following rule: 
 	$x_1 <_Y x_2$ for every $x_1 \in X_1, x_2 \in X_2$, while keeping the old relation $\leq_a$ for pairs within $X_1$ and within $X_2$.  
 	Then $\le_Y$ is a cut on $X$. Indeed, first note that $x_2 <_a x_1$ for all $x_2 \in X_2$ and $x_1 \in X_1$. 
 	For $x_2\in X_2$ choose $y\in Y$ with $x_2\le_a y$, while $y<_a x_1$ for every $x_1\in X_1$.
 	
 	Let $u<_Y v<_Y w$. If $\{u,v,w\}\subseteq X_i$ for some $i\in\{1,2\}$, then $u<_a v<_a w$, hence $[u,v,w]$ by Lemma \ref{l:cut}.2.
 	Otherwise $u\in X_1$ and $w\in X_2$.
 	If $v\in X_1$, then $w<_a u<_a v$, so $[w,u,v]$ (since $\le_a$ is a cut), and by cyclicity $[u,v,w]$.
 	If $v\in X_2$, then $v<_a w<_a u$, so $[v,w,u]$, and again by cyclicity $[u,v,w]$.
 	Thus $\le_Y$ is a cut. 
 	Since $Y$ has no supremum in $(X, \leq_a)$, the set of upper bounds $X_1$ has no minimum, and the set $X_2$ has no maximum. Because $(X, \leq_Y) = X_1 \oplus X_2$, the linear order $\leq_Y$ has neither a minimum nor a maximum, so it is a gap.
 	
 	\medskip 
 	(3) $\Rightarrow$ (1) 
 	For every $a \neq b$ we have 
 	$X=[a,b]_{\circ} \cup [b,a]_{\circ}$. Hence, $X$ is compact being a union of two compact subsets. 
 \end{proof}
 
 \subsection*{Lexicographic order} 
 \label{r:prod} 
 
 For every c-ordered set $C$ and a linearly ordered set $L$, one may define the so-called 
 \emph{c-ordered lexicographic product} $C \otimes_c L$. See, for example, \cite{CJ} and 
 also Figure 1. 
 \begin{figure}[h]  
 	\begin{center}   
 		\scalebox{0.15}{\includegraphics{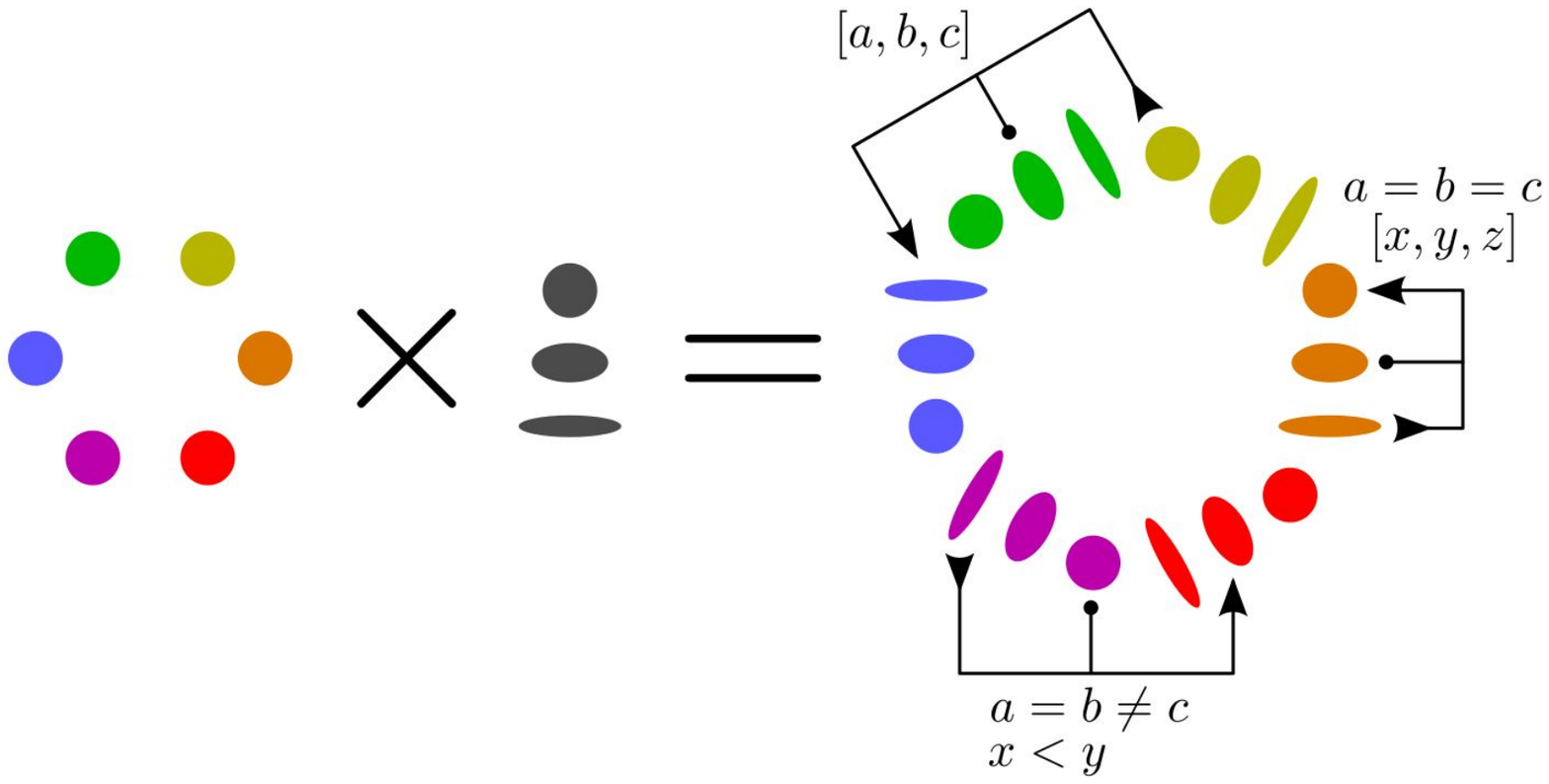}}
 		\caption{c-ordered lexicographic product (from Wikipedia - \textit{Cyclic order})}
 	\end{center} 
 	\label{fig:first}
 \end{figure} 
 
 \begin{defin} \label{d:lexic}  
 	More formally, 
 	$[(a,x), (b,y), (c,z)]$ (for distinct points) in $C \otimes_c L$ will mean that one of the following conditions is satisfied:
 	\begin{enumerate}
 		\item $[a, b, c]$.
 		\item $a = b \neq c$ and $x < y$.
 		\item $b = c \neq a$ and $y < z$.
 		\item $c = a \neq b$ and $z < x$.
 		\item $a = b = c$ and $[x, y, z]$ (in the cyclic order on $L$ induced by the linear order). 
 	\end{enumerate} 
 \end{defin} 
 
 The linear lexicographic product of two linearly ordered sets $L_1, L_2$ is denoted by $L_1 \otimes_l L_2$.  
 The following result is a particular case of \cite[Lemma 2.1]{Kem} proved by N. Kemoto. It can be obtained also directly using Fact \ref{f:propLOTS}.2. 
 
 \begin{f} \label{f:Kem} \cite{Kem} 
 	A lexicographic linearly ordered product $Y \otimes_l L$ of a compact linearly ordered space $Y$ and a compact linearly ordered space $L$ is compact in the interval topology.  
 \end{f}
 
 \begin{prop} \label{p:LexComp}  
 	A lexicographic c-ordered product $K \otimes_c L$ of a compact c-ordered space $K$ and a compact linearly ordered space $L$ is a compact c-ordered space.  
 \end{prop}
 \begin{proof} By Theorem \ref{t:c-comp}, we have to show that every interval  $[(a,u),(b,v)]_{\circ}$ of $K \otimes_c L$ is compact.  
 	
 	 First assume that either $a\neq b$, or $a=b$ and $u\leq v$.  
 	Consider the linearly ordered set $[a,b]_{\leq_a}$. It is compact by Theorem  \ref{t:c-comp} 
 	in its circular or linear topologies. Observe that $[(a,u),(b,v)]_{\circ}=[(a,u),(b,v)]_{\leq}$, where $\leq$ is the linear order inherited from the lexicographic linearly ordered product $Y:=[a,b]_{\leq_a} \otimes_l L$.  
 	The latter is compact by Fact \ref{f:Kem}. Then $[(a,u),(b,v)]_{\circ}$ is also compact (being a closed subset in $Y$). 
 	
 	It remains to consider the case $a=b$ and $v<u$.
 	
 	If $K=\{a\}$, then $K\otimes_c L$ is just the circular order induced by the linear order on $L$,
 	so the interval $[(a,u),(a,v)]_{\circ}$ is compact.
 	
 	Assume now that $|K|>1$. Choose $c\in K\setminus\{a\}$ and any $w\in L$.
 	Since $v<u$, Definition~\ref{d:lexic}.(4) implies that
 	$[(a,u),(c,w),(a,v)]$ in $K\otimes_c L$. Hence $(c,w)\in [(a,u),(a,v)]_{\circ}$.
 	By Remark \ref{r:chech}.3, the interval $[(a,u),(a,v)]_{\circ}$ is linearly ordered by the
 	standard cut at $(a,u)$, and therefore
 	\[
 	[(a,u),(a,v)]_{\circ}
 	=
 	[(a,u),(c,w)]_{\circ}\cup [(c,w),(a,v)]_{\circ}.
 	\]
 	Each of the intervals on the right-hand side has endpoints in different fibers, so by the argument
 	already proved above each of them is compact. Therefore $[(a,u),(a,v)]_{\circ}$ is compact as a
 	union of two compact sets.
 \end{proof}
 
 \subsection*{Uniform structures and uniform covers}
 
 For the standard definitions around uniform structures in terms of \textit{entourages} we refer to \cite{Eng89}. We use also an equivalent known approach in terms of coverings. Recall some auxiliary definitions. 

 For two covers $\alpha,\beta$ of a set $X$, we say that $\alpha$ \emph{refines} $\beta$ if for every $A \in \alpha$ there exists $B \in \beta$ such that $A \subseteq B$. 
 Notation: $\alpha \succ \beta$.   
 The \textit{star} of a subset $C \subset X$ with respect to $\a$ is the set 
 $$
 \St(C,\a)=\cup\{A \in \a: A \cap C \neq \emptyset\}. 
 $$
 For the singleton $C:=\{c\}$ we simply write $\St(c,\a)$. 
 So, $\St(c,\a)=\cup\{A \in \a: c \in A\}$. 
 The collection $\a^*:=\{\St(A,\a): A \in \a\}$ is a covering and is called the \textit{star of $\a$}. Always, 
 $\a \succ \a^*$. 
 
 \begin{defin} \label{d:cov-unif}[Coverings approach] \cite{Isb,Eng89}    
 	Let $\mathcal{U}$ be a family of coverings on a set $X$. 
 	Then $\mathcal{U}$ is said to be a \textit{(covering) uniformity} on $X$ if: 
 	\begin{itemize}
 		\item [(C1)] $\a,\beta \in \mathcal{U}$ implies that 
 		$\a \wedge \beta:=\{A \cap B: A \in \a, B \in \beta\} \in \mathcal{U}$; 
 		\item [(C2)] $\a \in \mathcal{U}$ and $\a \succ \beta$ imply that $\beta \in \mathcal{U}$; 
 		\item [(C3)] every element in $\mathcal{U}$ has a star-refinement in $\mathcal{U}$ (meaning that for every  $\beta \in \mathcal{U}$ there exists $\a \in  \mathcal{U}$ such that $\a^* \succ \beta$).   
 	\end{itemize} 
 \end{defin} 
 
 A subfamily $\mathcal{B} \subset {\mathcal U}$ such that each $\eps \in {\mathcal U}$ has a refinement $\delta \succ \eps$ with
 $\delta \in \mathcal{B}$, is said to be a (uniform) \emph{base} of ${\mathcal U}$. A subfamily $\Gamma \subset {\mathcal U}$ is a \textit{subbase} (or, a \textit{prebase}) if the enriched family $\Gamma^{\wedge}:=\{\g_1 \wedge \cdots \wedge \g_n : \ \g_i \in \Gamma, n \in \N\}$ is a base of ${\mathcal U}$.  
 An abstract set $\mathcal{B}$ of coverings on $X$ is a base of some uniformity $\cU$ if and only if 
$$
 	\forall P_1,P_2 \in \mathcal{B} \ \exists P_3 \in \mathcal{B} \ \ P_3^{*} \succ P_1 \wedge P_2. 
$$

 Each uniformity $\mathcal{U}$ on $X$ defines a topology $top(\mathcal{U})$ on $X$: a subset $A \subseteq X$ is open iff for each $a \in A$ there exists a covering $P \in \cU$ such that $\St(a,P) \subset A$. 
 In terms of entourages: 
 for each $a \in A$ there exists an entourage (binary relation on $X$) $\eps \in \mathcal{U}$ such that $\eps(a):=\{x \in X: (a,x) \in \eps\} \subseteq A$. 
 
 Below we consider mainly only \textit{Hausdorff uniform structures}. This is the case when the uniform covers separate the points. That is, for every $x \neq y$ in $X$ there exists $\alpha \in \mathcal{U}$ such that $y \notin \St(x,\alpha)$. 
 
 Recall that a covering uniform space $(X,\mathcal{U})$ is called \emph{precompact} (or, \textit{totally bounded}) if
 $\mathcal{U}$ has a base consisting of finite covers of $X$. If the uniformity is Hausdorff, then it is equivalent to say that the uniform completion is compact.

 \begin{lem} \label{l:intBase} 
 	Let $(X,\mathcal{U})$ be a uniform (covering) space and $\mathcal{B}$ is its uniform base. For every $\beta \in \mathcal{B}$ consider $\beta^{\rm{int}}:=\{\rm{int}(B): B \in \beta\}$. Then $\mathcal{B}^{\rm{int}}:=\{\beta^{\rm{int}}: \beta \in \mathcal{B}\}$ also is a uniform base. 	
 \end{lem}
 \begin{proof} 
 	Choose $\alpha \in \mathcal{U}$ such that $\alpha^* \succ \beta$. Then $\alpha \succ \beta^{\rm{int}}$.
 	Indeed, if $\St(A,\a) \subseteq B$, then for every $x \in A$ one has $\St(x,\a) \subseteq B$, so $x \in \rm{int}(B)$, hence $A \subseteq \rm{int}(B)$.
 \end{proof}
 
 \subsection*{Compactifications of spaces and group actions}
 
 In the present paper ``compactification" does not necessarily imply a topological embedding.  
 More precisely, a continuous dense map $\nu \colon X \to Y$ into a compact Hausdorff space $Y$ is called a \textit{compactification} of $X$. If $X$ and $Y$ are $G$-spaces and $\nu$ is a $G$-map, then $\nu$ is said to be a $G$-\textit{compactification}. If $\nu$ is also a topological embedding, it is called \textit{proper}. Recall that $\nu$ is proper iff the canonically induced Banach unital subalgebra $\mathcal{B}_{\nu}:=\{f \circ \nu: f \in C(Y)\}$ of $C_b(X)$ separates points and closed subsets of $X$. 
 
 Two compactifications $\nu_1 \colon X \to Y_1, \nu_2 \colon X \to Y_2$ are equivalent if there exists a homeomorphism $h \colon Y_1 \to Y_2$ such that $\nu_2=h \circ \nu_1$. Or, iff these compactifications have the same induced Banach algebras $\mathcal{B}_{\nu_2}=\mathcal{B}_{\nu_1}$. 
 More generally, $\nu_1$ dominates $\nu_2$ means that $\mathcal{B}_{\nu_2} \subseteq \mathcal{B}_{\nu_1}$ iff there exists a continuous (onto) factor $\g \colon Y_1 \to Y_2$ such that $\nu_2=\gamma \circ \nu_1$.  
 If $X$ is a $G$-space and $\nu_1, \nu_2$ are  $G$-compactifications then the map $\g \colon Y_1 \to Y_2$ above necessarily is a $G$-map.  
 
 The \textit{Samuel compactification} of a uniform space $(X,\mathcal{U})$ is the compactification induced by the algebra $\mathrm{Unif}^b(X,\mathcal{U})$ of all $\mathcal{U}$-uniformly continuous bounded real functions on $X$.  
 
 \medskip

 Let $G$ be a semitopological group. 
 By $\mathcal{N}(g)$ we denote set of all neighborhoods at the point $g$ in $G$. So, $V \in \mathcal{N}(e)$ if and only if $g_0 V \in \mathcal{N}(g_0)$.  
 \begin{defin} \label{d:equiun} 
 	Let 
 	$\pi \colon G \times X \to X$ be a group action of a (semi)topological group $G$ and 
 	$\mathcal{U}$ be a compatible uniformity on a topological space $X$ defined by coverings (or by entourages). We call the action:
 	\begin{enumerate}
 		\item
 		{\it $\mathcal{U}$-saturated} if every $g$-translation $ \rho_g \colon X \to X$ 
 		is $\mathcal{U}$-uniformly continuous. 
 		If $\mathcal{U}$ is saturated then the corresponding homomorphism 
 		$h_{\pi}\colon G \to \Aut(X,\mathcal{U}), \ g \mapsto  \rho_g$ is well defined, where $\Aut(X,\mathcal{U})$ is the automorphism group of the uniform space.
 		
 		\item
 		{\it $\mathcal{U}$-bounded} (at $e \in G$) if for every $\varepsilon \in \mathcal{U}$
 		there exists $V \in \mathcal{N}(e)$ such that 
 		
 		(coverings) $  \ \ \ 
 		\{Vx: x\in X\} \ \text{refines} \ \eps   
 		$ 
 		
 		(entourages) \ \ \   $(x,vx) \in \varepsilon$ for each $x\in X$ and $v \in V$. 
 		
 		\item \emph{$\mathcal{U}$-equiuniform}  
 		if and only if it is $\mathcal{U}$-saturated and $\mathcal{U}$-bounded.  
 		It is equivalent to say that the corresponding homomorphism 
 		$h_{\pi}\colon G \to \Aut(X,\mathcal{U})$ is continuous, where $\Aut(X,\mathcal{U})$ carries the topology of uniform convergence.
 	\end{enumerate} 
 	Sometimes we say also that $\mathcal{U}$ is $G$-\textit{equiuniform} or $G$-\textit{bounded}. 
 \end{defin}

 Definition \ref{d:equiun}.2 of boundedness appears in \cite{Br} and  \cite{Vr-Embed77} under the names: \emph{motion equicontinuous} and \emph{``bounded uniformity"} for topological groups. 
 The term ``equiuniform" was introduced in 
 \cite{Me-EqComp84}. 
  
 \begin{remarks} \label{f:G-bound} \ 
 	\begin{enumerate}
 		\item
 		Every $\mathcal{U}$-equiuniform action is continuous. 
 		\item 
 		Every compact $G$-space $K$ is equiuniform  
 		(with respect to its unique compatible uniformity).
 		
 		\item If the action on $X$ is $\Ucal$-equiuniform then both the completion and the Samuel compactification admit  natural continuous actions of $G$ which extend the original action. 
 		
 		\item \cite{Me-EqComp84} Proper $G$-compactifications on a $G$-space $X$ are exactly completions of equiuniform precompact uniformities on $X$.  
 	\end{enumerate}
 \end{remarks}

 \begin{f} \label{f:G-compactFacts} 
 	(\cite[Lemma 4.5]{Me-cs07} for topological groups $G$)  
 	
 	\noindent Let $(Y,\mathcal{U})$ be a uniform space 
 	and let $\pi \colon G \times Y \to Y$ be an action of a (semi) topological group $G$ with uniform $g$-translations.
 	\begin{enumerate}
 		\item  
 		 Suppose that there exists a $G$-invariant dense subset $X \subseteq Y$ such that the 
 		inherited action $G \times X \to X$ is 
 		$\mathcal{U}|_X$-equiuniform. Then the original action $\pi$ on $Y$ is $\mathcal{U}$-equiuniform  and continuous. 
 		\item Let $\pi \colon G \times X \to X$ be a (continuous) $\mathcal{U}$-equiuniform action on  $(X,\mathcal{U})$. Then the canonically extended 
 		$G$-action on the completion
 		$\widehat{\pi} \colon G \times \widehat{X} \to \widehat{X}$  
 		is $\widehat{\mathcal{U}}$-equiuniform and (continuous).    
 	\end{enumerate}	
 \end{f}
 \begin{proof}   
 	(1) 
 	We show that the action $\pi \colon G \times Y \to Y$
 	is $\mathcal{U}$-equiuniform (the continuity of $\pi$ is an easy corollary).  
 	Let $g_0 \in G$ and $\eps \in {\mathcal{U}}$. There exists 
 	an entourage $\eps_1 \in {\mathcal{U}}$ such that $\eps_1 \subset \eps$ and $\eps_1$ is a \textit{closed} subset of $Y \times Y$ (indeed, note that if $\eps_2\circ \eps_2 \circ \eps_2 \subseteq \eps$ then $\eps_2 \subseteq cl(\eps_2) \subseteq \eps$). Since $\pi_X \colon G \times X \to X$ is $\mathcal{U}|_X$-equiuniform, one may choose a neighborhood 
 	$U(g_0)$ of $g_0$ such that $(g_0x,gx) \in \eps_1$ for every $g \in U(g_0)$ and $x \in X$. For a given $y \in Y$, 
 	choose any net $(x_i)$ in $X$ which tends to $y$ in $Y$. Then for any given pair $(g_0, g)$ we have $\lim_i g_0x_i=g_0y$ and $\lim_i gx_i=gy$. Since $\eps_1$ is closed, we obtain that $(g_0y,gy) \in \eps_1 \subset \eps$ for every $g \in U(g_0)$ and $y \in Y$. 
 	
 	(2) Directly follows from (1) with $Y:=\widehat{X}$. 
 \end{proof}
 
 Recall that a continuous action of a topological group $G$ on a Tychonoff space $X$ does not necessarily admit a proper $G$-compactification; see the relevant results (about counterexamples and sufficient conditions) in \cite{Me-MaxEqComp}. 
 Equivalently, the greatest $G$-compactification $\beta_G \colon X \to \beta_G X$ is not necessarily proper (even for Polish $G$ and $X$) \cite{Me-Ex88}. Moreover,  $\beta_G X$ might be even a singleton for nontrivial $X$ (Pestov \cite{Pest-Smirnov}). Below we show that for order preserving actions no such counterexample is possible.

\subsection*{Circularly ordered inverse limits and compactifications} \label{s:InverseLim} 

\begin{lem} \label{invLimCirc} \cite{Me-OrdSem} 
	Let $X_{\infty}:=\underleftarrow{\lim} (X_i,I)$ be the inverse limit of the inverse system 
	$$\{f_{ij} \colon X_i \to X_j, \ \ j \leq i, \ i,j \in I\}$$ where $(I,\leq)$ 
	is a directed poset. Suppose that every $X_i$ is a c-ordered set with the c-order $R_i \subset X_i^3$ and each bonding map $f_{ij}$ is c-order preserving. On the inverse limit $X_{\infty}$ define a ternary relation $R$ as follows. 
	An ordered triple $(a,b,c) \in X_{\infty}^3$ belongs to $R$ iff $[p_i(a),p_i(b),p_i(c)]$ is in $R_i$ for some $i \in I$, where $p_i \colon X_{\infty} \to X_i$ are the natural  projections.  
	\begin{enumerate}
		\item Then $R$ is a c-order on $X_{\infty}$ and each projection 
		map $p_i \colon X_{\infty} \to X_i$ is c-order preserving.
		\item Assume in addition that every $X_i$ is a compact c-ordered space and each bonding map $f_{ij}$ is continuous. Then the inverse limit $X_{\infty}$ is also a c-ordered (nonempty) compact space.
	\end{enumerate}	
\end{lem}
\begin{proof} (1)  
	\nt	\textbf{Claim 1}: The family $\{p_i\}$ of all projections is 3-point separating.  
	
	Indeed, since $a,b,c \in X_{\infty}$ are distinct,  there exist indices $j(a,b), j(a,c), j(b,c) \in I$ such that 
	$$a_{j(a,b)} \neq b_{j(a,b)}, \  a_{j(a,c)} \neq c_{j(a,c)}, \ b_{j(b,c)} \neq c_{j(b,c)}.$$
	
	Since $I$ is directed we may choose $i \in I$ which dominates all three indexes 
	$j(a,b)$, $j(a,c), j(b,c)$. Then $a_i,b_i,c_i$ are distinct. 
	
\nt	\textbf{Claim 2}: \textit{If $[a_i,b_i,c_i]$ for some $i \in I$ and $a_j,b_j,c_j$ are distinct in $X_j$ for some $j \in I$ then $[a_j,b_j,c_j]$.}
	
	Indeed, choose an index $k \in I$ such that $i \leq k, j \leq k$ then $a_k,b_k,c_k$ are distinct.   
	Necessarily $[a_k,b_k,c_k]$. Otherwise, $[b_k,a_k,c_k]$ by the Totality axiom.  Then also $[b_i,a_i,c_i]$ because the bonding map $f_{ki} \colon X_k \to X_i$ is c-order preserving	and $a_i,b_i,c_i$ are distinct in $X_i$ (since $[a_i,b_i,c_i]$). 
	By $[a_k,b_k,c_k]$ it follows that $[a_j,b_j,c_j]$ because the bonding map $f_{kj} \colon X_k \to X_j$ is c-order preserving. 
	
	Now we show that $R$ is a c-order (Definition \ref{newC})  on  $X_{\infty}$. 
	
	The Cyclicity axiom is trivial. 
	Asymmetry axiom is easy by Claim 2. 
	
	Transitivity: by Claims 1 and 2 there exists $k \in I$ such that $[a_k,b_k,c_k]$ and $[a_k,c_k,d_k]$. Hence, $[a_k,b_k,d_k]$ by the transitivity of $R_k$. Therefore, $[a,b,d]$ in $X_{\infty}$ by the definition of $R$. 
	
	Totality: 
	if $a, b, c \in X$ are distinct, then $a_j,b_j,c_j$ are distinct for some $j \in I$ by Claim 1. By the totality of $R_j$ we have $[a_j, b_j, c_j]$ $\vee$ $[a_j, c_j, b_j]$, 
	hence also $[a, b, c]$ $\vee$ $[a, c, b]$ in $R$. 
	 
	We show that each projection $p_i \colon X_{\infty} \to X_i$ is c-order preserving.  Condition (COP1) of Definition \ref{d:c-ordMaps} is satisfied for every $i \in I$ 
	by Claim 2 and the definition of $R$. In order to verify condition (COP2) of Definition \ref{d:c-ordMaps}, assume that $p_i(a)=p_i(b)$ for some distinct $a, b \in X_{\infty}$. 
	We have to show that $p_i$ is constant on one of the closed intervals $[a,b], [b,a]$. 
	If not then there exist $u,v \in X_{\infty}$ such that $[a,u,b], [b,v,a]$ but $p_i(u) \neq p_i(a) \ne p_i(v)$.  As in the proof of Claim 1 one may choose an index $k \in I$ such that the elements $p_k(a), p_k(b), p_k(u), p_k(v)$ are distinct in $X_k$. Then we get that the bonding map $f_{ki} \colon X_k \to X_i$ does not satisfy condition (COP2).  
	This contradiction completes the proof.  
	
	(2) Let $\lambda_{\infty}$ be the usual topology of the inverse limit $X_{\infty}$.  
	It is well known that the inverse limit $\lambda_{\infty}$ of compact Hausdorff spaces 
	(with continuous $f_{ij}$) is nonempty and compact Hausdorff. Let $\lambda_{\medcirc}$ be the interval topology  
	of the c-order $R$ on $X_{\infty}$, 
	where $(X_{\infty},R)$ is defined as in Lemma \ref{invLimCirc}. We have to show that $\lambda_{\infty}=\lambda_{\medcirc}$. Since $\lambda_{\medcirc}$ is Hausdorff, it is enough to show that $\lambda_{\infty} \supseteq \lambda_{\medcirc}$. This is equivalent to showing 
	that every interval $(u,v)_o$ is $\lambda_{\infty}$-open in $X_{\infty}$ for every distinct $u,v \in X_{\infty}$. 
	Let $w \in (u,v)_o$; that is, $[u,w,v]$. By our definition of the c-order $R$ of $X_{\infty}$ we have $[u_i,w_i,v_i]$ in $X_i$ for some $i \in I$. 
	The interval $O_i:=(u_i,v_i)_o$  is open in $X_i$. Then its preimage $p_i^{-1}(O_i)$ is $\lambda_{\infty}$-open in $X_{\infty}$. On the other hand, 
	$$w \in p_i^{-1}(O_i) \subseteq (u,v)_o.$$
	Indeed, if $x \in p_i^{-1}(O_i)$ then $p_i(x) \in (u_i,v_i)_o$. This means that 
	$[u_i,x_i,v_i]$ in $X_i$. By the definition of $R$ we get that $[u,x,v]$ in $X_{\infty}$. So, $x \in (u,v)_o$. 	
\end{proof}

\begin{thm} \label{t:lim}  \cite{GM-UltraHom21} 
	Let $(X,R)$ be a c-ordered set and let $G$ be a subgroup of $H_+(X)$ with the pointwise topology, where $X$ carries the discrete topology $\t_{discr}$.  
	Then there exist: a c-ordered compact zero-dimensional space $X_{\infty}$ and a map $\pi_{\infty} \colon X \to X_{\infty}$ such that 
	\ben 
	\item $X_{\infty}=\underleftarrow{\lim} (X_F, \ I)$ is the inverse limit of finite c-ordered sets $X_F$, where $F \in I=Cycl(X)$. 
	
	\item $X_{\infty}$ is a compact c-ordered $G$-space and 
	$\pi_{\infty} \colon (X,\t_{discr}) \to X_{\infty}$ is a dense c-order preserving $G$-map which is a topological embedding of $(X,\t_{discr})$ into the compact space $X_{\infty}$.   
	\een 
\end{thm}
\begin{proof}  
	Let $F:=\{t_1,t_2, \cdots, t_m\}$ be an $m$-cycle on $X$.   
	We have a natural equivalence ``modulo-$m$" between $m$-cycles (with the same support). 
	For every given cycle $F:=\{t_1,t_2, \cdots, t_m\}$  
	define the corresponding finite disjoint covering $cov_F$ of $X$, by adding to the list: all 
	points $t_i$ and nonempty intervals $(t_i,t_{i+1})_o$ between the cycle points. More precisely we consider the following disjoint cover which can be thought of an equivalence relation on $X$.  
	$$
	cov_F:=\{t_1, (t_1,t_2)_o, t_2, (t_2,t_3)_o, \cdots, t_m, (t_m,t_1)_o \}. 
	$$ 
	Moreover, $cov_F$ naturally defines also a finite c-ordered set $X_F$ by ``gluing the points" of 
	$(t_i,t_{i+1})_o$ for each $i$. 
	So, the c-ordered set $X_F$ is the factor-set of the equivalence relation $cov_F$ and it contains at most $2m$ elements. 
	In the extremal case of $m=1$ (that is, for $F=\{t_1\}$) we define $cov_F:=\{t_1, X \setminus\{t_1\}\}$.  
	We have the following canonical c-order preserving onto map 
	\begin{equation} \label{projection} 
		\pi_F \colon X \to X_F, \  \ 
		\pi_F(x) =
		\begin{cases}
			t_i & {\text{for}} \ x=t_i\\
			(t_i,t_{i+1})_o & {\text{for}} \ x \in (t_i,t_{i+1})_o,
		\end{cases}
	\end{equation}  
	
	Let $I=Cycl(X)$ be the set of all nonempty finite
	injective cycles in $X$, considered up to cyclic
	permutation.
	Every finite subset $A\subseteq X$ with
	$|A|=m\geq1$ defines a cycle
	$F_A:\{1,\ldots,m\}\to X$, after choosing an enumeration
	compatible with the circular order.
	This cycle has image $A$ and is uniquely determined
	up to cyclic permutation.
	Thus cycles with the same support represent the
	same element of $Cycl(X)$.
	
	For representatives $F_1:C_{m_1}\to X$ and
	$F_2:C_{m_2}\to X$, define $F_1\leq F_2$ whenever
	$F_1$ is a \emph{sub-cycle} of $F_2$.
	This means that $m_1\leq m_2$ and
	$F_1(C_{m_1})\subseteq F_2(C_{m_2})$.
	The definition is independent of the representatives
	and gives a partial order on $Cycl(X)$.
	
	This partial order is directed.
	Indeed, for $F_1,F_2\in Cycl(X)$, let $F_3$ be the
	cycle whose support is
	$F_1(C_{m_1})\cup F_2(C_{m_2})$.
	Then $F_1\leq F_3$ and $F_2\leq F_3$. 
	
	Since every $cov_F$ is a disjoint partition,
	$cov_F^*=cov_F$; hence the directed refinement property
	shows that the family $\{cov_F:F\in Cycl(X)\}$ is a
	uniform base.
	
	For every $F \in Cycl(X)$ we have the disjoint finite 
	$\mu_X$-uniform covering $cov_F$ of $X$.  
	As before we can look at $cov_F$ as a c-ordered (finite) set $X_F$. Also, as in equation \ref{projection} we have a c-order preserving natural map $\pi_F \colon X \to X_F$ which are $\mu_X$-uniformly continuous into the finite (discrete) uniform space $X_F$. 
	Moreover, 
	if $F_1 \leq F_2$ then $cov_{F_2}$ refines  $cov_{F_1}$. This implies that the equivalence relation $cov_{F_2}$ is finer than $cov_{F_1}$.  
	We have a c-order preserving (continuous) onto bonding map $f_{F_2,F_1} \colon X_{F_2} \to X_{F_1}$ between finite c-ordered sets. Moreover, $f_{F_2,F_1} \circ \pi_{F_2}=\pi_{F_1}$. 
	
	In this way we get an inverse system 
	$
	\{f_{F_2,F_1} \colon X_{F_2} \to X_{F_1}, \ \ F_1 \leq F_2\}$, 
	where $(I,\leq)=Cycl(X)$ be the directed poset defined above. 
	It is easy to see that $f_{F,F}=id$ and $f_{F_3,F_1}=f_{F_2,F_1} \circ f_{F_3,F_2}$. 
	
	Denote by $X_{\infty}:=\underleftarrow{\lim} (X_F, \ I)  \subset \prod_{F \in I} X_F$ the corresponding inverse limit. Its typical element is 
	$\{(x_F) \colon F \in Cycl(X)\} \in X_{\infty}$, where $x_F \in X_F$.  The set 
	$X_{\infty}$ carries a circular order $R$ as in Lemma \ref{invLimCirc}. 
	On the other hand this inverse limit $X_{\infty}$ is c-ordered as it follows from Lemma \ref{invLimCirc}. 
	Moreover, $X_{\infty}$ is a compact c-ordered space. 
	
	\textbf{Definition of $\pi_{\infty} \colon X \to X_{\infty}$}. 
	
	Observe that $f_{F_2,F_1} \circ \pi_{F_2}=\pi_{F_1}$ for every $F_1 \leq F_2$.  By the universal property of the inverse limit we have the canonical uniformly continuous map $\pi_{\infty} \colon X \to X_{\infty}$  
	which is a dense c-order preserving embedding with discrete $\pi_{\infty}(X)$. Furthermore, if $X$ is countable then $X_{\infty}$ is a metrizable compact. 		
	
	Note that $(X,\mu_X)$ can be treated as the weak uniformity with respect to the family of maps $$\{\pi_F\colon X \to X_F: \ F \in Cycl(X)\}$$ 
	into finite uniform spaces $X_F$. The corresponding topology of $\pi_{\infty}(X)$ is discrete. 
It is easy to see that it is an embedding of uniform
 spaces and $\pi_{\infty}(X)$ is dense in $X_{\infty}$. Since $X$ is a precompact uniform space, its uniform completion is a compact space and can be identified with $X_{\infty}$. 
	 
	 \textbf{Definition of the action}
	 $G\times X_\infty\to X_\infty$.
	 
	 For $g\in G$ and $F\in I$, let
	 $g_F \colon X_F\to X_{gF}$ be the circular order isomorphism
	 induced by $g$.
	 It is well defined because $g$ sends each member of
	 $cov_F$ to the corresponding member of $cov_{gF}$.
	 It satisfies
	 $g_F\circ\pi_F=\pi_{gF}\circ g$.
	 For $F_1\leq F_2$, we also have
	 $f_{gF_2,gF_1}\circ g_{F_2}
	 =g_{F_1}\circ f_{F_2,F_1}$,
	 as follows by composing both sides with the
	 surjection $\pi_{F_2}$. 
	 For $x=(x_F)_{F\in I}\in X_\infty$, define
	 $g_\infty x$ by
	 \[
	 (g_\infty x)_{gF}:=g_F(x_F)
	 \quad(F\in I),
	 \qquad\text{equivalently}\qquad
	 (g_\infty x)_F
	 =g_{g^{-1}F}(x_{g^{-1}F}).
	 \]
	 This defines an element of $X_\infty$.
	 Indeed, for $F_1\leq F_2$, the preceding identity
	 and the compatibility of $x$ give
	 $f_{gF_2,gF_1}(g_{F_2}(x_{F_2}))
	 =g_{F_1}(f_{F_2,F_1}(x_{F_2}))
	 =g_{F_1}(x_{F_1})$.
	 Since $F\mapsto gF$ is an order automorphism of $I$,
	 these are precisely the required compatibility
	 conditions.
	 
	 The identities $e_F=\operatorname{id}_{X_F}$ and
	 $(gh)_F=g_{hF}\circ h_F$ imply
	 $e_\infty=\operatorname{id}_{X_\infty}$ and
	 $(gh)_\infty=g_\infty\circ h_\infty$.
	 Thus we obtain an action of $G$ on $X_\infty$.
	 Writing $p_F\colon X_\infty\to X_F$ for the coordinate
	 projection, we have
	 $p_F\circ g_\infty
	 =g_{g^{-1}F}\circ p_{g^{-1}F}$.
	 Consequently, every $g_\infty$ is continuous and
	 has continuous inverse $(g^{-1})_\infty$.
	 
	 The action preserves the circular order.
	 Indeed, if $[x,y,z]$ in $X_\infty$, choose $F$
	 such that $[x_F,y_F,z_F]$ in $X_F$.
	 Applying the circular order isomorphism $g_F$ gives
	 $[(g_\infty x)_{gF},(g_\infty y)_{gF},
	 (g_\infty z)_{gF}]$ in $X_{gF}$.
	 Hence $[g_\infty x,g_\infty y,g_\infty z]$
	 in $X_\infty$.
	 
	 To prove joint continuity, fix $x\in X_\infty$
	 and $F\in I$.
	 The pointwise stabilizer
	 $G_{(F)}:=\{h\in G:ht=t\text{ for every }t\in F\}$
	 is open in $G$.
	 Every $h\in G_{(F)}$ preserves each member of $cov_F$,
	 so $hF=F$ and $h_F=\operatorname{id}_{X_F}$.
	 Consequently, $(h_\infty y)_F=y_F$ for every
	 $y\in X_\infty$.
	 Thus the action sends
	 $G_{(F)}\times p_F^{-1}(\{x_F\})$
	 into $p_F^{-1}(\{x_F\})$.
	 These cylinder neighborhoods form a base at $x$,
	 because $I$ is directed.
	 This proves continuity at every $(e,x)$.
	 Continuity everywhere follows from the action law
	 and the fact that fixed translations on $G$ and
	 $X_\infty$ are homeomorphisms.
	 
	 Finally, for $a\in X$ and $F\in I$, we have
	 $(g_\infty\pi_\infty(a))_{gF}
	 =g_F(\pi_F(a))=\pi_{gF}(ga)$.
	 Therefore $g_\infty\pi_\infty(a)=\pi_\infty(ga)$,
	 so $\pi_\infty$ is a $G$-map.
	 
	 Since each $g$ permutes the uniform base
	 $\{cov_F:F\in I\}$, it is a uniform automorphism
	 of $(X,\mu_X)$.
	 Under the preceding identification of its completion
	 with $X_\infty$, its canonical extension is
	 therefore $g_\infty$.
	 Hence $X_\infty$ is a compact circularly ordered
	 $G$-space with the required equivariant embedding.  
\end{proof} 
 
\begin{thm} \label{t:limLin} \cite{Me-OrdSem} 
	Let $(X,\leq)$ be a linearly ordered set and $G$ be a subgroup of $H_+(X)$ with the pointwise topology (arising from discrete \(X\)). 
	Then there exist: a linearly ordered compact zero-dimensional space $X_{\infty}$ and a map $\nu \colon X \hookrightarrow X_{\infty}$ such that 
	\ben 
	
	\item $X_{\infty}=\underleftarrow{\lim} (X_F, I)$ is the inverse limit of finite linearly ordered sets $X_F$, where $F \in I$. 
	
	\item $X_{\infty}$ is a compact linearly ordered $G$-space and 
	$\nu \colon X \hookrightarrow X_{\infty}$ is a dense topological $G$-embedding of the discrete set $X$ such that $\nu$ is a LOP map.  
	\een 
\end{thm}
\begin{proof}
	Argue as in Theorem~\ref{t:lim}, replacing finite cycles by finite chains
	$F=\{t_1<\cdots<t_m\}\in \rm{Chain}(X)$ and circular quotients by the corresponding finite linearly
	ordered quotients $X_F$ determined by the clopen partitions
	\[
	\rm{cov}_F=\{(\leftarrow,t_1),\{t_1\},(t_1,t_2),\{t_2\},\dots,(t_{m-1},t_m),\{t_m\},(t_m,\to)\}.
	\]
	Then the inverse limit $X_\infty=\varprojlim (X_F,I)$ is a compact zero-dimensional LOTS, and the
	canonical map $\nu:X\to X_\infty$ is a dense topological embedding of the discrete set $X$ such that $\nu$ is LOP. 
\end{proof}

\begin{remark} \label{r:Q} 
	Let $X=(\Q,\leq)$ be the rationals with usual order but equipped with the discrete topology. Consider the Polish  automorphism group $G:=\Aut(\Q,\leq)$ of all LOP permutations 
	with the pointwise topology.
	One may apply Theorem \ref{t:limLin} getting the linearly ordered $G$-compactification 
	$\nu \colon X \hookrightarrow X_{\infty}$ (where $X_{\infty}$ is metrizable and zero-dimensional). This compactification has a remarkable property (as we show in \cite{Me-MaxEqComp}). Namely, 
	$\nu$ is the \textit{maximal $G$-compactification} for the $G$-space $\Q$ (where $\Q$ is discrete). The same is true for every dense subgroup $G$ of $\Aut(\Q,\leq)$ (e.g.,  for Thompson's group $F$). 
	    
	Similar result is valid for the circular version. Namely, the rationals on the circle with its circular order $\Q_0:=(\Q / \Z,\circ)$, the automorphism group $G:=\Aut(\Q_0,\circ)$ 
	and its dense subgroups $G$ (for instance, Thompson's circular group $T$). 
\end{remark}  

\begin{f} \label{f:GM} \cite{GM-UltraHom21,Me-AutKey}   
Let $X=(\Q_0,\circ)$ and let  $G:=\Aut(\Q_0,\circ)$ be a Polish group with pointwise topology,
and let $H:=St(q)$ be the stabilizer subgroup of a point $q \in \Q_0$. Then the following hold:  
	\begin{enumerate}
		\item The natural right uniformity $\mu_r$ on $G/H=Q_0$ is a precompact equiuniformity with a uniform base $\mathcal B_r$
		consisting of the disjoint coverings $cov_F$, $F\in \rm{Cycl}(\Q_0)$.
		The completion of $\mu_r$ is the greatest $G$-compactification $\beta_G (\Q_0)$ of the discrete $G$-space  $\Q_0$. 
		\item $\beta_G (\Q_0) = \rm{trip}(\T,\Q_0)  =\underleftarrow{\lim} (X_F, I)$ is the inverse limit of finite c-ordered sets $X_F$.
		
		\noindent Geometrically, $\rm{trip}(\T;\Q_0)$ is a circularly ordered metrizable compact space which we get from the circle $\T$ after replacing any rational point $q\in \Q_0$ by the ordered triple $[q^{-}, q, q^{+}]$.  See Figure \ref{fig:two}. 
		\begin{figure}[h]  
			\begin{center} 
				\scalebox{0.08}{\includegraphics{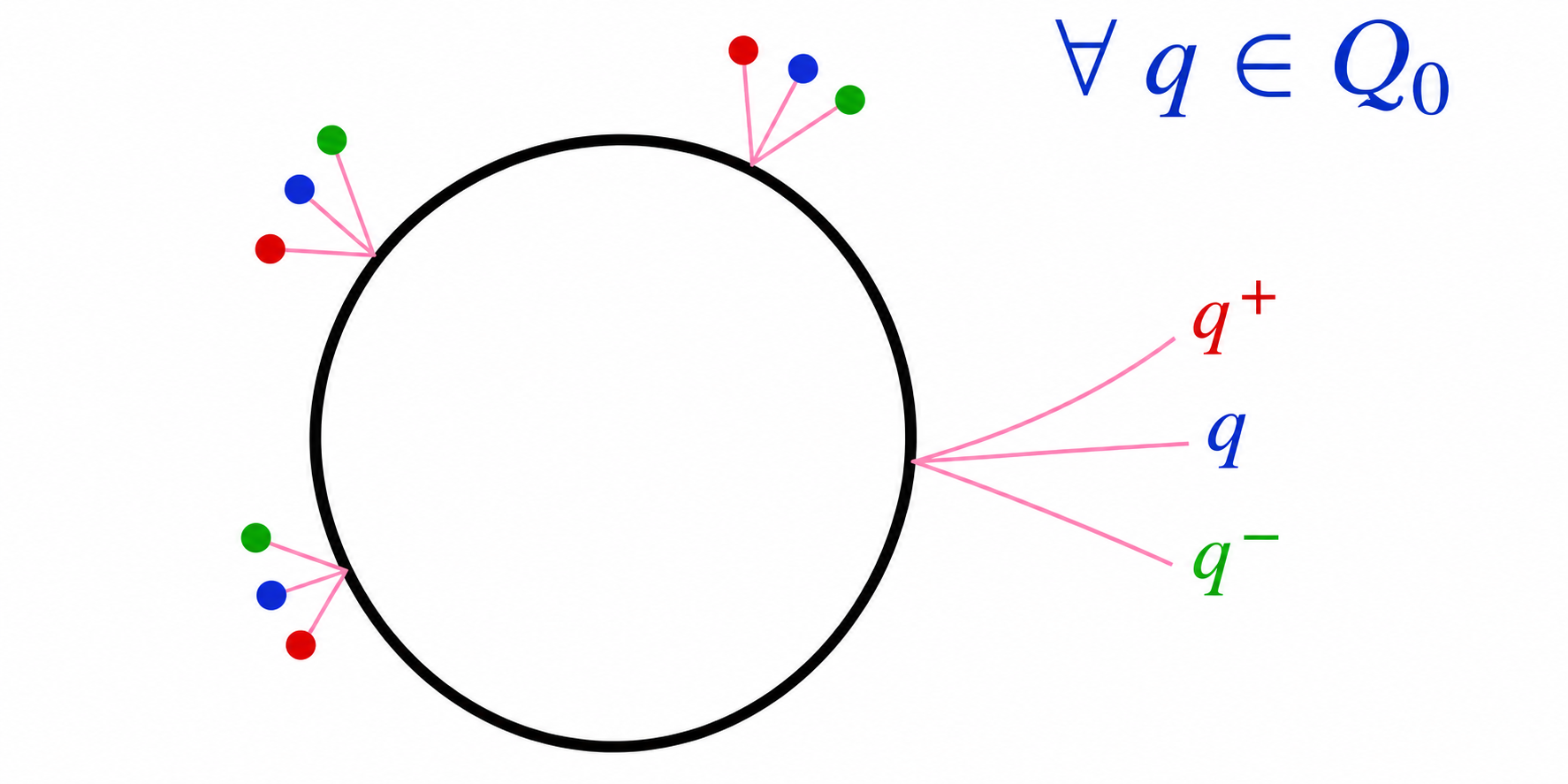}}
				\caption{Geometric description of $\beta_G (\Q_0)$}
			\end{center} 
			\label{fig:two}
		\end{figure}   
		\item $\beta_G (\Q_0)=M(G) \cup \Q_0$ and $M(G)=\operatorname{Split}(\T,\Q_0)$ is the universal minimal $G$-system of $G$.  
	
		\noindent Geometrically, $M(G)$ is the circle after splitting any rational point (removing isolated $q$ from that triple $q^{-}, q, q^{+}$). 
	\end{enumerate}	
\end{f} 
Recall that the Polish topological group $H_+(\T,\circ)$ is Roelcke precompact \cite{GM-tame}.   
The same is true for $\Aut(\Q_0,\circ)$ \cite{GM-UltraHom21}).  
Several important generalizations of these results can be found also in \cite{Sorin25-Sbornik}.  
Many remarkable results about automorphism groups of ordered spaces were obtained recently by B.V. Sorin, G.B. Sorin and K.L. Kozlov.  
Maximal $G$-compactifications of 
ultratransitive $G$-actions on discrete circularly (and linearly) ordered sets admit a general description \cite{Sorin25-Sbornik,KS}. Sometimes such compactifications are minimal ordered compactifications forcing the uniqueness of ordered $G$-compactifications in certain cases. 
Furthermore, a systematic study of Roelcke precompactness for automorphism groups and their subgroups for several ultratransitive actions was done in \cite{Sorin25-Sbornik,Sorin23-Roelcke}. For Ellis  compactifications in this context we refer to \cite{KS,Sorin23-Roelcke}.  

\subsection*{Orderly topological groups} Let us say that a topological group $G$ is \emph{orderly} (c-\emph{orderly}) if $G$ is a topological subgroup of $H_+(X)$ for some linearly (resp., circularly) ordered compact space $X$. 

\begin{thm} \label{t:LinCase} 
	Let $G$ be an abstract group. 
	The following are equivalent:
	\ben 
	\item $G$ is left linearly orderable (for the standard definition and properties see \cite{DNR}); 
	\item $(G,\t_{discr})$ is orderly.  
	\item $G$ embeds algebraically into the group $\Aut(X,\leq)$ for some linearly ordered set $(X,\leq)$.  
	\een
	In assertion (2), the compact space can additionally be chosen to be zero-dimensional.  
\end{thm}  
\begin{proof}
	(1) $\Rightarrow$ (2) 
	One may use the compactification $\nu \colon X \hookrightarrow X_{\infty}$ from Theorem \ref{t:limLin}, where $G=X$, $K=X_{\infty}$ and $\dim X_{\infty}=0$.   
	
	(2) $\Rightarrow$ (3) Trivial. 
	
	(3) $\Rightarrow$ (1)  The well-known proof (\cite{CR-b,DNR}) is to use a \emph{dynamically lexicographic order} on $G$.  
\end{proof}

The following result is a circular analog of a well-known result for linear orders. 

\begin{f} \label{t:S_+} \emph{(Zheleva \cite{Zheleva97})} 
Let $(X,\circ)$ be a c-ordered set and $G \times X \to X$ be an effective c-order preserving action. Then the group $G$ (e.g., $\Aut(X,\circ)$) admits a left invariant c-order. 
\end{f}

The following result is slightly stronger than Fact \ref{t:S_+}. 

\begin{thm} \label{t:c-ordCase} \cite{Me-OrdSem} 
Let $G$ be an abstract group. 
The following are equivalent:
\ben 
\item $G$ is left circularly orderable;   
\item $(G,\t_{discr})$ is c-orderly.  
\item $G$ embeds algebraically into the group $\Aut(X,\circ)$ for some circularly ordered set $(X,\circ)$.    
\een 
In assertion (2), the compact space can additionally be chosen to be zero-dimensional.
\end{thm} 
\begin{proof}
(1) $\Rightarrow$ (2)  
Let $(G,R)$ be a c-ordered group. By Theorem \ref{t:lim}, the COP compactification $\nu \colon G \to K=G_{\infty}$ is a c-order-preserving proper $G$-compactification which induces a topological embedding of (discrete) $G$ into $H_+(K)$.  

(2) $\Rightarrow$ (3) Trivial. 

(3) $\Rightarrow$ (1) Apply Fact \ref{t:S_+}.  
\end{proof}

\section{Split spaces and topological properties of COTS} \label{s:split} 

\begin{defin}[Single-split space] \label{d:cover} Let $(X,R)$ be a circularly ordered space and fix $c\in X$. Consider the standard cut $\leq_c$ (see Remark \ref{r:chech}) at $c$ where $c$ becomes the least element. Denote it by $c^{-}$. Then we get a linearly ordered set $X$ by declaring $x <_c y$ whenever $(c,x,y) \in R$. Adding to $X$ a new point $c^{+}$ as the greatest element we get a linearly ordered set $X(c)=[c^-,c^+]=X\cup \{c^+\}$ and a natural onto projection map 
	$$
	q \colon X(c) \to X, \ \ \ q(c^{-})=q(c^{+})=c, \ q(x)=x \ \ \forall x \in (c^{-}, c^{+})
	$$ 
	sometimes denoted also by $q_c$.  
\end{defin}

Let $R_{\leq_c}$ be the circular order induced by the linear order $\leq_c$ on $X(c)$. 
The corresponding, in general weaker, circular order topology $\lambda_{R_{\leq_c}}$ is the topology $\lambda_{\leq_c}$ of the linear order by Lemma \ref{l:cWEAKER}.2.

\subsection*{Splitting a subset} 

\begin{lem}\label{c-doubling}
	Let $R$ be a circular order on a set $X$ and $\emptyset \neq A \subseteq X$. There exist: a canonically defined circularly ordered set
	$X_A = (\operatorname{Split}(X; A), R_A)$ and an onto map $q \colon X_A \to X$ such that the following conditions are satisfied:
	\begin{enumerate}
		\item [(i)] $q\colon X_A \to X$ is c-order preserving and the preimage $q^{-1}(a)$ of any $a \in A$ consists of exactly two points,
		while $q^{-1}(x)$ is a singleton for every $x  \in X \setminus A$.
		Moreover, $(X_A,\lambda_{R_A})$, as a topological space, is a $\LOTS$.
		
		\item [(ii)] $q\colon X_A \to X$ is a continuous closed (quotient) at most 2-to-1 map. In particular, $q$ is a perfect map.
		
		\item [(iii)] If $X$ is (locally) compact then $X_A$ is (locally) compact.
		
		\item [(iv)] If $X$ is separable and $A$ is countable, then $X_A$ is  separable.
		
		\item [(v)] If $X$ is second countable and $A$ is countable, then $X_A$ is second countable.
		
		\item [(vi)] If $X$ is metrizable and $A$ is countable then $X_A$ is metrizable.
	\end{enumerate}
\end{lem}

\begin{proof}
	\noindent{\bf (i)}
	$X_A$, as a set, is $\{a^{-}, a^{+}: a \in A \} \cup (X \setminus A)$. We have the onto projection
	$$q\colon X_A \to X,\qquad a^{\pm} \mapsto a,\qquad x \mapsto x \ \ \forall \ a \in A,\ \forall x \in X \setminus A.$$
	Define a natural circular order $R_A$ on $X_A$ by the following two rules:
	
	\smallskip
	\noindent$\bullet$ \ $[u,w,v]_{R_A}$ for every $(u,w,v) \in X_A^3$ such that $[q(u),q(w),q(v)]_R$ holds in $X$.
	
	\smallskip
	\noindent$\bullet$ \ $[a^{-}, u, a^{+}]_{R_A}$, \ $[u,a^{+}, a^{-}]_{R_A}$, \ $[a^{+}, a^{-}, u]_{R_A}$
	for every $a \in A$ and every $u \notin \{a^+,a^-\}$.
	
	The relation $R_A$ is a circular order. Indeed, it is the
	order induced by the embedding into
	$X\otimes_c\{-1,+1\}$ given by
	\[
	a^-\mapsto(a,+1),\qquad
	a^+\mapsto(a,-1)\quad(a\in A),\qquad
	x\mapsto(x,-1)\quad(x\in X\setminus A).
	\]
	In particular, $(a^+,a^-)_{R_A}=\emptyset$.
	The projection $q$ satisfies {\rm(COP1)} by construction.
	Its fibers are singletons or the convex adjacent pairs
	$
	\{a^+,a^-\}=[a^+,a^-]_{R_A}.
	$
	Hence $q$ also satisfies {\rm(COP2)}, and therefore is
	c-order preserving.

	For empty $A$ identify $X_A$ with $X$. If $A$ is nonempty, then $(X_A,\lambda_{R_A})$, as a topological space, is a LOTS.
	Indeed, take any fixed $a \in A$ and note that the induced circular order of the linearly ordered set
	$X_{A\setminus \{a\}}(a)=[a^-,a^+]$ is just the circular order $R_A$ of $X_A$. Now use Lemma \ref{l:cWEAKER}.2. 
	
	The remaining assertions are immediate when $|X|\le2$,
	since both $X$ and $X_A$ are finite discrete spaces.
	Henceforth assume that $|X|\ge3$.
	
	\noindent{\bf Preparations for (ii).}
	We will use the identification $X\setminus A\subseteq X_A$ (so that $q(x)=x$ for $x\in X\setminus A$).
	Also, for $u\in X_A$ 
	we write $\underline{u}:=q(u) \in X.$
	Thus $\underline{u}=u$ whenever $u\in X\setminus A$.
		
	\noindent \textbf{Claim 1.}
	For every $a\in A$ define
	$(a^+)^+=(a^-)^+=a^+$ and
	$(a^+)^-=(a^-)^-=a^-$.
	For every $x\in X\setminus A$ put $x^-=x^+=x$.
	Then the following hold.
	\begin{enumerate}
		\item
		If $a\in A$ and $[\underline{u},a,\underline{v}]_R$,
		then $(u,a^-)_{R_A}\cup(a^+,v)_{R_A}=(u,v)_{R_A}$.
		
		\item
		If $\underline{u}\neq\underline{v}$, then
		$(u^-,v^+)_{R_A}\subseteq(u,v)_{R_A}$.
		
		\item
		If $\underline{u}\neq\underline{v}$, then
		$q^{-1}((\underline{u},\underline{v})_R)
		=(u^-,v^+)_{R_A}$, and consequently
		$q((u^-,v^+)_{R_A})
		=(\underline{u},\underline{v})_R$.
		
		\item
		The following families are local bases:
		\begin{enumerate}
			\item
			$\{(v,a^-)_{R_A}:v\in X_A,\
			\underline{v}\neq a\}$ at $a^+$, for every $a\in A$.
			
			\item
			$\{(a^+,u)_{R_A}:u\in X_A,\
			\underline{u}\neq a\}$ at $a^-$, for every $a\in A$.
			
			\item
			$\{(u^-,v^+)_{R_A}:u,v\in X_A,\
			x\in(\underline{u},\underline{v})_R\}$
			at $x$, for every $x\in X\setminus A$.
		\end{enumerate}
	\end{enumerate}
	
	\begin{proof}
		Assertions (1)--(3) follow directly from the definition
		of $R_A$. The bases in (4a)--(4b) are the endpoint bases
		of the LOTS $[a^-,a^+]$.
		For (4c), every neighborhood of an unsplit point $x$
		contains an interval $(u,v)_{R_A}$ with
		$\underline{u}\neq\underline{v}$.
		Indeed, an interval containing $x$ whose endpoints
		both lie over $b$ equals $X_A\setminus q^{-1}(b)$;
		a point of $X\setminus\{b,x\}$ allows us to refine it
		to an interval with distinct endpoint images.
		Now apply (2) and (3).
	\end{proof}	 
			
			\noindent \textbf{Claim 2.}
			The projection $q\colon X_A\to X$ is a closed map.
			
			\begin{proof}
				Let $K\subseteq X_A$ be closed and put
				$U=X_A\setminus K$. Take $x\in X\setminus q(K)$.
				If $x\notin A$, Claims~1.4(c) and~1.3 give an open
				interval $I$ containing $x$ such that
				$q^{-1}(I)\subseteq U$. Hence $I\cap q(K)=\emptyset$.
				
				Suppose $x=a\in A$. Choose $b,c\in X$ with
				$[b,a,c]_R$ and put $V=U\cap(b^-,c^+)_{R_A}$.
				Both $a^+,a^-$ belong to the open set $V$.
				By Claim~1.4(a)--(b), choose basic neighborhoods
				$(u,a^-)_{R_A}\subseteq V$ of $a^+$ and
				$(a^+,v)_{R_A}\subseteq V$ of $a^-$.
				These containments imply
				$\underline{u}\in[b,a)_R$ and
				$\underline{v}\in(a,c]_R$;
				hence $[\underline{u},a,\underline{v}]_R$.
				By Claim~1.1,
				$(u,v)_{R_A}=(u,a^-)_{R_A}\cup(a^+,v)_{R_A}
				\subseteq U$.
				Claims~1.2--1.3 now give
				$q^{-1}((\underline{u},\underline{v})_R)
				=(u^-,v^+)_{R_A}\subseteq U$.
				Thus $(\underline{u},\underline{v})_R$ is an open
				neighborhood of $a$ disjoint from $q(K)$.
				Therefore $q(K)$ is closed.
			\end{proof}

\smallskip
\noindent{\bf (ii)}
For $\underline{u}\neq\underline{v}$, Claim~1.3 gives
$q^{-1}((\underline{u},\underline{v})_R)
=(u^-,v^+)_{R_A}$.
Hence $q$ is continuous.
By Claim~2, it is a closed surjection.
Since its fibers are finite, $q$ is a perfect
quotient map.

\smallskip
\noindent{\bf (iii)}
This follows from the perfectness of $q$, since
inverse images of compact subsets are compact.
For the local assertion, take compact neighborhoods
in $X$.

\smallskip
\noindent{\bf (iv)}
Let $D$ be a countable dense subset of $X$.
Then $D_A:=q^{-1}(D\cup A)$ is countable.
Let $W\subseteq X_A$ be nonempty and open.
If $W$ meets $q^{-1}(A)$, then $W\cap D_A\neq\emptyset$.
Otherwise choose $x\in W$, so that
$q^{-1}(q(x))=\{x\}$.
By closedness of $q$, there is an open neighborhood
$I$ of $q(x)$ such that $q^{-1}(I)\subseteq W$.
Since $D\cap I\neq\emptyset$, we again obtain
$W\cap D_A\neq\emptyset$.
Thus $D_A$ is dense in $X_A$.

\smallskip
\noindent{\bf (v)}
Let $\mathcal B$ be a countable base of $X$.
For each $a\in A$, choose disjoint open neighborhoods
$U_a^+$ and $U_a^-$ of $a^+$ and $a^-$.
The sets $q^{-1}(B)$ and $q^{-1}(B)\cap U_a^\pm$,
where $B\in\mathcal B$ and $a\in A$, form a
countable family of open subsets of $X_A$.

If an open set $W$ contains the entire fiber
$q^{-1}(x)$, closedness of $q$ gives
$x\in B\in\mathcal B$ with $q^{-1}(B)\subseteq W$.
This proves the base property at unsplit points.
If $W$ is an open neighborhood of $a^+$, apply
this observation to $W\cup U_a^-$.
We obtain $a\in B\in\mathcal B$ with
$q^{-1}(B)\subseteq W\cup U_a^-$, and hence
$a^+\in q^{-1}(B)\cap U_a^+\subseteq W$.
The argument at $a^-$ is symmetric.
Thus the indicated family is a countable base
of $X_A$.

\smallskip
\noindent{\bf (vi)}
Since $X$ is metrizable and $q$ is continuous,
$E:=(q\times q)^{-1}(\Delta_X)$ is a $G_\delta$
subset of $X_A^2$.
Moreover,
$E\setminus\Delta_{X_A}
=\{(a^+,a^-),(a^-,a^+):a\in A\}$
is countable.
Since $X_A$ is Hausdorff, each of these points
is closed in $X_A^2$.
Deleting them from $E$ shows that
$\Delta_{X_A}$ is a $G_\delta$ subset of $X_A^2$.
By (i) and Fact~\ref{f:propLOTS}.3,
$X_A$ is metrizable.
\end{proof}

For the c-ordered circle $X:=\T$ with $A=\T$ (splitting all points)  we get the ``double circle" of Ellis \cite{Ellis}, which we denote by $\T_{\T}$.  
It is COP-isomorphic to $\T\otimes_c\{-1,+1\}$
under the identification specified in the proof of (i).

\begin{cor} \label{cover}    
	Let $(X,R)$ be a circularly ordered space and $c \in X$. 
	Then for the single-split space $X(c)=[c^-,c^+]$ the following conditions are satisfied: 
	\begin{itemize}
		\item [(i)] The corresponding map $q \colon X(c) \to X$ is a continuous perfect c-order preserving map.  
		\item [(ii)] 
		$X$ is compact (resp. locally compact, metrizable, separable) if and only if $X(c)$ is compact (resp. locally compact, metrizable, separable).  
		\item [(iii)]  
		The restriction of $q$ to $X(c) \setminus \{c^-, c^+\}$ is a homeomorphism with $X \setminus \{c\}$. 
	\end{itemize} 
\end{cor} 
	\begin{proof}
		(i) Apply Lemma~\ref{c-doubling} to $A=\{c\}$,
		identifying $X_A$ with $X(c)$.
		
		(ii) The implications from $X$ to $X(c)$ follow
		from Lemma~\ref{c-doubling}.
		Conversely, compactness and separability are
		preserved by continuous surjections.
		Local compactness is preserved by perfect maps,
		and metrizability is preserved by perfect maps
		from metrizable spaces
		\cite[Theorem~4.4.15]{Eng89}.
		
		(iii) The restriction of $q$ to
		$q^{-1}(X\setminus\{c\})=X(c)\setminus\{c^-,c^+\}$
		is a continuous closed bijection onto
		$X\setminus\{c\}$, hence a homeomorphism.
	\end{proof}

\begin{cor} \label{c:normality} \ 
	\begin{enumerate}
		\item Every $\COTS$ is monotonically normal (hence  hereditarily collectionwise normal).  
		\item 
		A $\COTS$ $X$ is metrizable if and only if the diagonal of $X^2$ is a $G_{\delta}$ subset. 
		\item A $\COTS$ $X$ is separable if and only if it is hereditarily separable. 
	\end{enumerate}	 
\end{cor}
\begin{proof} (1) 
	By Fact \ref{f:propLOTS} every LOTS is monotonically normal (hence also, hereditarily collectionwise normal). 
	This property is hereditary 
	and is preserved by closed continuous surjective maps, \cite{HLZ}. Now apply Corollary \ref{cover} which shows that the map $q \colon X(c) \to X$ is closed for every COTS $X$. 
	
	(2) The necessity is immediate.
	Conversely, suppose that $\Delta_X$ is a $G_\delta$
	subset of $X^2$.
	For the continuous map $q\colon X(c)\to X$,
	the set $E:=(q\times q)^{-1}(\Delta_X)$ is a
	$G_\delta$ subset of $X(c)^2$.
	Since
	$E\setminus\Delta_{X(c)}
	=\{(c^-,c^+),(c^+,c^-)\}$
	is finite and its points are closed,
	$\Delta_{X(c)}$ is a $G_\delta$.
	By Fact~\ref{f:propLOTS}.3, $X(c)$ is metrizable.
	Corollary~\ref{cover}(ii) now implies that
	$X$ is metrizable. 
	
	(3) Assume that $X$ is separable. By
	Corollary~\ref{cover}(ii), the LOTS $X(c)$ is
	separable, hence hereditarily separable by \cite{LB}.
	Since hereditary separability is preserved by continuous
	surjections, $X=q(X(c))$ is hereditarily separable.
	The converse is immediate. 
\end{proof}

\begin{lem} \label{c-doublingGEN}
	In the setting of Lemma \ref{c-doubling}, assume a discrete group $G$ acts on $X$ by COP maps and $A\subseteq X$ is $G$-invariant. Then the action $G \times X \to X$ induces a continuous COP action $G \times X_A \to X_A$ such that $q \colon X_A\to~X$ is a $G$-map.
\end{lem}
\begin{proof}
	Define $g(a^\pm)=(ga)^\pm$ for $a\in A$ and
	$g(x)=gx$ for $x\in X\setminus A$.
	These formulas define an action preserving the
	two defining rules of $R_A$.
	Thus each $g$ acts by a COP homeomorphism.
	Since $G$ is discrete, the action is continuous.
	Equivariance of $q$ follows directly from the formulas.
\end{proof}

\begin{remark} \label{r:add} The $G$-space $X_A$ from Lemma \ref{c-doublingGEN} was denoted by $\operatorname{Split}(X,G; A)$ in \cite{GM-c}.  
	An important particular case is 
	\cite[Example 14.10]{GM1}. It gives 
	a concrete realization of the Sturmian symbolic system. 
	In this case the corresponding circularly ordered $\Z$-space is $\T_A=\operatorname{Split}(\T,\Z; A)$ with $A:=\{m\al, n(1-\al) : m, n \in \Z\}$ and the action of $\Z$ generated by an irrational angle $\al$.  
\end{remark}

\section{Generalized circularly ordered spaces and COP actions} 
\label{s:GCOTS} 

Recall that a topological space $X$ is said to be \textit{generalized orderable} (GO) (we use also  the notation: GLOTS) if $X$ is homeomorphic to a subspace of a LOTS. This concept is well known and goes back to E. \v{C}ech. For basic facts we refer to \cite{BenLut,Kem}. 

More informatively, define the class GLOTS$_{\curlyeqprec}$ of all 
\textit{generalized linearly ordered topological spaces} $(X,\leq,\tau)$. This means that there exist $(Y,\leq,\lambda_{\leq}) \in$ $\LOTS_\curlyeqprec$ and a bi-embedding 
$$i \colon (X,\leq,\tau) \hookrightarrow (Y,\leq,\lambda_{\leq}).$$

Define the natural forgetful assignment
$\operatorname{GLOTS}_{\curlyeqprec}\to\operatorname{GLOTS},
\ (X,\leq,\tau)\mapsto(X,\tau)$.

\begin{remark} \label{r:constr} 
It is well known that the following conditions are equivalent:
	\begin{enumerate}
		\item $(X,\leq,\tau) \in \operatorname{GLOTS}_{\curlyeqprec}$. 
		\item $\lambda_{\leq} \subseteq \tau$ and there exists a base $\mathcal{B}$ of the topology $\tau$, where every member $O \in \mathcal{B}$ is $\leq$-convex. 
		\item $(X,\leq,\tau)$ is bi-embedded into a LOTS $(Y,\leq_Y,\t_{\leq_Y})$ as a closed subset. 
	\end{enumerate}
	Let us recall the implication $(2) \Longrightarrow (3)$, which goes back to E. \v{C}ech. 
	Define:  
	$$X^-:=\{x\in X:\ [x,\to)\in \tau\setminus\lambda_{\leq}\},\qquad
	X^+:=\{x\in X:\ (\leftarrow,x]\in \tau\setminus\lambda_{\leq}\}.$$  
	Define a subset $X^*$ of $X \times \Z$ as follows:  
	$$
	X^*:=(X\times\{0\})
	\;\cup\; \{(x,k): x\in X^-,\ k<0,\ k\in\Z\}
	\;\cup\; \{(x,n): x\in X^+,\ 0<n,\ n\in\Z\}.
	$$
	Consider the linear order on $X^*$ inherited from the lexicographic linear order on $X \otimes_l \Z$. 
	Then 
	
	(A) the topological subspace $X \times \{0\}$ of $X^*$ is naturally homeomorphic to $(X,\tau)$.
	
	(B) $X$ is closed in $X^*$.   
	
	In fact, every point $u \in X^* \setminus X$ is isolated in $X^*$. This immediately implies (B).  
\end{remark} 

\begin{defin}[First definition] \label{d:GCO} \cite{Me-OrdSem} A topological space $X$ is said to be a \textit{generalized circularly orderable topological space} if $X$ is homeomorphic to a subspace of a COTS $Y$. Notation: $\GCOTS$, or sometimes simply: GCO as in \cite{Me-OrdSem}.  
\end{defin}

By Corollary \ref{c:COTSisCOMPACTIFIABLE}, one may assume that $Y$ is a \textbf{compact} COTS. Similar to the case of GLOTS, as expected, Definition \ref{d:GCO} is equivalent (see Proposition \ref{p:GCO-eq}) to the requirement: there exists a base $\mathcal{B}$ of the topology $\tau$ and a circular order $R$ such that $\lambda_R \subseteq \tau$ and every member $O \in \mathcal{B}$ is a $R$-convex subset (Definition \ref{d:c-convex}). 
More informatively, define the class $\GCOTS_{\medcirc}$ 
of all 
\textit{generalized circularly ordered topological spaces} $(X,R,\tau)$. This means that there exist $(Y,\circ,\lambda_{\medcirc}) \in$ COTS$_{\medcirc}$ and a bi-embedding 
$i \colon (X,R,\tau) \hookrightarrow (Y,\circ,\lambda_{\medcirc}).$ 

So, we have the natural forgetful assignment
$
\operatorname{GCOTS}_{\medcirc} \to \operatorname{GCOTS}, \ (X,R,\tau) \mapsto (X,\tau).
$

Every $(X,\leq,\tau) \in \GLOTS_\curlyeqprec$ is a partially ordered space (in the sense of Definition \ref{d:ord}) and every $(X,R,\tau) \in \GCOTS_{\medcirc}$ is a partially c-ordered space.
  
\begin{defin} \label{d:c-sing}  Let $(X,R,\tau) \in \GCOTS_{\medcirc}$ be a generalized circularly ordered space. Let us say that $u \in X$ is a \textit{right-singular point} and write $u \in X^{+}$ if there exists $a \in X$ such that $(a,u]$ is a $\tau$-open subset which is not $\lambda_{\medcirc}$-open. 
	Similarly, we say that $v \in X$ is \textit{left-singular} and write $v \in X^{-}$ if there exists $c \in X$ such that $[v,c)$ is a $\tau$-open subset which is not $\lambda_{\medcirc}$-open.   
\end{defin}  

The following is a natural, completely expected analog of the linear case (Remark \ref{r:constr}).  

\begin{prop} \label{p:GCO-eq}
	The following conditions are equivalent: 
	\begin{enumerate} 
		\item  $(X,R,\tau) \in \GCOTS_{\medcirc}$.
	
		\item $\lambda_{R} \subseteq \tau$ and there exists a base $\mathcal{B}$ of the topology $\tau$, where every member $O \in \mathcal{B}$ is $R$-convex (in the sense of Definition \ref{d:c-convex}). 
		
		\item $(X,R,\tau)$ is bi-embedded into a  $(Y,R_Y,\t_{R_Y}) \in \COTS$ as a closed subset. 
	\end{enumerate} 
\end{prop}
\begin{proof} $(3) \Longrightarrow (1)$ is trivial and $(1) \Longrightarrow (2)$ is straightforward (as in the case of LOTS) by observing that if $(X,R)$ is a c-ordered subset of a c-ordered set $(Y,R_Y)$ then $A \cap X$ is $R$-convex in $X$  for every $R_Y$-convex subset $A \subseteq Y$. In particular, it is true for $\t_{R_Y}$-open intervals $A \subset Y$.  
	
	The proof of the implication $(2) \Longrightarrow (3)$ is  based on a "circular version" of the proof from Remark \ref{r:constr}. In this case we consider the lexicographic \textbf{circularly} ordered product $Y:=X \otimes_c \Z$. 
	We define the same subset $X^* \subset X \times \Z$, as before but here $X^-, X^+$ are defined as in Definition \ref{d:c-sing}. This inclusion respects the c-order. 
	As in the linear construction of Remark~\ref{r:constr}, 
	the interval topology of $X^*$ induces precisely $\tau$
	on $X\times\{0\}$.

	Every point $u \in X^* \setminus X$ is isolated in $X^*$, where $X$ is identified with $X \times \{0\}$. This immediately implies	that $X$ is closed in $X^*$. For example, $(x,-1)$ is isolated in $X^*$ because the circular interval $((x,-2),(x,0))_{\circ}$ is just the singleton $\{(x,-1)\}$. 
\end{proof}

\begin{remark}\label{r:OmissionOfSorin}
	The concept of a GCO space was introduced in \cite[Section 2.2]{Me-OrdSem},
	after discussing convexity in circular orders.
	Recently, G.B. Sorin \cite{Sorin24} has also employed this definition (and the term GCO)
	in the equivalent formulation via convex sets.
	For the convenience of the reader, note that the definition of GCO spaces already
	appeared in \cite{Me-OrdSem}. We also note that \cite{Sorin24} cites another result
	from \cite{Me-OrdSem}, namely Theorem~\ref{t:c-comp} of the present work.
\end{remark}

Every circularly ordered set $(X,R)$ becomes a GCOTS under discrete topology. For example,   $(\Q_0,\circ,\t_{discr})$ in Fact \ref{f:GM} is a GCOTS.    

Recall that the classical Sorgenfrey space $\R_s$ is a GLOTS but not a LOTS. For every subset $Y$ of $\R$ we denote by $Y_s$ the corresponding subspace of the Sorgenfrey line.   
Let $(X,R)$ be a circularly ordered set with its interval  topology $\lambda_{\medcirc}$. In analogy with the classical Sorgenfrey space, one may define the circular Sorgenfrey topology generated by the base of all "semi-open" intervals of the type $[a,b)_R$. The corresponding  topology $\t_{s}$ is stronger than the interval topology $\lambda_{\medcirc}$. In fact, by Proposition \ref{p:GCO-eq} it is easy to check that $(X,R,\t_{s}) \in \GCOTS$. In particular, it is true for the circle $X=\T$. The corresponding  Sorgenfrey circle we denote by $\T_s$. It is not hard to show that the following topological spaces are homeomorphic: $\R_s \cong [0,2)_s \cong [0,1) \sqcup [1,2) \cong \T_s$. 
So, $\T_s$ (being homeomorphic to $\R_s$) is not metrizable but has a $G_{\delta}$ diagonal.  
Therefore, by Corollary \ref{c:normality}.2, $\T_s$ is not COTS.  
It is well known that $\R_s \in $ GLOTS. 
So, $\T_s \in \GLOTS \setminus \COTS$.  
In fact, one may show that this is not an isolated phenomenon but a general principle, as Proposition \ref{p:NoGCOTS}(b) demonstrates.
 
\begin{lem} \label{l:easyPropGCO} \ 
	\begin{enumerate}
		\item $\LOTS$ $\subsetneq$ $\GLOTS$ and $\COTS$ $\subsetneq$ $\GCOTS$. 

		\item \rm{comp-GLOTS}=\rm{comp-LOTS} $\subsetneq$ \rm{comp-GCOTS}=\rm{comp-COTS}. 
		\item Let $X$ be a $\COTS$. Suppose that $Y \subset X$ is a subset such that $Y \neq X$. Then $Y$ is a $\GLOTS$.  
	\end{enumerate}
	\begin{proof} 
		(1) LOTS $\neq$ GLOTS. The Sorgenfrey line $\R_s$ is a well known counterexample. 
		
		COTS $\neq$ $\GCOTS$. This time the Sorgenfrey circle $\T_s$ is a GCOTS but not a COTS.
	
		(2) The equalities follow by compactness of $\tau$ and Hausdorfness of interval topologies. For the inclusion apply Proposition \ref{inclusion}. For the inequality, take the circle $\T$.   
		
		(3) 	Take $c \in X \setminus Y$ and consider 
		$q \colon X(c)=[c^-,c^+] \to X$. 
		The induced map $q \colon q^{-1}(Y) \to Y$ 
		is a homeomorphism by Corollary \ref{cover}(iii). 
		Hence, $Y$ is embedded into the LOTS $[c^-,c^+]$.  	 	
	\end{proof}	
\end{lem}

Lemma \ref{l:easyPropGCO}.3 implies the following rigidity property: if $\T \subseteq Y$ and $Y$ is a COTS then $\T=Y$.    
   
While $\operatorname{GCOTS}_{\medcirc}$ and $\operatorname{GLOTS}_{\curlyeqprec}$ are essentially different classes, on the topological level (after forgetful assignment) $\GLOTS$ and $\GCOTS$ are very close. Already  Lemma \ref{l:easyPropGCO}.3 and Proposition \ref{p:NoGCOTS}(b) confirm this.    
Even more drastically, below in Corollary \ref{c:OutComp} we show that $\GLOTS \setminus Comp=\GCOTS \setminus Comp$. 

\begin{prop}\label{p:NoGCOTS} \ 
	\rm{(a)} $\GLOTS \subsetneq \GCOTS$; \rm{(b)}   
	\(
	\GCOTS\setminus \COTS = \GLOTS\setminus \COTS.
	\)
\end{prop}

\begin{proof} (a) 
	If $X$ is a GLOTS then it is a subspace of a LOTS. It is well known that every LOTS, in turn, is a subspace of a compact LOTS which is COTS by Proposition  \ref{inclusion}.
	Hence, $\GLOTS \subset \GCOTS$. 
	This inclusion is proper because the circle is a counterexample.   

	(b) Now it is enough to observe that $\GCOTS\setminus \COTS
	 \subseteq \GLOTS$ by Lemma \ref{l:easyPropGCO}.3.  	
\end{proof}

Let $(X,\leq_X,\t)$ be a GLOTS and $(Y,\leq_Y, \t_{\leq_Y})$  be a LOTS. A bi-embedding (i.e. a topological LOP embedding) $i \colon X \to Y$ is said to be an \textit{extension}. If, in addition,  $i(X)$ is $\t_{\leq_Y}$-dense in $Y$, then it is a \textit{d-extension}, \cite{MK}.  
A compact linearly ordered d-extension is a \textit{linearly ordered compactification} which is proper. 
Similar definitions make sense for GCOTS and COTS (see \cite{Sorin24}).

It is well known (Miwa-Kemoto \cite{MK}) that for every GLOTS $(X,\t,\leq)$ there exists a linear ``\textbf{minimal} d-extension" (we prefer the term  ``\textbf{least}" instead of ``minimal") 
$$
\theta \colon (X,\t,\leq) \to (\widetilde{X},\widetilde{\tau},\leq^{\sim}), 
$$
where $(\widetilde{X},\widetilde{\tau},\leq^{\sim})$ is a LOTS, $\widetilde{\tau}=\lambda_{\leq^{\sim}}$, $\theta$ is a LOP and topological \textit{dense} embedding 
such that for every linearly ordered $d$-extension
$f\colon X\to Y$ there exists a unique LOP topological
embedding $\widetilde f\colon\widetilde X\to Y$
satisfying $\widetilde f\circ\theta=f$.

The following is a natural circular analog.

\begin{lem}[least d-extension] \label{l:d-ext}  \cite{Sorin24} 
	For every $(X,R,\t) \in \GCOTS_{\medcirc}$ there exists a circular least  d-extension 
	$$
	\theta \colon (X,R,\t) \to 
	(\widetilde{X}, \widetilde{R},\widetilde{\tau}),
	$$
	where $(\widetilde{X}, \widetilde{R},\widetilde{\tau}) \in \COTS$, $\t_{\widetilde{R}}=\widetilde{\tau}$, $\theta$ is a COP and topological dense embedding such that for every circularly ordered d-extension $f \colon X \to Y$ there exists a (unique) COP topological embedding of $\widetilde{X}$ into $Y$ which extends $f$.	
\end{lem}
\begin{proof} Sketch: 
	The proof is similar to the linear case \cite{MK}, 
	but uses the \textit{circular} lexicographic product $X \otimes_c \{-1,0,1\}$. Define 
	$$
	\widetilde{X}:=(X\times\{0\})
	\;\cup\; \{(x,-1): x \in X^- \}
	\;\cup\; \{(x,1): x \in X^+ \},  
	$$ 	
	where $X^-, X^+$ are defined as in Definition \ref{d:c-sing} and $\theta(x)=(x,0)$. 
\end{proof}

\subsection*{Group actions on GLOTS}  
\label{s:ActGLOTS}

\begin{thm}[Separate $\Rightarrow$ joint continuity]  \label{t:COTS_joint_cont}
	Let $(X,R,\tau)$ be a $\GCOTS$.
	Suppose $\pi \colon G\times X\to X$ is a separately continuous COP action of a \textbf{semitopological} group $G$.
	Then the action map $\pi$ is jointly continuous. A similar result remains true for $(X,\leq,\tau) \in \GLOTS$ and LOP actions.
\end{thm}
\begin{proof}
	Fixed translations in $G$ and the action maps of fixed
	elements are homeomorphisms, so it suffices to prove
	joint continuity at $(e,x_0)$.
	If $x_0$ is isolated, continuity of its orbit map
	gives a neighborhood $V$ of $e$ with
	$Vx_0=\{x_0\}$, and we take $W=\{x_0\}$.
	
	Otherwise, let $O$ be a neighborhood of $x_0$.
	Choose an intrinsic circular interval $(p,q)_R$
	containing $x_0$, and a $\tau$-open convex neighborhood
	$U$ with $x_0\in U\subseteq O\cap(p,q)_R$.
	Use the linear cut order $\le_p$, in which $U$
	is convex.
	
	Choose $s,t\in U$ with $s<_p t$ such that
	$[s,t]_R$ contains an open neighborhood $W$ of $x_0$.
	Indeed, choose $s<_p x_0$ whenever possible, and
	otherwise put $s=x_0$; choose $t$ similarly on the
	other side.
	If $s<_p x_0<_p t$, take $W=(s,t)_R$.
	If $s=x_0=\min U$, take $W=U\cap(p,t)_R$;
	the case $t=x_0=\max U$ is symmetric.
	
	Choose $\tau$-open neighborhoods $U_s,U_t\subseteq U$
	of $s,t$, respectively, such that $u<_p v$
	whenever $u\in U_s$ and $v\in U_t$.
	By continuity of the orbit maps, there is a
	neighborhood $V$ of $e$ such that
	$Vs\subseteq U_s$ and $Vt\subseteq U_t$.
	For every $g\in V$, we therefore have $gs<_p gt$.
	Since $g$ preserves the circular order and $U$
	is linearly convex,
	$gW\subseteq g([s,t]_R)=[gs,gt]_R\subseteq U\subseteq O$.
	This proves joint continuity.
	
	For a $\GLOTS$, the same argument applies directly
	in its given linear order.
\end{proof}

It is well known (see, for example, \cite{Arens}) that for  every locally compact space $K$ and every subgroup $G$ of $H(K)$ there is a coarsest topology that makes the natural action $G \times K \to K$ continuous. It is the \textbf{g-topology of Arens} \cite{Arens} which for compact $K$ coincides with   the compact-open topology.      
Therefore, Theorem \ref{t:COTS_joint_cont} directly leads to the following result (strengthened below in Theorem \ref{t:myCOTS-G-comp}).  

\begin{cor} \label{c:Arens} 
	Let $(X,R,\t)$ be a locally compact $\GCOTS$. Then the  pointwise topology $\sigma^p$ on every subgroup  $G \subseteq H_+(X,R)$ coincides with the g-topology of Arens (compact-open topology, if $X$ is compact). Similarly, for $\GLOTS$.  
\end{cor}

\begin{remark} \label{r:AboutArens}  Every compact GLOTS is a LOTS. So, for compact LOTS $(X,\leq,\lambda_{\leq})$  Corollary \ref{c:Arens} gives a result of B. Sorin \cite[Theorem 2]{Sorin-B-22}; cf. Remark \ref{r:p-topologies} below. 
\end{remark}

\begin{defin} \label{d:EventFixed} 
	Let $G \times X \to X$ be an action of a semitopological group $G$ on $X$. 
	We say that:
	\begin{enumerate}
		\item a point $p \in X$ is \textit{eventually $G$-fixed} if there exists $V \in \mathcal{N}(e)$ such that $gp=p$ for every $g \in V$. Equivalently: if the stabilizer subgroup $\operatorname{St}(p)$ is open (clopen).   
		\item a subset $B \subseteq X$ \textit{eventually is in} $C \subseteq X$ if there exists $V \in \mathcal{N}(e)$ such that $gB \subseteq C$ for every $g \in V$.
	\end{enumerate}   
\end{defin}

Every eventually $G$-fixed point $p$ eventually belongs to every set $B \subseteq X$ such that $p \in B$. 

\begin{lem} \label{l:eventCIRCULAR} 
	Let $(X,R,\t)$ be a $\GCOTS$ and let $\pi \colon G \times X \to X$ be a (separately) continuous COP action of a  topological group $G$ on $(X,\t)$. Assume that $[b,c) \in \t$ or $(c,b] \in \t$. Then the point $b$ is eventually $G$-fixed. A similar result is true for $\GLOTS$ with LOP action.  
\end{lem}
\begin{proof}
	We treat only the circular case when $[b,c)$ is $\tau$-open; the proof of other cases is similar. 
	
	If $b$ is isolated in $(X,\tau)$ then $b$ is eventually $G$-fixed directly by continuity of the orbit map $\zeta_{b} \colon G \to X$. So assume $b$ is not $\t$-isolated. 
	Then $(b,c)$ is nonempty (otherwise, $[b,c)=\{b\}$).  
	Choose any $s \in (b,c)$. 
	Since $\t$ is Hausdorff and has a \textbf{convex} base, one may choose convex $\t$-open disjoint neighborhoods $O_1, O_2, O_3$ of $b,s,c$ respectively. Then by Lemma \ref{l:ConvexProprties}.3 we have $[O_1, O_2, O_3]$. 
	By continuity of the orbit maps $\zeta_{b}, \zeta_{s}$ and $\zeta_{c}$ there exists a \textbf{symmetric} $V\in\mathcal{N}(e)$ such that $Vb\subseteq O_1, Vs \subseteq O_2, Vc\subseteq O_3$. Then $V$ preserves the cyclic structure $[b,s,c]$. That is, 
	$$
	[g_1b, g_2s, g_3c] \qquad \text{for all } g_1, g_2, g_3 \in V.
	$$
	Since $[b,c)$ is $\t$-open and $\zeta_b \colon G \to X$ is continuous, we may suppose, in addition, that $gb \in [b,c)$ (hence, also $g^{-1}b \in [b,c)$) for every $g \in V=V^{-1}$. We claim that $gb=b$ for every $g \in V$. 
	
	Assuming the contrary, consider $g_0\in V$ with $g_0b \neq b$ (equivalently, $g_0^{-1}b \neq b$). Then necessarily $g_0b \in (b,s) \subset (b,c)$ and $[b,g_0b,s,g_0c]$. So, we have 
	$$
	[b,g_0b,g_0c].  
	$$
	Applying $g_0^{-1}$ and using that the action is COP we get
	$$
	[g_0^{-1}b,b,c].  
	$$
	This implies $g_0^{-1}b \notin [b,c)$, which is a contradiction. 
\end{proof}

\begin{remark}  
	In Lemma \ref{l:eventCIRCULAR} it is important that $G$ is a topological group (namely, the continuity of the inverse $G \to G, g \mapsto g^{-1}$). Indeed, Lemma \ref{l:eventCIRCULAR} fails for $G=X:=(\R_s,\t_s)$ the Sorgenfrey line with the usual left action. Then $G$ is a paratopological group (that is, the multiplication is continuous), the action is continuous, $(X,\leq,\t_s) \in \GLOTS_{\curlyeqprec}$ and all the intervals $[b,c)$ are open, yet there are no eventually $G$-fixed points.     
\end{remark}

Let $(X,R,\tau)$ be a generalized circularly ordered space (GCOTS), and let
\[
\theta \colon 
(X,R,\t) \to 
(\widetilde{X}, \widetilde{R},\widetilde{\tau}), \ \ \theta(x)=(x,0)
\]
denote the least $d$–extension (dense COP embedding into a COTS); see Lemma \ref{l:d-ext}. 

We identify $x$ with $(x,0)$ and, whenever the
corresponding points occur, write $x^-=(x,-1)$
and $x^+=(x,1)$.
Thus $x^-$ and $x^+$ are the immediate predecessor
and successor of $x$, respectively.
Here the superscripts indicate lexicographic positions;
in the single-split space $X(c)=[c^-,c^+]$ of
Definition~\ref{d:cover}, they label the least and
greatest endpoints. 

Here, using Lemma \ref{l:eventCIRCULAR}, we show that there exists a canonical COP extension of the continuous COP action to $\widetilde X$ which is also continuous. 

\begin{prop} \label{p:ActionLeastExtension}  
Let $(X,R,\tau)$ be a $\GCOTS$ and let
$\pi\colon G\times X\to X$ be a (separately) continuous
COP action of a topological group $G$ on $X$.	
 On the $\COTS$ $\widetilde X$ with $\t_{\widetilde{R}}=\widetilde{\tau}$
	define $\widetilde g \colon \widetilde X\to\widetilde X$ for $g\in G$ by
	\begin{equation} \label{e:permutes} 
	\widetilde g(x):=gx, \quad 	
	\widetilde g(x^\pm)=(gx)^\pm 	
	\end{equation}
	 Then:
	\begin{enumerate}
		\item Each $\widetilde g$ is a COP isomorphism (hence, a $\t_{\widetilde R}$-homeomorphism) of the $\COTS$ $(\widetilde X, \widetilde R)$;
		\item $\widetilde{\pi} \colon G \times \widetilde{X} \to \widetilde{X}$ is a $G$-action extending the action on $X$, and $\theta$ is $G$–equivariant: $\widetilde g\circ \theta=\theta\circ g$;
		\item The map $G\times \widetilde X\to \widetilde X$, $(g,z)\mapsto \widetilde g z$, is continuous. 	
	\end{enumerate}  
		(*) A similar result is true for $\GLOTS$. 
\end{prop}
\begin{proof}
	(1) Every $g\in G$ is a homeomorphism of $(X,\tau)$
	and a circular order automorphism of $(X,R)$, hence
	also a homeomorphism for $\lambda_R$.
	It sends half-open arcs to half-open arcs of the
	same type and preserves the defining conditions
	for the singular-point classes.
	Therefore $g(X^-)=X^-$ and $g(X^+)=X^+$, so the
	maps in~\eqref{e:permutes} are well defined.
	By the lexicographic construction of $\widetilde X$,
	the map $(x,k)\mapsto(gx,k)$ preserves
	$\widetilde R$ and has inverse $\widetilde{g^{-1}}$.
	Consequently, $\widetilde g$ is a circular order
	automorphism and a homeomorphism of $\widetilde X$.
	
	(2) The formulas in~\eqref{e:permutes} give
	$\widetilde e=\operatorname{id}_{\widetilde X}$,
	$\widetilde{gh}=\widetilde g\circ\widetilde h$,
	and $\widetilde g\circ\theta=\theta\circ g$.
	Hence they define an action extending the original
	one, and $\theta$ is equivariant.
	
	(3) Since every $\widetilde g$ is continuous,
	it remains to check continuity of the orbit maps.
	For $u=\theta(x)$, the orbit map is
	$g\mapsto\theta(gx)$ and is continuous.
	If $u=x^-$ or $u=x^+$ is an added point,
	the defining half-open interval at $x$ is
	$\tau$-open.
	By Lemma~\ref{l:eventCIRCULAR}, the stabilizer
	$H:=\operatorname{St}(x)$ is open in $G$.
	Formula~\eqref{e:permutes} shows that $H$ fixes $u$.
	The orbit map of $u$ is therefore constant on each
	left coset of $H$, and hence continuous.
	Thus the extended action is separately continuous,
	and Theorem~\ref{t:COTS_joint_cont} gives joint
	continuity.
	
	(*) 
	The same argument, using the defining rays, proves
	the GLOTS assertion.
\end{proof}

\section{Novák's regular completion and COP compactifications} 
\label{s:Novak} 

A (proper) compactification of a GCOTS $(X,R,\t)$ is defined similar to the theory of GLOTS. Namely, 
let $(K,R_K,\t_{R_K})$ be a compact COTS with its c-order $R_K$ and interval topology $\t_{R_K}$.  
A proper compactification $j \colon X \to K$ is said to be a \textit{compactification of GCO-spaces}, or, 
\textit{COP compactification} if $j$ is a bi-embedding, meaning that $j$ is both a topological and order embedding. 

The usual partial order between the compactifications (up to isomorphisms) inherits the partial order between GCO-compactifications. Namely, let  
$j_i \colon X \to K_i$ be two compactifications of $(X,R,\t)$ into $(K_i,R_{K_i},\t_{R_{K_i}})$. We say that $j_2$ \textit{dominates} $j_1$ if there exists a (necessarily, unique) continuous map $F \colon K_2 \to K_1$ such that 
$F \circ j_2=j_1$. In this case it is easy to see that we automatically get that $F$ is COP. This directly follows from the following general result.  

\begin{lem} \label{l:AutomaticCOP} 
	Let $F \colon (Y_1,R_1,\tau_1) \to (Y_2,R_2,\tau_2)$ be a continuous surjective map 
	between $\GCOTS$. Let $(X_1,R_{X_1})$ and $(X_2,R_{X_2})$ be circularly ordered sets, and let $f \colon X_1 \to X_2$ be a COP map. Assume that $j_1 \colon X_1 \to Y_1$ and $j_2 \colon X_2 \to Y_2$ are COP maps such that $j_1(X_1)$ is dense in $Y_1$ and $F \circ j_1 = j_2 \circ f$. 
	Then $F$ is also a COP map.  
\end{lem}

\begin{proof} 
	First, we prove the main case when $Y_2$ contains at least three elements. By Proposition \ref{p:Prop-c-ord}.1, it is sufficient to show condition (COP1).  
	Assume the contrary, and let $y_1, y_2, y_3 \in Y_1$ be such that $[y_1,y_2,y_3]$ holds in $Y_1$ but $[F(y_1),F(y_3),F(y_2)]$ holds in $Y_2$. 
	By Lemma \ref{l:c-is-Open}.1, there exist open neighborhoods $O_i$ of $y_i$ in $Y_1$ and $U_i$ of $F(y_i)$ in $Y_2$ (for $i \in \{1,2,3\}$) such that $[O_1,O_2,O_3]$ and $[U_1,U_3,U_2]$ hold. By the continuity of $F$, we may additionally assume that $F(O_i) \subseteq U_i$. 
	
	Since $j_1(X_1)$ is dense in $Y_1$, we may choose $x_1,x_2,x_3 \in X_1$ such that $j_1(x_i) \in O_i$. Then $[O_1,O_2,O_3]$ forces $[j_1(x_1),j_1(x_2),j_1(x_3)]$. 
	Because $F(j_1(x_i)) \in U_i$, the relation $[U_1,U_3,U_2]$ forces $[F(j_1(x_1)),F(j_1(x_3)),F(j_1(x_2))]$. Using the commutativity $F \circ j_1 = j_2 \circ f$, this translates to:
	\[
	[j_2(f(x_1)), j_2(f(x_3)), j_2(f(x_2))]. 
	\]  
	However, (COP1) and the relation $[j_1(x_1), j_1(x_2), j_1(x_3)]$ implies $[x_1,x_2,x_3]$ in $X_1$. Since $j_2 \circ f$ is a COP map, it preserves this cycle, yielding $[j_2(f(x_1)), j_2(f(x_2)), j_2(f(x_3))]$. This contradicts the relation established above, completing the proof for this case. 
	
	If $|Y_2|=1$, the assertion is trivial.
	Let $Y_2=\{y_1,y_2\}$ and suppose that $F$ fails
	{\rm(COP2)}. Then, say, $A:=F^{-1}(y_1)$ is not convex.
	Choose an injective cycle $[a_1,a_2,a_3,a_4]$ in $Y_1$
	with $a_1,a_3\in A$ and $a_2,a_4\notin A$.
	Since the two fibers of $F$ are clopen, Lemma~\ref{l:c-is-Open}(1)
	gives pairwise disjoint ordered open neighborhoods $O_i$
	of $a_i$, with $O_1,O_3\subset A$ and
	$O_2,O_4\subset F^{-1}(y_2)$.
	By density choose $x_i\in X_1$ with $j_1(x_i)\in O_i$.
	Then $[x_1,x_2,x_3,x_4]$ is a cycle.
	But the COP map $j_2\circ f=F\circ j_1$ takes alternately
	the values $y_1,y_2,y_1,y_2$ on this cycle, contradicting
	{\rm(COP2)}. 
\end{proof}

Recall that by (complete=compact) Theorem \ref{t:c-comp} a COTS $(X,R)$ is compact in the interval topology $\lambda_R$ if and only if $(X,R)$ is complete (that is, iff it has no gaps). 

It is important to take into account that if $j \colon (X,R_X) \to (Y,R_Y)$ is an embedding of circularly ordered sets with $R_Y$ complete (equivalently, with compact $(Y,\lambda_{R_Y})$), then this ``completion" is not necessarily a compactification of the topological space $(X,\lambda_{R_X})$ because 
$j \colon (X,\lambda_{R_X}) \to (Y,\lambda_{R_Y})$ is not necessarily continuous. Indeed, the concept of GCOTS  demonstrates this. 

For a concrete example, consider the discrete circled rationals $(\Q_0,R,\t_{discr})$ and its COP $\t_{discr}$-topological dense embedding into the compact COTS $Y:=\rm{trip}(\T;\Q_0)$ (see Fact \ref{f:GM}.2). Then 
the dense COP inclusion $\Q_0 \hookrightarrow \rm{trip}(\T;\Q_0)$ is not a compactification of the COTS $(\Q_0,\lambda_{R})$ (but of the GCOTS $(\Q_0,\t_{discr})$).  
So, an algebraic completion of a c-ordered set need not be continuous
and hence it is not automatically a compactification.
This justifies the following careful definition and Theorem \ref{t:CompletionEmbeddingAndDensity}. 

\begin{defin} \label{d:TopCompletion} 
Let us say that a COP map $f \colon (X,R_X) \to  (Y,R_Y)$ is a \textit{topological completion} if $(Y,R_Y)$ is complete  and $f$ is a compactification with respect to the interval topologies. 	
\end{defin} 

Below in Theorem \ref{t:CompletionEmbeddingAndDensity} we show that \textit{Novak's regular completion} $\nu\colon X\to \Xcal_r$, defined by $x\mapsto \leq_x$ of every circularly ordered set is a (proper) compactification  hence a \textbf{topological completion}. 

\begin{defin}[Nov\'ak \cite{Novak-cuts}]  \label{d:Novak3blocks} 
	A \emph{regular cut} on $(X,R)$ is either a gap or a lower point-cut $\leq_x$ (for some $x\in X$).
	Denote by $\Xcal_r$ the set of all regular cuts in $(X,R)$. Circular order
	$\Rcal$ on $\Xcal_r$ is defined as follows: for distinct $L_1,L_2,L_3\in \Xcal_r$,
	\[
	[L_1,L_2,L_3]_{\Rcal}
	\ \Longleftrightarrow\
	\exists\ \text{a disjoint partition } X=A\sqcup B\sqcup D\ \text{into nonempty pieces such that}
	\]
	\[
	(X,L_1)=A\oplus B\oplus D,\qquad
	(X,L_2)=B\oplus D\oplus A,\qquad
	(X,L_3)=D\oplus A\oplus B,
	\]
	where $U\oplus V$ denotes the ordinal sum of linearly ordered sets. 
	Then $(\mathcal{X}_r, \mathcal{R})$ is complete by \cite[Theorem 5.6]{Novak-cuts}. 
	The map $\nu\colon X\to \Xcal_r$, defined by $z\mapsto \leq_z$, 
	is an embedding of circularly ordered sets by 
	\cite[Corollary 4.5]{Novak-cuts}: 
	$
	[a,b,c]_R\iff [\nu(a),\nu(b),\nu(c)]_{\Rcal}\quad(\text{for distinct } a,b,c \in X).
	$
\end{defin}

\begin{thm} \label{t:CompletionEmbeddingAndDensity}
	Let $(X,R)$ be a c-ordered set and $(\Xcal_r,\Rcal)$ its Novák regular completion.
	Let $\mathrm{int}_r$ be the interval topology of the circularly ordered set $(\Xcal_r,\Rcal)$.
	Then the canonical map
	\[
	\nu \colon (X,\lambda_{\medcirc})\longrightarrow (\Xcal_r,\mathrm{int}_r),\qquad x \mapsto {\leq_x},
	\]
	is a COP topological embedding, and $\nu(X)$ is dense in the compact $\COTS$ $(\Xcal_r,\mathrm{int}_r)$.
\end{thm} 
\begin{proof}
	The cases $|X|\leq2$ are immediate, so assume
	$|X|\geq3$.
	By Definition~\ref{d:Novak3blocks}, $\nu$ is an
	embedding of circularly ordered sets and
	$(\Xcal_r,\Rcal)$ is complete.
	Hence $(\Xcal_r,\mathrm{int}_r)$ is compact by
	Theorem~\ref{t:c-comp}.
	
	\emph{Density.}
	Let $(L_1,L_2)_{\Rcal}$ be nonempty and choose
	$M$ in this interval.
	By Definition~\ref{d:Novak3blocks}, there are
	nonempty blocks $A,B,D$ partitioning $X$ such that
	$(X,L_1)=A\oplus B\oplus D$,
	$(X,M)=B\oplus D\oplus A$, and
	$(X,L_2)=D\oplus A\oplus B$.
	Choose $z\in B$ and split $B$ at $z$ as
	$B=B_1\oplus B_2$, where $B_1$ consists of the
	points preceding $z$ and $z$ is the least element
	of $B_2$.
	Then the lower point-cut at $z$ is
	$\nu(z)=B_2\oplus D\oplus A\oplus B_1$.
	The nonempty blocks $A\cup B_1$, $B_2$ and $D$
	therefore witness
	$[L_1,\nu(z),L_2]_{\Rcal}$.
	Thus every nonempty open interval meets $\nu(X)$,
	proving density.
	
	\emph{Continuity.}
	The intervals with at least one endpoint in $\nu(X)$
	form a subbase of $\mathrm{int}_r$.
	Indeed, given $M\in(L_1,L_2)_{\Rcal}$, choose
	$c\in\nu(X)\setminus\{M\}$.
	If $c$ is an endpoint, the interval already belongs
	to this family.
	If $c\in(L_1,L_2)_{\Rcal}$, one of
	$(L_1,c)_{\Rcal}$ and $(c,L_2)_{\Rcal}$ contains $M$
	and is contained in $(L_1,L_2)_{\Rcal}$.
	Otherwise $c\in(L_2,L_1)_{\Rcal}$, and
	$(L_1,c)_{\Rcal}\cap(c,L_2)_{\Rcal}
	=(L_1,L_2)_{\Rcal}$.
	This proves the subbase assertion.
	
	It therefore suffices to check the preimages of
	$(L,\nu(b))_{\Rcal}$ and $(\nu(a),L)_{\Rcal}$,
	where $a,b\in X$ and $L\in\Xcal_r$.
	If $L$ is a lower point-cut, these preimages are
	circular open intervals in $X$, possibly empty,
	because $\nu$ is an order embedding.
	If $L=\gamma$ is a gap, the three-block definition
	gives
	$\nu^{-1}((\gamma,\nu(b))_{\Rcal})
	=\{x\in X:x<_\gamma b\}$ and
	$\nu^{-1}((\nu(a),\gamma)_{\Rcal})
	=\{x\in X:a<_\gamma x\}$.
	These are open rays for the cut order $\leq_\gamma$.
	Since this order has neither a minimum nor a maximum,
	Lemma~\ref{l:cWEAKER}(2) gives
	$\lambda_{\leq_\gamma}=\lambda_{\medcirc}$.
	Thus these preimages are open, and $\nu$ is continuous.
	
	\emph{Embedding.}
	For distinct $a,b\in X$, the order-embedding property
	gives
	$\nu((a,b)_R)
	=(\nu(a),\nu(b))_{\Rcal}\cap\nu(X)$.
	Hence $\nu$ is open onto its image.
	Together with injectivity and continuity, this proves
	that $\nu$ is a topological embedding.
\end{proof}


\begin{cor} \label{c:COTSisCOMPACTIFIABLE} 
Every $(Y,R_Y,\t) \in \GCOTS_{\medcirc}$  admits a proper topological COP compactification. 	
\end{cor}
\begin{proof}  
	Let  
$\theta:(Y,R_Y,\tau)\to(\widetilde Y,\widetilde R,\widetilde\tau)$
be the least $d$-extension from Lemma~\ref{l:d-ext}.
Thus $\theta$ is a dense COP topological embedding into a COTS.
By Theorem~\ref{t:CompletionEmbeddingAndDensity}, the Nov\'ak
completion
$\nu:\widetilde Y\to(\widetilde Y)_r$
is a proper COP compactification.
Hence $\nu\circ\theta:Y\to(\widetilde Y)_r$ is a proper
COP compactification.  
\end{proof}

Note that Corollary \ref{c:COTSisCOMPACTIFIABLE} is a part of results proved recently by G.B. Sorin \cite[Corollary 3.2]{Sorin24} using different methods. 

\begin{cor} \label{c:OutComp} 
a) $\GLOTS \setminus Comp=\GCOTS \setminus Comp$; 
b) $\GCOTS \setminus \GLOTS \subset comp\COTS$.  	
\end{cor}
\begin{proof}    
\rm{(a)} 	
$\subseteq$-\textbf{part} follows from the inclusion $\GLOTS \subseteq \GCOTS$ (Proposition \ref{p:NoGCOTS}(a)). 

$\supseteq$-\textbf{part}. 
Let $X \in \GCOTS$ and $X$ is not compact. Then by Corollary \ref{c:COTSisCOMPACTIFIABLE} there exists a proper compactification 
$j \colon X \to K$ where $K$ is a compact COTS. Since $X$ is not compact we have $j(X) \neq K$. Then $j(X)$ (hence, also $X$) is $\GLOTS$ by Lemma \ref{l:easyPropGCO}.3.  

(b) follows by (a) and the inclusion 
$\GLOTS \subset \GCOTS$ from Proposition \ref{p:NoGCOTS}(a).
\end{proof}

\begin{thm}[Minimality of Nov\'ak's regular completion]\label{t:NovakMinimality}
	Let $i_Y \colon (X,R) \to (Y,R_Y)$ be any proper circularly ordered compactification.
	Then there exists a unique continuous circular order–preserving (COP) surjection
	$
	\Psi \colon Y \to (\Xcal_r,\Rcal)
	$
	such that $\Psi \circ i_Y = \nu$.
\end{thm} 
\begin{proof}
	Identify $X$ with $i_Y(X)\subseteq Y$.
	The finite case is immediate, since density then
	implies $Y=X$ and $\Xcal_r=\nu(X)$.
	Hence assume that $X$ is infinite.
	
	\smallskip
	\noindent\emph{Definition of $\Psi$.}
	For $x\in X$, put $\Psi(x)=\nu(x)$.
	For $y\in Y\setminus X$, let $\le_y$ be the cut on
	$X$ defined by $u<_y v$ if and only if
	$[y,u,v]_{R_Y}$, for distinct $u,v\in X$.
	We show that $\le_y$ is a gap.
	
	For every $b\in X$, the open rays of $\le_y$ are
	$X\cap(y,b)_{R_Y}$ and $X\cap(b,y)_{R_Y}$.
	Since $X$ carries the subspace topology inherited
	from $Y$, we obtain $\lambda_{\le_y}\subseteq\lambda_R$.
	The reverse inclusion follows from
	Lemma~\ref{l:cWEAKER}(1).
	Thus the two topologies coincide, and
	Lemma~\ref{l:cWEAKER}(3) excludes the possibility
	that $\le_y$ has exactly one endpoint.
	Nor can it have both a minimum $m$ and a maximum $M$:
	then $y\in(M,m)_{R_Y}$ and
	$(M,m)_{R_Y}\cap X=\emptyset$, contradicting density.
	Therefore $\le_y$ is a gap.
	
	We may consequently define $\Psi(y)=\le_y$.
	In particular, $\Psi^{-1}(\{\nu(x)\})=\{x\}$ for
	every $x\in X$, since the added points map to gaps
	and $\nu$ is injective.
	
	\smallskip
	\noindent\emph{The interval identity.}
	Fix distinct $x_1,x_2\in X$ and suppose that
	$[x_1,y,x_2]_{R_Y}$.
	Put $A=[x_1,y)_{R_Y}\cap X$,
	$B=[y,x_2)_{R_Y}\cap X$, and
	$D=[x_2,x_1)_{R_Y}\cap X$.
	These sets partition $X$.
	The sets $A,D$ contain $x_1,x_2$, respectively.
	Also $B$ is nonempty: it contains $y$ when $y\in X$;
	otherwise its nonemptiness follows because $x_2$
	is not a least element of the gap $\le_y$.
	
	The three cuts $\nu(x_1)$, $\Psi(y)$, and $\nu(x_2)$
	order these blocks as $A\oplus B\oplus D$,
	$B\oplus D\oplus A$, and $D\oplus A\oplus B$,
	respectively.
	Definition~\ref{d:Novak3blocks} therefore gives
	$[\nu(x_1),\Psi(y),\nu(x_2)]_{\Rcal}$.
	Interchanging $x_1,x_2$ and using totality gives
	the converse whenever $y\notin\{x_1,x_2\}$.
	The endpoint cases follow from
	$\Psi^{-1}(\{\nu(x_i)\})=\{x_i\}$.
	Consequently,
	\[
	\Psi^{-1}\big((\nu(x_1),\nu(x_2))_{\Rcal}\big)
	=(x_1,x_2)_{R_Y}.
	\]
	
	\smallskip
	\noindent\emph{The required base and continuity.}
	We verify that intervals with endpoints in $\nu(X)$
	form a base of $\mathrm{int}_r$.
	Let $M\in(L_1,L_2)_{\Rcal}$.
	Choose the nonempty blocks $A,B,D$ witnessing this
	relation in Definition~\ref{d:Novak3blocks}, so that
	$(X,L_1)=A\oplus B\oplus D$,
	$(X,M)=B\oplus D\oplus A$, and
	$(X,L_2)=D\oplus A\oplus B$.
	
	Choose $a\in A$.
	Splitting $A$ at $a$ in this description gives
	$\nu(a)\in[L_1,M)_{\Rcal}$.
	If some $b\in B$ satisfies $\nu(b)\neq M$,
	splitting $B$ at $b$ similarly gives
	$\nu(b)\in(M,L_2)_{\Rcal}$.
	Otherwise, injectivity of $\nu$ implies that
	$B=\{b_0\}$ and $M=\nu(b_0)$.
	Then $(X,L_2)=D\oplus A\oplus\{b_0\}$ has a
	greatest element.
	Since $L_2$ is regular, it must therefore be a
	lower point-cut $\nu(b)$, with $b\in D$.
	In either case,
	$M\in(\nu(a),\nu(b))_{\Rcal}
	\subseteq(L_1,L_2)_{\Rcal}$.
	This proves the base assertion.
	The interval identity above now implies that
	$\Psi$ is continuous.
	
	\smallskip
	\noindent\emph{Surjectivity, COP, and uniqueness.}
	Since $Y$ is compact and $\Xcal_r$ is Hausdorff,
	$\Psi(Y)$ is closed.
	It contains the dense subset $\nu(X)$, so $\Psi$
	is surjective.
	Lemma~\ref{l:AutomaticCOP} now shows that $\Psi$
	is COP.
	Finally, two continuous maps from $Y$ to the
	Hausdorff space $\Xcal_r$ agreeing on the dense
	subset $X$ must coincide.
	Thus $\Psi$ is the unique continuous extension
	satisfying $\Psi\circ i_Y=\nu$.
\end{proof}

\begin{remark} \label{r:bingo}
	Theorems \ref{t:CompletionEmbeddingAndDensity} and \ref{t:NovakMinimality} demonstrate that Nov\'{a}k's regular completion is the \emph{least} circularly ordered compactification of $(X,R)$, fully analogous to the linear theory.
	For a LOTS $(X,\leq)$, there are two classical approaches to the least orderable compactification: Kaufman’s order-topological construction via closed ideals \cite{Kaufman1967} and the Dedekind--MacNeille completion \cite{BM}.
	As noted in \cite[p.~580]{BM}, they coincide. Nov\'{a}k’s regular cuts (gaps and lower point-cuts) serve as the exact circular counterparts of Dedekind cuts (gaps and principal cuts), satisfying the same universal property in the circular setting.
\end{remark}

\subsection*{Additional remarks on COP compactifications} 

It is well known that any compactification of a topological space $X$ can be described as a diagonal product $j \colon X \to Y=cl(j(X)) \subset [0,1]^I$ of continuous functions 
$F:=\{f_i: X \to [0,1]\}_{i \in I}$. 

Below we show that the circle $\T$ may play a similar role for GCO-compactifications. 

For every $(X,\circ,\t) \in \GCOTS$ denote by $C_+(X,\T)$ the class of all $\t$-continuous $\circ$-COP maps into the circle $\T$. 
 
\begin{prop} \label{p:CompToCircle}
	Let $(X,\circ,\t)$ be a $\GCOTS$. Then 
	\begin{enumerate}
		\item $C_+(X,\T)$ is a $3$-point separating family. 
		\item There exists a $\tau$-topological embedding
		$\nu \colon X\to\T^I$, where $\T^I$ carries the product topology,
		such that $Y:=\overline{\nu(X)}$ admits a circular order
		compatible with its subspace topology and making
		$\nu \colon X\to Y$ a COP embedding. In particular, $Y$ is a
		compact COTS.
	\end{enumerate} 
\end{prop}
\begin{proof}
	By Corollary~\ref{c:COTSisCOMPACTIFIABLE}, choose
	a proper COP compactification $j:X\hookrightarrow K$,
	where $K$ is a compact COTS.
	
	(1) We first show that $C_+(K,\T)$ separates triples.
	Fix distinct $c,u,v\in K$.
	The single-split space $K(c)=[c^-,c^+]$ is a compact
	LOTS, and $q_c:K(c)\to K$ is a quotient map
	by Corollary~\ref{cover}(i)--(ii).
	After interchanging $u,v$ if necessary, we have
	$c^-<_c u<_c v<_c c^+$.
	On the closed set $\{c^-,u,v,c^+\}$ prescribe the
	values $0,1/3,2/3,1$, respectively.
	By Lemma~\ref{l:Nachbin}, this assignment extends
	to a continuous LOP function $h:K(c)\to[0,1]$.
	
	Let $p:[0,1]\to\T$ be the quotient map
	$p(t)=e^{2\pi it}$.
	Since $p(h(c^-))=p(h(c^+))=1$, the map $p\circ h$
	descends to a continuous map $f:K\to\T$ satisfying
	$f\circ q_c=p\circ h$.
	The values $f(c),f(u),f(v)$ are pairwise distinct.
	
	To verify {\rm(COP1)}, take a positively ordered
	triple in $K$ whose $f$-images are distinct.
	After a cyclic relabeling, its representatives
	in $[c^-,c^+)$ have the form $s_1<_c s_2<_c s_3$.
	Monotonicity of $h$ and distinctness of the
	$f$-images give $h(s_1)<h(s_2)<h(s_3)$,
	and the first and last values cannot be $0$ and $1$
	simultaneously.
	Their images under $p$ therefore have the same
	circular orientation.
	Since $f$ has at least three image points,
	Proposition~\ref{p:Prop-c-ord}(1) implies that $f$
	is COP.
	Thus $C_+(K,\T)$ separates triples, and restriction
	along $j$ proves~(1).
	
	(2) Put $I=C_+(K,\T)$.
	By~(1), this family separates points when $|K|\geq3$.
	For $|K|\leq2$, point separation follows from the
	fact that any injective map $K\to\T$ is continuous
	and COP.
	Hence the evaluation map
	$e:K\to\T^I$, $e(k)=(f(k))_{f\in I}$,
	is a continuous injection and therefore a
	topological embedding.
	
	Give $e(K)$ the circular order transported from $K$
	and put $\nu=e\circ j$.
	Then $\nu$ is a topological COP embedding.
	Since $j(X)$ is dense in $K$ and $e(K)$ is compact,
	$\overline{\nu(X)}=e(K)$.
	Thus this closure is a compact COTS with the
	required properties.
\end{proof}

Recall that for GO-spaces there exists a linearly ordered version of Stone-Čech compactification, which often is called \textit{Nachbin compactification} (for more details see \cite{Blatter1975,KentRichmond1988,BM,HK}).  

\begin{f}[Nachbin Compactification for GLOTS] \label{RepLem} 
	For every $\GLOTS$ $(X,\leq,\tau)$ there exists the greatest (maximal) proper GO-compactification $\mu \colon X \to \beta_{\prec} X$. For any $f \in C_+(X,[0,1])$ there exists $F \in C_+(\beta_{\prec} X,[0,1])$ such that $f=F \circ \mu$. 
\end{f}  

As expected, one may give a GCOTS analog of this fact but replacing $[0,1]$ by $\T$. 

\begin{prop}[Nachbin Compactification for GCOTS]  \label{p:maxCORDcomp}
	For every $\GCOTS$ $(X,R,\tau)$ there exists a greatest (maximal) proper COP compactification $\mu \colon X \to \beta_{R} X$.
	Moreover, for any $f \in C_+(X,\T)$, there exists a unique $F \in C_+(\beta_{R} X,\T)$ such that $f=F \circ \mu$.
\end{prop}  
\begin{proof}
	The empty case is immediate, so assume that $X\ne\varnothing$.
	
	\emph{Step 1: Construction of a proper COP compactification.}
	By Corollary~\ref{c:COTSisCOMPACTIFIABLE}, choose a proper
	COP compactification $X\subset K_0$, and fix $c\in X$.
	Consider the single-split map
	$q_0:K_0(c)=[c^-,c^+]\to K_0$, and put $L=q_0^{-1}(X)$.
	By Corollary~\ref{cover}, $K_0(c)$ is a compact LOTS and
	$q_0$ is closed. Thus $L$ is a GLOTS whose minimum and
	maximum are $c^-$ and $c^+$.
	
	Let $B=\beta_{\prec}L$ be the greatest proper GO
	compactification supplied by Fact~\ref{RepLem}.
	Since $L$ is dense in $B$, its endpoints remain the
	endpoints of $B$.
	Identify them by the map
	$q:B\to K:=B/(c^-\sim c^+)$.
	Give $K$ the circular order induced by the linear order
	of $B\setminus\{c^+\}$, representing the identified
	endpoint by $c^-$.
	The inverse images of circular open intervals under $q$
	are open intervals or unions of open end rays in $B$.
	Hence $q$ is continuous, and compactness of $B$ implies
	that $K$ is a compact COTS and $q$ is a closed quotient map.
	
	The inclusion $L\subset B$ induces a dense circular
	order embedding $\mu:X\to K$ satisfying
	$\mu\circ(q_0|_L)=q|_L$.
	It is also a topological embedding.
	Indeed, $L$ is saturated for both $q_0$ and $q$.
	A closed map restricted to a saturated subspace is
	closed onto its image: if $C$ is closed in $L$, write
	$C=D\cap L$ with $D$ closed in $B$; then
	$q(C)=q(D)\cap q(L)$. 
	Consequently, $q_0|_L:L\to X$ and $q|_L:L\to q(L)$
	are closed quotient maps with the same fibers.
	They therefore induce the same quotient topology,
	proving the assertion.
	
	\smallskip
	\emph{Step 2: Extension of circle-valued COP maps.}
	Let $f\in C_+(X,\T)$.
	Constant maps extend trivially.
	Otherwise, put $y_0=f(c)$, define
	$p:[0,1]\to\T$ by $p(t)=y_0e^{2\pi it}$, and let
	$\varphi=f\circ(q_0|_L)$.
	By (COP2) and Lemma~\ref{l:ConvexProprties}(1),
	$X\setminus f^{-1}(y_0)$ is circularly convex.
	Its inverse image
	$A:=\varphi^{-1}(\T\setminus\{y_0\})$
	is a nonempty open linearly convex subset of $L$.
	Indeed, for $u<v$ in $A$, the reverse circular arc
	from $q_0(v)$ to $q_0(u)$ contains $c$.
	Convexity therefore forces the forward arc, and hence
	the whole linear interval $[u,v]$, to lie in $A$.
	
	Write $L\setminus A=I_0\sqcup I_1$, where $I_0$ and $I_1$
	are respectively the initial and terminal parts.
	In particular, $c^-\in I_0$ and $c^+\in I_1$.
	Define $h:L\to[0,1]$ by taking $h=0$ on $I_0$,
	$h=1$ on $I_1$, and
	$h=(p|_{(0,1)})^{-1}\circ\varphi$ on $A$.
	Then $p\circ h=\varphi$, $h(c^-)=0$, and $h(c^+)=1$.
	If $u<v$ belong to $A$ and $\varphi(u)\ne\varphi(v)$,
	condition (COP1), applied to $[c,q_0(u),q_0(v)]$,
	gives $[y_0,\varphi(u),\varphi(v)]$ and hence
	$h(u)<h(v)$.
	It follows that $h$ is nondecreasing on $L$.
	
	The map $h$ is continuous on the open set $A$.
	Fix $a\in A$.
	On the open ray $L_{<a}$, the values of $h$ lie in
	$[0,h(a)]$, and
	$h=(p|_{[0,h(a)]})^{-1}\circ\varphi$.
	Similarly, on $L_{>a}$ one has
	$h=(p|_{[h(a),1]})^{-1}\circ\varphi$.
	Since $0<h(a)<1$, both restrictions of $p$ are
	homeomorphisms onto their images.
	These rays together with $A$ cover $L$, so $h$ is continuous.
	
	By Fact~\ref{RepLem}, $h$ extends to a continuous
	nondecreasing map $H:B\to[0,1]$.
	Since $H(c^-)=0$ and $H(c^+)=1$, the map $p\circ H$
	descends to a continuous map $F:K\to\T$ satisfying
	$F\circ\mu=f$.
	
	We verify that $F$ is COP.
	For (COP1), take a positive triple in $K$ with pairwise
	distinct $F$-images, and cyclically relabel its
	representatives in $B\setminus\{c^+\}$ as
	$s_1<s_2<s_3$.
	Then $H(s_1)<H(s_2)<H(s_3)$, and the extreme values
	cannot be $0$ and $1$ simultaneously.
	Thus
	$[F(q(s_1)),F(q(s_2)),F(q(s_3))]$ holds in $\T$.
	For (COP2), every fiber over $y\ne y_0$ is the image
	under $q$ of a linearly convex level set of $H$
	avoiding the endpoints, and is therefore circularly convex.
	The complement of the $y_0$-fiber is
	$q(H^{-1}((0,1)))$, which is likewise circularly convex.
	Lemma~\ref{l:ConvexProprties}(1) therefore gives convexity
	of the $y_0$-fiber as well.
	Hence $F$ is COP.
	Its uniqueness follows from the density of $\mu(X)$.
	
	\smallskip
	\emph{Step 3: The greatest compactification.}
	Let $\sigma \colon X\hookrightarrow K'$ be any proper COP
	compactification, and put $I=C_+(K',\T)$.
	By Proposition~\ref{p:CompToCircle}, the evaluation map
	$e \colon K'\to\T^I$, $e(k')=(g(k'))_{g\in I}$,
	is a topological embedding.
	Taking the diagonal of the extensions of
	$g\circ\sigma$, $g\in I$, supplied by Step 2 gives
	a continuous map $D \colon K\to\T^I$ satisfying
	$D\circ\mu=e\circ\sigma$.
	Compactness and density give
	$D(K)=\overline{D(\mu(X))}
	=\overline{e(\sigma(X))}=e(K')$.
	Thus $\Phi=e^{-1}\circ D \colon K\to K'$ is continuous and
	onto, with $\Phi\circ\mu=\sigma$.
	By Lemma~\ref{l:AutomaticCOP}, $\Phi$ is COP.
	Hence $\mu$ is the greatest proper COP compactification,
	and we may set $\beta_R X=K$.
	
	\smallskip
	\emph{Step 4: The full evaluation construction.}
	For each $f\in C_+(X,\T)$, denote its unique extension
	by $F_f$, and define
	$E:K\to\T^{C_+(X,\T)}$ by $E(z)=(F_f(z))_f$.
	Every $g\in C_+(K,\T)$ equals $F_{g\circ\mu}$,
	by uniqueness.
	Since these maps separate points of $K$ by
	Proposition~\ref{p:CompToCircle}, $E$ is a continuous
	injection and hence a topological embedding.
	Finally, compactness and density give
	$E(K)=\overline{\{(f(x))_{f\in C_+(X,\T)}:x\in X\}}$.
	Endowed with the circular order transported from $K$,
	this closure realizes $\beta_R X$ as the compactification
	defined by the full evaluation map.
\end{proof}

\section{Convex uniform structures on GCOTS}
\label{s:ConvexUniform} 

Let $(X,\tau,\leq)$ be a GLOTS. A (covering) $\t$-compatible uniformity $\mathcal{U}$ on $X$ is said to be a \textit{GO-uniformity} in the sense of D. Buhagiar and T. Miwa \cite{BuhagiarMiwa} if $\mathcal{U}$ admits a uniform base $\B$ such that every $\a \in \B$ is a $\leq$-convex cover, meaning that every member of $\a$ is convex.  Convex uniformities for LOTS were defined and studied (including the completions) by A.A. Borubaev \cite{Borubaev}. 

We give a similar definition for GCOTS. 

\begin{defin} \label{d:convexUNIF} 
	Let $(X,R,\tau)$ be a GCOTS. A $\t$-compatible  (covering) uniformity $\mathcal{U}$ on $X$ is said to be a \textit{GCO-uniformity} (or, simply a \textit{convex uniformity}) if $\mathcal{U}$ admits a uniform base $\B$ such that every $\beta \in \B$ is a $R$-convex cover. 	
\end{defin} 

\begin{lem} \label{l:IntConvGCO} 
 For every convex subset $C \subset X$ of a $\GCOTS$ $(X,R,\tau)$ its $\tau$-interior and $\t$-closure both are $R$-convex. 	
\end{lem}
\begin{proof}    
Suppose \(\overline C\) is not convex. There are then four points such that 
	\[
	[a,b,c,d],
	\qquad
	a,c\in\overline C,\quad b,d\notin\overline C.
	\]
	Choose small ordered neighborhoods of \(a,c\), and points \(a',c'\in C\) in them. Then
	$
	[a',b,c',d],
	$
	contradicting convexity of \(C\).
	
	The interior assertion follows from
	$
	\operatorname{int}C
	=
	X\setminus\overline{X\setminus C}
	$
	and the fact that complements of circularly convex sets are convex by Lemma \ref{l:ConvexProprties}(1). 
\end{proof}

\begin{cor} \label{l:openFINITE} 
Every convex GCO-uniformity $\mathcal U$ contains a
uniform base $\mathcal B$ such that each cover
$\alpha\in\mathcal B$ consists of open convex subsets. If the uniformity $\mathcal{U}$ is precompact then we can assume that every cover $\a \in \mathcal{B}$ is finite and open. 		
\end{cor}
\begin{proof}
Combine Lemmas \ref{l:intBase} and \ref{l:IntConvGCO}.  	
\end{proof}

The class of GCO-uniformities is closed under subspaces
and uniform completions.

\begin{prop} \label{p:ORDERING-BM} 
	Let $\mathcal{U}$ be a GCO-uniformity on a $\GCOTS$ $(X,\tau,R)$. Consider the \textit{uniform completion} $i \colon (X,\mathcal{U}) \to (\widehat{X},\widehat{\mathcal{U}})$.
	Then 
	\begin{enumerate}
		\item $\widehat{\mathcal{U}}$ is a GCO-uniformity on $\widehat{X}$ with respect to a (uniquely defined) circular order $\widehat{R}$ such that $i$ is a COP embedding.
		\item Let $\pi \colon G \times X \to X$ be a (continuous) $\mathcal{U}$-equiuniform COP action of a (semi) topological group $G$ on $(X,\mathcal{U})$. Then the canonically extended 
		$G$-action on the completion
		$\widehat{\pi} \colon G \times \widehat{X} \to \widehat{X}$  
		is $\widehat{\mathcal{U}}$-equiuniform,  continuous and COP.    
	\end{enumerate} 
\end{prop}  
\begin{proof}
	(1) The assertion is immediate when $|X|\leq 2$, since
	then $\widehat X=X$. Hence assume that $|X|\geq 3$.
	Write $\widehat\tau$ for the completion topology and
	identify $X$ with its dense image under $i$.
	By Corollary~\ref{l:openFINITE}, choose a uniform base
	$\mathcal B$ of $\mathcal U$ consisting of open convex
	covers.
	
	We use the realization of $\widehat X$ by minimal
	Cauchy filters on $X$. The filter corresponding to
	$\Phi\in\widehat X$ consists of those subsets of $X$
	which contain the trace of some neighborhood of
	$\Phi$ in $\widehat X$.
	For every open $A\subseteq X$, put
	$\widehat A:=\{\Phi\in\widehat X:A\in\Phi\}$.
	Consequently,
	\[
	\widehat A
	=\widehat X\setminus
	\overline{X\setminus A}^{\,\widehat\tau},
	\qquad
	\widehat A\cap X=A.
	\]
	In particular, $\widehat A$ is $\widehat\tau$-open.
	By the standard description of the uniform completion,
	the covers
	$\widehat\alpha:=\{\widehat A:A\in\alpha\}$,
	where $\alpha\in\mathcal B$, form a uniform base of
	$\widehat{\mathcal U}$.
	
	For any finitely many distinct filters
	$\Phi_1,\ldots,\Phi_n\in\widehat X$, we can therefore
	choose pairwise disjoint open convex sets
	$A_j\in\Phi_j$.
	Indeed, choose pairwise disjoint
	$\widehat\tau$-open neighborhoods $O_j$ of the $\Phi_j$
	and then choose $\alpha\in\mathcal B$ sufficiently fine
	that every member of $\widehat\alpha$ containing
	$\Phi_j$ lies in $O_j$.
	Selecting $A_j\in\alpha$ with
	$\Phi_j\in\widehat A_j$ gives
	$A_j=\widehat A_j\cap X\subseteq O_j\cap X$,
	as required.
	
	\smallskip
	\emph{Construction of the circular order.}
	For distinct $\Phi,\Psi,\Theta\in\widehat X$, choose
	pairwise disjoint open convex sets
	$A\in\Phi$, $B\in\Psi$ and $C\in\Theta$.
	Define $[\Phi,\Psi,\Theta]_{\widehat R}$ if
	$[A,B,C]_R$.
	By Lemma~\ref{l:ConvexProprties}(3), exactly one of
	$[A,B,C]_R$ and $[A,C,B]_R$ holds.
	
	This definition is independent of the choices.
	Indeed, if $A',B',C'$ are another such triple, the
	filter property gives points
	$a\in A\cap A'$, $b\in B\cap B'$ and $c\in C\cap C'$.
	These points are distinct, and
	Lemma~\ref{l:ConvexProprties}(3) gives
	$[A,B,C]_R\iff[a,b,c]_R\iff[A',B',C']_R$.
	Choosing pairwise disjoint convex representatives
	simultaneously for all the distinct filters occurring
	in a circular order axiom reduces that axiom to the
	corresponding axiom for $R$.
	Thus $\widehat R$ is a circular order.
	
	The filter representing $x\in X$ is its neighborhood
	filter, so every member of that filter contains $x$.
	The construction therefore gives
	$[i(x_1),i(x_2),i(x_3)]_{\widehat R}
	\iff[x_1,x_2,x_3]_R$
	for distinct $x_1,x_2,x_3\in X$.
	Hence the canonical uniform embedding $i$ is also
	a COP embedding.
	
	\smallskip
	\emph{Compatibility with the completion topology
		and uniformity.}
	We first show that $\widehat R$ is open in
	$\widehat\tau^3$.
	Suppose that $[\Phi,\Psi,\Theta]_{\widehat R}$ is
	witnessed by $A,B,C$ as above.
	Then $\widehat A,\widehat B,\widehat C$ are pairwise
	disjoint $\widehat\tau$-open neighborhoods of
	$\Phi,\Psi,\Theta$, respectively.
	For every
	$(\Phi',\Psi',\Theta')\in
	\widehat A\times\widehat B\times\widehat C$,
	the same sets $A,B,C$ witness
	$[\Phi',\Psi',\Theta']_{\widehat R}$.
	Thus
	$\widehat A\times\widehat B\times\widehat C
	\subseteq\widehat R$.
	In particular, every circular interval is
	$\widehat\tau$-open, and hence
	$\lambda_{\widehat R}\subseteq\widehat\tau$.
	
	We next prove that the $\widehat\tau$-closure of every
	convex subset $D\subseteq X$ is $\widehat R$-convex.
	Suppose otherwise.
	There would be four distinct points in cyclic order
	$[\Phi_1,\Phi_2,\Phi_3,\Phi_4]_{\widehat R}$ such that
	$\Phi_1,\Phi_3\in\overline D^{\,\widehat\tau}$ and
	$\Phi_2,\Phi_4\notin\overline D^{\,\widehat\tau}$.
	By openness of $\widehat R$, choose pairwise disjoint
	$\widehat\tau$-open neighborhoods $V_j$ of $\Phi_j$
	such that every choice of one point from each $V_j$
	has this cyclic order.
	Shrinking $V_2$ and $V_4$, we may also assume that
	both are disjoint from $\overline D^{\,\widehat\tau}$.
	Choose $x_1\in V_1\cap D$ and $x_3\in V_3\cap D$.
	By density of $X$, choose $x_2\in V_2\cap X$ and
	$x_4\in V_4\cap X$.
	Then $[x_1,x_2,x_3,x_4]_R$, with
	$x_1,x_3\in D$ and $x_2,x_4\notin D$.
	Neither $[x_1,x_3]_R$ nor $[x_3,x_1]_R$ is therefore
	contained in $D$, contradicting its convexity.
	
	Now let $A\in\alpha$, where $\alpha\in\mathcal B$.
	Since $A$ is convex, $X\setminus A$ is convex by
	Lemma~\ref{l:ConvexProprties}(1).
	The preceding argument shows that
	$\overline{X\setminus A}^{\,\widehat\tau}$ is convex
	in $\widehat X$.
	Applying the same lemma in $\widehat X$, its complement
	$\widehat A$ is convex as well.
	Thus $\widehat{\mathcal U}$ has a uniform base
	consisting of open $\widehat R$-convex covers.
	Together with
	$\lambda_{\widehat R}\subseteq\widehat\tau$,
	Proposition~\ref{p:GCO-eq} shows that
	$(\widehat X,\widehat\tau,\widehat R)$ is a $\GCOTS$.
	Consequently, $\widehat{\mathcal U}$ is a
	GCO-uniformity.
	
	\smallskip
	\emph{Uniqueness.}
	Let $S$ be another circular order extending $R$
	for which $\widehat{\mathcal U}$ is a GCO-uniformity.
	Then $\lambda_S\subseteq\widehat\tau$, so
	Lemma~\ref{l:c-is-Open} implies that $S$ is also open
	in the completion topology on distinct triples.
	If $S$ and $\widehat R$ differed, their openness
	would give pairwise disjoint nonempty
	$\widehat\tau$-open sets $U,V,W$ such that
	$[U,V,W]_{\widehat R}$ and $[U,W,V]_S$.
	Choosing $u\in U\cap X$, $v\in V\cap X$ and
	$w\in W\cap X$ contradicts the fact that both orders
	restrict to $R$ on $X$.
	Hence $S=\widehat R$.
	
	\smallskip
	(2) By Fact~\ref{f:G-compactFacts}(2), the canonical
	extension $\widehat\pi$ is
	$\widehat{\mathcal U}$-equiuniform and jointly continuous.
	For every $g\in G$, its extended translation
	$\widehat g$ is a uniform homeomorphism of $\widehat X$,
	with inverse $\widehat{g^{-1}}$.
	Since $\widehat g\circ i=i\circ g$ and $i(X)$ is dense,
	Lemma~\ref{l:AutomaticCOP} implies that
	$\widehat g$ is COP.
	Thus $\widehat\pi$ is a COP action.
\end{proof}

\subsection*{An internal description of the uniformity for Nov\'{a}k's completion } 

By Theorems \ref{t:CompletionEmbeddingAndDensity} and \ref{t:NovakMinimality}, Nov\'{a}k regular completion $(\mathcal{X}_r, \mathrm{int}_r)$ is the minimal COTS compactification of $(X, \lambda_{\medcirc})$. This implies that $(X, \lambda_{\medcirc})$ is compatible with a unique precompact uniformity $\mu_R$ whose completion is $\mathcal{X}_r$. 
Theorem \ref{t:NovakUniformityInternal} provides a constructive, internal description of $\mu_{R}$ using only the order.
	  
\begin{thm}\label{t:NovakUniformityInternal}
	Let $(X,R)$ be a $\COTS$ with at least three elements.
	For every finite cycle $F=[a_1,\ldots,a_n]$ in $X$,
	with $n\geq3$, define its star cover by
	\[
	\mathcal C_F
	:=\{(a_i,a_{i+2})_R:1\leq i\leq n\},
	\]
	where indices are taken modulo $n$.
	Then the family $\mathfrak B$ of all such covers
	is a base for a precompact convex uniformity
	$\mu_{\medcirc}$ compatible with the interval topology
	$\lambda_{\medcirc}$.
	Moreover, $\mu_{\medcirc}$ equals the Nov\'ak
	compactification uniformity $\mu_R$.
\end{thm}

\begin{proof}
	Let $K=\Xcal_r$ be Nov\'ak's regular completion,
	with circular order $\Rcal$, and let $\mathbf M$
	be its unique compatible uniformity.
	Identify $X$ with its image under $\nu$.
	By the interval-base argument in the proof of
	Theorem~\ref{t:NovakMinimality}, the intervals
	$(a,b)_{\Rcal}$ with distinct endpoints $a,b\in X$
	form a base for the topology of $K$.
	
	For every finite cycle $F=[a_1,\ldots,a_n]$ in $X$,
	with $n\geq3$, consider the corresponding family
	$\mathcal C_F^K
	:=\{(a_i,a_{i+2})_{\Rcal}:1\leq i\leq n\}$.
	This is a finite open cover of $K$:
	each marker $a_i$ lies in
	$(a_{i-1},a_{i+1})_{\Rcal}$, and every point strictly
	between consecutive markers $a_i,a_{i+1}$ lies in
	$(a_i,a_{i+2})_{\Rcal}$.
	Since $\nu$ is a circular order embedding, the
	restriction of $\mathcal C_F^K$ to $X$ is exactly
	$\mathcal C_F$.
	By compactness of $K$, every $\mathcal C_F^K$
	is a uniform cover for $\mathbf M$.
	
	We show that these covers form a base of $\mathbf M$.
	Let $\mathcal V$ be an open cover of $K$.
	Using the preceding interval base and compactness,
	choose a finite cover
	$\{I_j=(u_j,v_j)_{\Rcal}:1\leq j\leq m\}$
	refining $\mathcal V$, with all $u_j,v_j\in X$.
	Let $F=[p_1,\ldots,p_N]$ be a finite cycle in $X$
	containing all these endpoints, enlarged if necessary
	so that $N\geq3$.
	
	Fix $i$ and choose $j$ such that $p_{i+1}\in I_j$.
	Both endpoints of $I_j$ belong to $F$.
	Since $p_i,p_{i+1},p_{i+2}$ are consecutive in $F$,
	we have
	$(p_i,p_{i+2})_{\Rcal}\subseteq I_j$.
	Indeed, in the linear order obtained by cutting $K$
	at $u_j$, consecutiveness gives
	$u_j\leq p_i<p_{i+1}<p_{i+2}\leq v_j$.
	Thus $\mathcal C_F^K$ refines $\mathcal V$.
	
	The covers $\mathcal C_F^K$ are therefore cofinal
	among the open covers of $K$, which form a base
	of its compatible uniformity $\mathbf M$.
	Consequently, their restrictions $\mathcal C_F$
	form a base of the induced uniformity $\mu_R$ on $X$.
	This proves that $\mathfrak B$ is a uniform base
	and that $\mu_{\medcirc}=\mu_R$.
	
	The uniformity is precompact because its basic covers
	are finite, and it is convex because their members
	are circular intervals.
	Finally, it induces $\lambda_{\medcirc}$ because
	$\nu$ is a topological embedding.
\end{proof}

Recall that for every LOTS there exists the minimal linearly ordered compactification, the classical Dedekind compactification (see, for example, \cite{Fed,Kaufman1967,Blatter1975,KentRichmond1988}). One may give an internal characterization of the corresponding precompact uniform structure   
using a linear modification of Theorem \ref{t:NovakUniformityInternal}. We omit the details and only formulate the description.

\begin{thm}[Uniformity for minimal LOP compactifications of LOTS] \label{t:DedekindUniformityInternal}
	Let $(X,\leq,\lambda_{\leq})$ be a $\LOTS$. Let $\mathfrak{B}$ be the family of all finite ``star covers" of $X$, where for any finite chain $F = [a_1, \dots, a_n]$ in $X$, the star cover $\mathcal{C}_F$ is defined as:
	$\mathcal C_F:=\{(a_{i-1},a_{i+1})_{\leq}:
	i=1,\ldots,n\}$, 
	where we put $a_0=-\infty$, $a_{n+1}=+\infty$. 
	Let $\mu_{\leq}$ be the uniformity on $X$ generated by the family $\mathfrak{B}$ as a base.
	Then $\mu_{\leq}$ is a precompact convex uniformity
	compatible with the interval topology $\lambda_{\leq}$
	and is equal to the Dedekind compactification uniformity. 
	Moreover, the star covers associated with two-element chains form a uniform subbase of $\mu_{\leq}$. 
\end{thm}

\begin{defin} \label{d:IntUnif} We call the uniformity 
 $\mu_{\medcirc}$ (= $\mu_R$) from Theorem \ref{t:NovakUniformityInternal} \textit{circular interval uniformity}.   
Similarly, for linear orders $(X,\leq)$ the uniformity $\mu_{\leq}$ from Theorem \ref{t:DedekindUniformityInternal} we call \textit{linear interval uniformity}.  
Mostly we say simply: \textit{interval uniformity} where the context is clear. 
	
\end{defin}

\begin{ex} \label{l:CompactCase} 
	Every compact COTS $X$ with its unique topologically compatible uniformity $\mathcal{U}_X$ is a (compact) GCO-uniformity. The same is true for every compact LOTS $X$. These are interval uniformities in the sense of Definition \ref{d:IntUnif}. 	
\end{ex}	
\begin{proof}
By Theorem \ref{t:NovakUniformityInternal} the interval uniformity $\mu_{\medcirc}$ is convex and compatible with the interval topology for every COTS. Since $X$ is compact, we obtain that its canonical unique uniformity $\mathcal{U}_X$ coincides with the convex uniformity $\mu_{\medcirc}$. Thus, $\mathcal{U}_X$ is a convex uniformity. 

For the case of LOTS use Theorem \ref{t:DedekindUniformityInternal}. 
\end{proof}

It is known that for every $(X,\leq,\tau) \in \GLOTS_{\curlyeqprec}$  
there exists a minimal LOP compactification (see Blatter \cite{Blatter1975} or Kent-Richmond \cite{KentRichmond1988}). One may show that this is just the Dedekind compactification of the ``least d-extension" $\theta(X)$.  Similar result is valid for circular orders.  

\begin{prop} \label{p:via-d-ext} 
	Let $(X,R,\tau)$ be a $\GCOTS_{\medcirc}$ with the least d-extension 
	$
	\theta \colon (X,R,\t) \to (\widetilde{X}, \widetilde{R},\tau^{\sim})
	$. 	
	Then the Novák completion of $(\widetilde{X}, \widetilde{R})$ is the minimal COP compactification of $(X,R,\tau)$. 
\end{prop}
\begin{proof}
	Let $\s \colon (X,R,\t) \to (Y,R_Y)$ be any proper COP compactification. By Lemma \ref{l:d-ext} there exists a (unique) COP topological embedding $\tilde{\s} \colon \widetilde{X} \to Y$ which extends $\s$. Clearly it is a proper COP compactification of the COTS  $\widetilde{X}$. 	 
	Consider now the Novák completion 
	$
	\nu \colon \widetilde{X} \longrightarrow (\widetilde{\Xcal})_r 
	$
	 of the COTS 
	$(\widetilde{X},\widetilde{R})$. 	  
	By Theorem \ref{t:NovakMinimality} 
	there exists a (unique) continuous onto COP map $q \colon Y \to (\widetilde{\Xcal})_r$ such that $q \circ \tilde{\s}=\nu$. This implies that $q \circ \tilde{\s} \circ \theta=\nu \circ \theta$. Since $\tilde{\s} \circ \theta =\s$, we obtain $q \circ \s= \nu \circ \theta$. Thus, the COP proper compactification $\nu \circ \theta \colon X \to (\widetilde{\Xcal})_r$ has the minimality property. 
\end{proof}

\subsection*{Uniformity for minimal COP compactifications of GCOTS} 
\label{s:UniformGCOTSmin} 

\begin{defin}\label{d:gen-star-cov}
	Let $(X,R,\tau)$ be a $\GCOTS_{\medcirc}$.
	For a finite cycle $F=[a_1,\ldots,a_n]$ in $X$,
	with $n\geq3$, define the
	\emph{generalized cycle star cover}
	$\mathcal C_F^G:=\bigcup_{i=1}^n\mathcal U(a_i)$,
	where indices are taken modulo $n$ and:
	\begin{itemize}
		\item[(i)] If $a_i\notin X^-\cup X^+$, put
		$\mathcal U(a_i)=\{(a_{i-1},a_{i+1})_R\}$.
		\item[(ii)] If $a_i\in X^-\setminus X^+$, put
		$\mathcal U(a_i)=
		\{(a_{i-1},a_i)_R,[a_i,a_{i+1})_R\}$.
		\item[(iii)] If $a_i\in X^+\setminus X^-$, put
		$\mathcal U(a_i)=
		\{(a_{i-1},a_i]_R,(a_i,a_{i+1})_R\}$.
		\item[(iv)] If $a_i\in X^-\cap X^+$, put
		$\mathcal U(a_i)=
		\{(a_{i-1},a_i)_R,\{a_i\},(a_i,a_{i+1})_R\}$.
	\end{itemize}
	The following lemma shows that $\mathcal C_F^G$
	is a finite cover by $\tau$-open circularly convex sets.
\end{defin}

\begin{lem}\label{l:predec}
	Let $\theta:(X,R,\tau)\to
	(\widetilde X,\widetilde R,\widetilde\tau)$
	be the circular least $d$-extension.
	Identify $X$ with $\theta(X)$.
	Then:
	\begin{enumerate}
		\item $X\cap(x^-,y)_{\widetilde R}=[x,y)_R$
		for every $x\in X^-$ and $y\in X\setminus\{x\}$.
		\item $X\cap(a,x^+)_{\widetilde R}=(a,x]_R$
		for every $x\in X^+$ and $a\in X\setminus\{x\}$.
		\item Let $F=[a_1,\ldots,a_n]$ be a finite cycle
		in $X$, with $n\geq3$.
		Define $F^\theta$ by listing, for each $a_i$, in order:
		$a_i^-$ if $a_i\in X^-$, then $a_i$, and then
		$a_i^+$ if $a_i\in X^+$.
		Thus a point singular on both sides contributes
		the block $a_i^-,a_i,a_i^+$, with $a_i$ occurring once.
		The trace on $X$ of the standard cycle star cover
		$\mathcal C_{F^\theta}$ is exactly $\mathcal C_F^G$.
	\end{enumerate}
\end{lem}

\begin{proof}
	In the circular lexicographic order, $x^-$ is the
	immediate predecessor of $x$, and $x^+$ is its
	immediate successor.
	Since the other endpoint is distinct from $x$,
	this proves (1) and (2).
	
	For (3), write $\ell_i,r_i$ for the first and last
	points of the block over $a_i$.
	For an arc whose endpoints lie over distinct points
	of $X$, replacing a left endpoint $a_i$ by $r_i$,
	or a right endpoint $a_i$ by $\ell_i$, does not
	change its trace on $X$: the point $a_i$ remains
	excluded, and no other point of $X$ lies between
	the old and new endpoints.
	
	If $a_i\notin X^-\cup X^+$, its star interval
	$(r_{i-1},\ell_{i+1})_{\widetilde R}$ has trace
	$(a_{i-1},a_{i+1})_R$.
	For the block $a_i^-,a_i$, the stars centered at
	these two points are
	$(r_{i-1},a_i)_{\widetilde R}$ and
	$(a_i^-,\ell_{i+1})_{\widetilde R}$.
	Their traces are respectively
	$(a_{i-1},a_i)_R$ and $[a_i,a_{i+1})_R$, by (1).
	The block $a_i,a_i^+$ gives the symmetric pair
	$(a_{i-1},a_i]_R$ and $(a_i,a_{i+1})_R$, by (2).
	Finally, the block $a_i^-,a_i,a_i^+$ gives
	$(a_{i-1},a_i)_R$, $\{a_i\}$, and
	$(a_i,a_{i+1})_R$, because
	$X\cap(a_i^-,a_i^+)_{\widetilde R}=\{a_i\}$.
	These are precisely the four families in
	Definition~\ref{d:gen-star-cov}.
	
	Since $\theta$ is a topological embedding and
	$\mathcal C_{F^\theta}$ is an open cover of
	$\widetilde X$, its trace is a $\tau$-open cover of $X$.
	All its members are circularly convex.
\end{proof}

\begin{thm}\label{t:GCOTS-Min-UnifInt}
	Let $(X,R,\tau)$ be a $\GCOTS_{\medcirc}$.
	If $|X|\geq3$, let $\mathfrak B^G$ be the family
	of all covers $\mathcal C_F^G$ from
	Definition~\ref{d:gen-star-cov}, where $F$ ranges
	over finite cycles with at least three elements.
	If $|X|\leq2$, let $\mathfrak B^G$ consist of the
	singleton cover $\{\{x\}:x\in X\}$.
	Then $\mathfrak B^G$ is a base for a precompact
	GCO-uniformity $\mu^G_{\medcirc}$ compatible with $\tau$.
	Moreover, $\mu^G_{\medcirc}$ equals the compactification
	uniformity $\mu_R$ of the minimal proper COP
	compactification $m \colon X\to X_m$.
\end{thm}

\begin{proof}
	If $|X|\leq2$, then $X$ is finite discrete and is
	its own minimal proper COP compactification.
	The assertion follows from the singleton cover.
	Assume henceforth that $|X|\geq3$.
	
	Identify $X$ with its image under the least
	$d$-extension $\theta:X\to\widetilde X$.
	By Theorem~\ref{t:NovakUniformityInternal}, the
	cycle star covers $\mathcal C_H$, where $H$ ranges
	over finite cycles in $\widetilde X$ with at least
	three elements, form a base of the circular interval
	uniformity $\mu_{\widetilde R}$.
	
	Adding markers refines a cycle star cover.
	Indeed, let $J=(p,q)$ be a star interval centered
	at $b$ in the larger cycle, and choose a star
	interval $I=(u,v)$ of the smaller cycle containing $b$.
	Since $u,v$ are also markers of the larger cycle,
	consecutiveness of $p,b,q$ gives
	$u\leq p<b<q\leq v$ in the cut order at $u$.
	Hence $J\subseteq I$.
	
	Every finite cycle $H$ in $\widetilde X$ is
	contained in some $F^\theta$.
	To obtain $F$, take all underlying points of $X$
	represented by the elements of $H$, and enlarge
	this finite set if necessary to have at least
	three elements.
	Thus $\mathcal C_{F^\theta}$ refines $\mathcal C_H$.
	Consequently, the covers $\mathcal C_{F^\theta}$
	are cofinal among the cycle star covers of
	$\widetilde X$.
	By Lemma~\ref{l:predec}(3), their traces on $X$
	are exactly the covers $\mathcal C_F^G$.
	These traces therefore form a base of
	$\mu^G_{\medcirc}:=\theta^*(\mu_{\widetilde R})$.
	This uniformity is compatible with $\tau$ because
	$\theta$ is a topological embedding.
	It is precompact and convex because its basic
	covers are finite and convex.
	
	Let $\nu:\widetilde X\to K$ be Nov\'ak's regular
	completion, and let $\mathbf M$ be the unique
	compatible uniformity of $K$.
	Theorem~\ref{t:NovakUniformityInternal} gives
	$\mu_{\widetilde R}=\nu^*(\mathbf M)$.
	By Proposition~\ref{p:via-d-ext}, we may identify
	$X_m$ with $K$ and write $m=\nu\circ\theta$.
	Therefore
	$\mu^G_{\medcirc}
	=\theta^*(\nu^*(\mathbf M))
	=m^*(\mathbf M)=\mu_R$.
\end{proof}

Note that in \cite{Sorin24} G.B. Sorin gives a  lattice-theoretic description of all proper COP compactifications of any GCOTS in terms of gaps and order completions. This successfully extends the classical case of GLOTS coming back to R. Kaufman \cite{Kaufman1967} and V. Fedorchuk \cite{Fed}. 

In this work we give a different approach based on convex uniformities. 

\subsection*{Representation of COP compactifications by GCO-uniformities}

Write $\operatorname{Comp}_{\mathrm{COP}}(X)$ for the poset of proper COP compactifications $j\colon X\to Y$ of $(X,R,\tau)$.  
Let $\operatorname{Unif}_{\mathrm{GCO}}(X)$ be the poset of precompact GCO-uniformities on $(X,R,\tau)$ partially ordered by inclusions.

\begin{thm} \label{t:UniformRep}
	There is a natural order isomorphism
	\[
	\operatorname{Comp}_{\mathrm{COP}}(X)\ \longleftrightarrow\ \operatorname{Unif}_{\mathrm{GCO}}(X),
	\]
	sending a compactification $j\colon X\to (Y,\Rcal_Y)$ to the pulled-back uniformity $\mathcal{U}_Y$  
	and a GCO-uniformity $\mathcal U$ to its completion
	compactification
	$i_{\mathcal U}:X\to\widehat X_{\mathcal U}$. 
\end{thm}
\begin{proof}
	Define
	$
	\Phi \colon \operatorname{Comp}_{\mathrm{COP}}(X) \to \operatorname{Unif}_{\mathrm{GCO}}(X), \  \Phi(j):=\mathcal{U}_Y,
	$
	where $j \colon X\to Y$ is a proper COP compactification and $\mathcal{U}_Y$ is the pulled-back
	uniformity on $X$ induced by the unique uniformity of the compact COTS $Y$.
	By Example \ref{l:CompactCase} 
	$U_Y$ is a precompact GCO-uniformity on $X$.
	
	Conversely, for every $\mathcal{U} \in \operatorname{Unif}_{\mathrm{GCO}}(X)$ let
	$
	\Psi(\mathcal{U}):=i_{\mathcal{U}} \colon X\to \widehat X_{\mathcal{U}},
	$
	where $\widehat X_{\mathcal{U}}$ is the uniform completion of $(X,\mathcal{U})$.
	By Proposition~\ref{p:ORDERING-BM}, $\widehat X_{\mathcal{U}}$ carries a canonically defined compact
	circular order $\widehat R$, and $i_{\mathcal{U}}$ is a proper COP compactification of $X$.
	
	These assignments are order preserving. Indeed, if $j_1\le j_2$ in $\operatorname{Comp}_{\mathrm{COP}}(X)$, so that there exists 
	a COP map $p \colon Y_2\to Y_1$ with $j_1=p\circ j_2$, then $p$ is uniformly continuous between compact
	uniform spaces, hence the pulled-back uniformities satisfy $\mathcal{U}_{Y_1}\subseteq \mathcal{U}_{Y_2}$.
	Thus $\Phi(j_1)\subseteq \Phi(j_2)$.
	
	Conversely, if $\mathcal{U}_1\subseteq \mathcal{U}_2$ in $\operatorname{Unif}_{\mathrm{GCO}}(X)$, then the identity map on $X$ is uniformly
	continuous $(X,\mathcal{U}_2)\to (X,\mathcal{U}_1)$. By the universal property of uniform completions, it extends
	uniquely to a continuous map $\widehat p \colon \widehat X_{\mathcal{U}_2}\to \widehat X_{\mathcal{U}_1}$ such that
	$i_{\mathcal{U}_1}=\widehat p\circ i_{\mathcal{U}_2}$. Since both $i_{\mathcal{U}_1}$ and $i_{\mathcal{U}_2}$ are COP embeddings with dense
	range, Lemma~\ref{l:AutomaticCOP} implies that $\widehat p$ is COP. Hence $\Psi(\mathcal{U}_1)\le \Psi(\mathcal{U}_2)$.
	
	Finally, the two assignments are inverse to each other. If $j \colon X\to Y$ is a proper COP compactification,
	then $Y$ is the uniform completion of $(X,\mathcal{U}_Y)$; therefore $\Psi(\Phi(j))$ is equivalent to $j$.
	If $\mathcal{U} \in \operatorname{Unif}_{\mathrm{GCO}}(X)$, then the pulled-back uniformity induced by the completion compactification
	$i_{\mathcal{U}} \colon X \to \widehat X_{\mathcal{U}}$ is exactly $\mathcal{U}$, so $\Phi(\Psi(\mathcal{U}))=\mathcal{U}$. 
\end{proof}

Below we use convex uniformities for continuous actions
and $G$-compactifications. They also yield a short proof
of the fragmentability of COP functions
(Theorem~\ref{monot}). 
 
\section{Compactifications of order preserving group actions}  
 \label{s:G-comp}

Let $X$ be a topological space and $\pi \colon G \times X \to X$ be a group action. A topology $\s$ on $G$ is called \textit{admissible} \cite{Arens} if $(G,\s)$ is a topological group which makes $\pi$ jointly continuous. 

\begin{thm} \label{t:myCOTS-G-comp} 
	Let $(X,R)$ be a $\COTS$ with its interval topology $\lambda_{\medcirc}$ and 
	$G \subseteq H_+(X,\lambda_{\medcirc})$ with the pointwise topology $\sigma^p_{X}$. 
	Let $\pi_X \colon G \times X \to X$ be the induced COP separately continuous action. Then 
	\begin{enumerate} 
	\item $\mu_{\medcirc}$ is equiuniform under the action of the semitopological group $(G,\sigma^p_{X})$. 
	\item Nov\'ak's COP compactification $\nu \colon (X,\lambda_{\medcirc})\longrightarrow (\Xcal_r,\mathrm{int}_r)$ 
	admits a canonically defined COP action  $\pi_r \colon (G,\sigma^p_{X}) \times \Xcal_r \to \Xcal_r$ which extends the given action $\pi_X$ and is continuous. 
			
		\item $(G,\sigma^p_{X})$ is a topological group and $\pi_X$ is continuous (i.e., $\sigma^p_{X}$ is an admissible topology).  
	\end{enumerate}
	
(*) Similar results are true for every $\LOTS$ $(X,\leq) $ and every subgroup $G \subseteq H_+(X,\lambda_{\leq})$. 		
\end{thm} 

\begin{proof}
	(1) Since the action is COP, every $g$-translation moves a cycle to a cycle. Hence every cover $\mathcal{C}_F = \{ (a_i, a_{i+2}) \}_{i=1}^n$ is moved to $\mathcal{C}_{gF} = \{ (ga_i, ga_{i+2}) \}_{i=1}^n$. This yields that  $\mu_{\medcirc}$ is $G$-saturated.
	To establish equiuniformity, we must show $G$-boundedness (Definition \ref{d:equiun}): for any star-cover $\mathcal{C}_F = \{ (a_i, a_{i+2}) \}_{i=1}^n$ of a finite cycle $F$, there exists $V \in \mathcal{N}(e)$ such that $\{Vx : x \in X\}$ refines $\mathcal{C}_F$. 
 
 If $n=3$, then either $X=F$, in which case the assertion is immediate,
 or we may add one point of $X\setminus F$ to $F$.
 Since the resulting star cover refines $\mathcal C_F$, we may therefore
 assume that $n\geq4$.
 
 Construct a refined finite cycle
 $F^*=[c_1,\ldots,c_m]$ by choosing
 $b_i\in(a_i,a_{i+1})$ whenever this interval is
 nonempty, and putting $b_i=a_i$ otherwise; in $F^*$
 we list only the resulting distinct points.
 
 For every $c_k\in F^*$ choose
 $D_k\in\mathcal C_F$ with $c_k\in D_k$.
 By Lemma~\ref{l:c-is-Open}(1), we may choose
 pairwise disjoint open neighborhoods
 $c_k\in O_k\subseteq D_k$ such that
 $[O_1,\ldots,O_m]$.
 Since the action is separately continuous and
 $F^*$ is finite, there exists $V\in N(e)$ such that
 $Vc_k\subseteq O_k$ for every $k$.
 In particular, $Vc_k\subseteq D_k$ for every
 $c_k\in F^*$.
 
 Now let $(a_i,a_{i+1})\neq\emptyset$.
 Then
 $[a_{i-1},a_i,b_i,a_{i+1},a_{i+2}]$
 is a subcycle of $F^*$.
 For every $g\in V$ we therefore have
 $[a_{i-1},ga_i,gb_i,ga_{i+1},a_{i+2}]$.
 If $x\in(a_i,b_i)$, then
 $gx\in(ga_i,gb_i)\subseteq(a_{i-1},a_{i+1})$.
 If $x\in(b_i,a_{i+1})$, then
 $gx\in(gb_i,ga_{i+1})\subseteq(a_i,a_{i+2})$.  
 
 Together with the already established inclusions
 $Vc_k\subseteq D_k$ for the marker points, this shows
 that $\{Vx:x\in X\}$ refines $\mathcal C_F$.
 Hence $\mu_{\medcirc}$ is $G$-bounded and therefore
 $G$-equiuniform.  
	
	(2) Because the $G$-action on $(X,\mu_{\medcirc})$
	is equiuniform and $X_r$ is its uniform completion,
	Proposition~\ref{p:ORDERING-BM}(2) guarantees that
	the extended action
	$\pi_r:G\times X_r\to X_r$
	is continuous and COP.
	
	(3) Put $K:=X_r$.
	By part~(2), the action
	$(G,\sigma_X^p)\times K\to K$ is jointly continuous.
	Hence every evaluation map $g\mapsto gy$, $y\in K$,
	is $\sigma_X^p$-continuous, and therefore
	$\sigma_K^p\subseteq\sigma_X^p$.
	The reverse inclusion is immediate from $X\subseteq K$,
	so $\sigma_X^p=\sigma_K^p$.
	
	Since $K$ is a compact COTS,
	Corollary~\ref{c:Arens} 
	implies that $\sigma_K^p$ coincides with the
	compact-open topology on $G$.
	Hence $\sigma_X^p$ is a topological group topology.
	Moreover, joint continuity of the original action
	follows already by restricting the continuous action
	$G\times K\to K$ to $G\times X$.  
	 
	(*) Similar arguments apply to $\LOTS$ by substituting cycles with finite chains.
\end{proof}

\begin{cor} \label{c:GCOTSareG-compactifiable} 
For every $\GCOTS$ (hence, for every $\COTS$) $G$-space $X$ there exists a proper COP $G$-compactification. Namely, the minimal COP (Novak's)  compactification of the d-extension $\widetilde X$.  
\end{cor}
\begin{proof} 
Combine Theorem \ref{t:myCOTS-G-comp} and Proposition \ref{p:ActionLeastExtension} about circularly ordered d-extension $\widetilde X$ which is a $\COTS$ $G$-space.	
\end{proof}	

\begin{remark} \label{r:p-topologies} 
The LOTS case (*) of Theorem \ref{t:myCOTS-G-comp} was proved by direct methods by Ovchinnikov \cite{Ovch} and also by B.V. Sorin \cite[Corollary 2]{Sorin-B-22} (using different ideas and properties of Dedekind compactifications). 
For the particular case of compact COTS (and LOTS), Theorem \ref{t:myCOTS-G-comp} can be derived also from Corollary \ref{c:Arens}. For two important cases of ordered spaces $\T$ and $[0,1]$, 
the coincidence of pointwise and compact-open topologies on $H_+(X)$ was mentioned, without presenting the proof, in \cite[Remark 7.9.2]{GM-AffComp} and \cite[Section 8]{GM-tame}, respectively. 	
\end{remark}

\subsection*{G-compactifications of generalized circularly ordered $G$-spaces}




\begin{thm} \label{t:special} 
Let \(X\) be a GCOTS with a continuous COP action of a topological group \(G\). Assume that \(G\) has an open subgroup \(H\) such that
$
H/H^\circ
$
is compact, where \(H^\circ\) denotes the identity component of \(H\). 
Let
$
j \colon X\hookrightarrow K
$
be a proper COP compactification such that every \(g\)-translation on \(X\) extends to a homeomorphism of \(K\). Then the extended action
$
G\times K\longrightarrow K
$
is jointly continuous. 
Moreover, every point of \(K\setminus j(X)\) 
is eventually \(G\)-fixed.

The analogous statement holds for GLOTS and LOP compactifications.
\end{thm} 
\begin{proof}  
Identify \(X\) with its image in its compactifications. Let
$
i \colon X\hookrightarrow M
$
be the minimal proper COP compactification, equipped with its continuous extended \(G\)-action. 
For GCOTS, its continuity follows from Corollary \ref{c:GCOTSareG-compactifiable}. 

First we show that the remainder \(M\setminus X\) is totally disconnected. 

Suppose \(C\subseteq M\setminus X\) is connected. Choose \(c\in X\) and cut the circular order of $M$ at \(c\). 
More precisely, we consider
$
M\setminus\{c\}
$
with the linear order induced by this cut. Its order topology agrees with its subspace topology inherited from \(M\). 
Since \(C\subseteq M\setminus X\) and \(c\in X\), we have \(C\subseteq M\setminus\{c\}\). Thus \(C\) remains connected when regarded as a subset of this linearly ordered space. 
 If it contained two distinct points, connectedness would force it to contain the interval between them. A nondegenerate connected interval contains a nonempty open interval of \(M\), which must meet the dense subset \(X\). This is impossible.

Since $H^\circ$ is connected, the orbit $H^\circ z$
is connected. It is contained in $M\setminus X$,
because $X$ is $G$-invariant.
Hence the preceding argument gives
$H^\circ z=\{z\}$.

Now we check that 
every \(H\)-orbit in the remainder is finite. 
Fix \(z\in M\setminus X\). Since \(H^\circ\) fixes \(z\), the continuous orbit map factors through the compact space \(H/H^\circ\). Therefore
$
P:=Hz
$ 
is compact.  

Choose $c\in X$.
The cut at $c$ induces a linear order on $P$.
Its order topology is coarser than the compact
subspace topology inherited from $M$; since the former
is Hausdorff, the two topologies coincide.
Thus $P$ is a compact LOTS and has a first and a
last point.

If \(P\) were infinite, its last and first points would form a jump in its induced circular order. Since \(H\) acts transitively and preserves the circular order on \(P\), every point of \(P\) would have an immediate predecessor and an immediate successor. Thus every point of \(P\) would be isolated. Compactness would then force \(P\) to be finite, a contradiction. 
Hence \(P\) is finite. The continuous orbit map \(H\to P\) now shows that \(H_z\) is open in \(H\). Since \(H\) is open in \(G\),
we obtain that 
$
G_z\ \text{is open in } G.
$

Note that 
for the linear case, this step is shorter: a compact transitive LOP orbit has a least point, and every order automorphism fixes that least point. Therefore the orbit is a singleton.

Now our aim is to transfer the open stabilizers to \(K\). 
Minimality provides a continuous COP factor map
\[
p:K\longrightarrow M,
\qquad p|_X=\operatorname{id}_X.
\]
It is equivariant: the relevant identities hold on the dense subset \(X\).

We need two elementary properties of \(p\). 
First,
$
p^{-1}(X)=X.
$
Indeed, if \(p(y)=x\in X\), take a net \(x_\lambda\in X\) converging to \(y\) in \(K\). Then \(x_\lambda\to x\) in \(M\). Because \(X\hookrightarrow M\) is an embedding, this is convergence to \(x\) in \(X\), and hence in \(K\). Therefore \(y=x\).

Second, for \(z\in M\setminus X\), we have 
$
|p^{-1}(z)|\le2.
$
The fibre is convex and disjoint from \(X\). After cutting at a point of \(X\), three points in the fibre would give a nonempty open interval contained in the fibre, contradicting density of \(X\).

The open subgroup $G_z$ preserves $p^{-1}(z)$.
If this fibre is a singleton, it is fixed pointwise.
If $p^{-1}(z)=\{a,b\}$, then $a,b$ form an adjacent
pair. A COP homeomorphism preserving this pair cannot
exchange its points: if, say, $(a,b)=\emptyset$, then
exchanging $a$ and $b$ would imply $(b,a)=\emptyset$,
which is impossible since the fibre is disjoint from
the dense subset $X$. Hence $G_z$ fixes
$p^{-1}(z)$ pointwise.
In the linear case, order preservation gives the same
conclusion immediately.

Thus, for every $y\in K\setminus X$, we have
$G_{p(y)}\subseteq G_y$, so $G_y$ is open.

To conclude continuity, note that every orbit map at
a point of $K\setminus X$ is locally constant.
Orbit maps at points of $X$ are continuous by
hypothesis, and every $g$-translation on $K$ is
continuous. Hence the action on $K$ is separately
continuous. Theorem~\ref{t:COTS_joint_cont} now
implies that the action is jointly continuous.
\end{proof}

\begin{thm} \label{t:G-bounded-GCOTS}
Let	$\pi \colon G \times X \to X$ be a $\t$-continuous COP action of a topological group $G$ on a $\GCOTS$ $(X,R,\t)$. Assume that $G$ either is locally compact or connected. If $\sigma \colon X \to Y$ is a proper COP compactification of $(X,\t)$ which admits an extension of the action $\pi_Y \colon G \times Y \to Y$ with continuous $g$-translations $Y \to Y$, then $\sigma$  necessarily is a $G$-compactification (that is, $\pi_Y$ is a continuous action).   
\end{thm}
\begin{proof}
	The hypothesis on $G$ in Theorem~\ref{t:special}
	is satisfied in both cases.
	
	If $G$ is connected, take $H=G$.
	
	If $G$ is locally compact, then $G/G^\circ$ is totally
	disconnected and locally compact. By van Dantzig's theorem,
	it has a compact open subgroup $C$ (see, for example, \cite[Lemma 2.5]{Reid}). Let 
	$q\colon G\to G/G^\circ$ and put $H=q^{-1}(C)$.
	Then $H$ is open, $H^\circ=G^\circ$, and
	$H/H^\circ\simeq C$ is compact.
	
	For every $g\in G$, the extended $g$-translation
	on $Y$ is a homeomorphism, since its inverse is the
	continuous extension of the $g^{-1}$-translation.
	Therefore all the hypotheses of
	Theorem~\ref{t:special} are satisfied, and the
	extended action $G\times Y\to Y$ is jointly continuous.
	Thus $\sigma$ is a $G$-compactification.
\end{proof}

%
%
%

\begin{remark} \label{r:correction} 
An analogue of Theorem~\ref{t:G-bounded-GCOTS}
holds, with the same proof, for GLOTS and LOP actions. 
For locally compact $G$, this corrects
\cite[Theorem 3.18]{Me-MaxEqComp}, where the assertion was stated for arbitrary topological groups; that generality is not valid. 
\end{remark}

\section{Fragmented functions, tame families}  
\label{s:fr} 

By a classical fact, going back to R. Baire (1899),  every monotone function $X \to \R$ is a Baire 1 function for every interval $X \subseteq \R$. In Theorem \ref{monot} we  show this result can be generalized to much more general order preserving functions.   

A function $f \colon X \to Y$ is said to be 
{\em of Baire class 1} if the inverse image of every open set in $Y$ is an $F_\sigma$ set (the union of countably many closed sets) in $X$.  
Notation: $f \in \mathcal{B}_1(X,Y)$ and  $f \in \mathcal{B}_1(X)$ for $Y=\R$.  
If $X$ is separable and metrizable then a real valued function $f \colon  X \to \R$ is Baire 1 if and only if $f$ is a pointwise limit of a sequence of continuous functions (see, for example, \cite{Dulst, GMU08}).  
A function $f \colon X \to Y$ has the {\em point of continuity property}  
if for every closed nonempty subset $A$ of $X$ the restriction $f|_A \colon A \to Y$ has a point of continuity. 
When $X$ is compact or Polish and $(Y,d)$ is a (pseudo)metric space then $f\colon X \to Y$ is fragmented iff $f$ has the point of continuity property (see \cite{GM1}).   
Recall the following (slightly generalized) definition of fragmentability which comes from Banach space theory 
and is effectively used also in dynamical systems theory \cite{Me-nz,GM1,GM-survey}.  

\begin{defin}\label{d:frag}
	Let $X$ be a topological space and $(Y,\mu)$ a uniform space. 
	A map $f \colon X\to(Y,\mu)$ is \emph{fragmented} if for every entourage $\eps \in\mu$ and every nonempty set $A\subseteq X$ there exists a nonempty relatively open set $O\subseteq A$ such that 
	$
	f(O)\times f(O)\subseteq \eps.
	$
	Equivalently, $f(O)$ is \emph{$\eps$-small}. We write $f\in\mathcal{F}(X,Y)$, and $f\in\mathcal{F}(X)$ when $Y=\R$ with its usual uniformity.
\end{defin}

\begin{f} \label{f:fr}  \cite[p. 137]{Dulst} For every Polish space $X$, we have $\mathcal{F}(X)=\mathcal{B}_1(X)$. More generally, if $X$ is Polish and $(Y,d)$ is a separable metric space then $f \colon X \to Y$ is fragmented if and only if $f$ is a Baire class 1 function. 
\end{f}

\begin{lem}\label{l:lift-fragm}
	Let $(Y,\mu)$ be a uniform space, $(X,R)$ a $\COTS$, $c\in X$, and $f \colon X \to(Y,\mu)$.
	Then $f\circ q_c \colon X(c)\to Y$ is fragmented if and only if $f \colon X\to Y$ is fragmented.
\end{lem}

\begin{proof}
	($\Rightarrow$) Assume $f\circ q_c$ is fragmented. Fix an entourage $\eps \in\mu$ and a nonempty $A\subseteq X$.
	If $A \subseteq \{c\}$, then $f(A)$ is a singleton and thus trivially $\eps$-small. So, we may assume $A\setminus\{c\} \neq \emptyset$.
	By Corollary~\ref{cover}(iii), $q_c$ restricts to a homeomorphism 
	$q_c \colon X(c)\setminus\{c^-,c^+\} \rightarrow  X\setminus\{c\}$.
	The set $B:=q_c^{-1}(A\setminus\{c\})$ is nonempty in $X(c)\setminus\{c^-,c^+\}$.
	Since $f\circ q_c$ is fragmented, there exists a nonempty relatively open subset $U\subseteq B$ such that $(f\circ q_c)(U)$ is $\eps$-small.
	Let $O:=q_c(U)$. Because $q_c$ is a homeomorphism off the split points, $O$ is nonempty and relatively open in $A\setminus\{c\}$, which is itself an open subspace of $A$. Thus, $O$ is relatively open in $A$, and $f(O)=(f \circ q_c)(U)$ is $\eps$-small. Hence, $f$ is fragmented.
	
	($\Leftarrow$) 
	Let $A \subseteq X(c)$ be nonempty and $\eps \in \mu$.
	By the fragmentability of $f$ applied to the subset $q_c(A) \subseteq X$, there exists an open subset $O \subseteq X$ such that $q_c(A) \cap O$ is nonempty and $f(q_c(A) \cap O)$ is $\eps$-small in $Y$.
	Let $W := A \cap q_c^{-1}(O)$. Because $q_c$ is continuous, $W$ is a relatively open subset of $A$. Furthermore, $q_c(W) = q_c(A) \cap O \neq \emptyset$, making $W$ nonempty. Consequently, $(f \circ q_c)(W) = f(q_c(A) \cap O)$ is $\eps$-small, proving $f \circ q_c$ is fragmented.
\end{proof}
 
The following result is a generalization of the classical theorem of Baire and its proof relies on convex uniformities.    
   
\begin{thm} \label{monot} Let $(X,R,\t)$ be a $\GCOTS$
	(or $\GLOTS$) and let $Y$ be a $\COTS$ (resp. $\LOTS$) space with a \textbf{precompact convex}  compatible uniformity $\mu$ on $Y$.
	Then 
	\begin{enumerate}
		\item 
		Every order preserving map $f\colon  (X,\t) \to (Y,\mu)$ is fragmented.
		\item If, in addition, $(X,\t)$ is Polish and $\mu$ is metrizable and separable (e.g., $Y$ is a bounded interval in $\R$) then $f\colon (X,\t) \to Y$ is of Baire class 1.   
	\end{enumerate}
\end{thm}   
\begin{proof}
	(1) Fix a nonempty subset $M\subseteq X$ and an
	entourage $\eps\in\mu$.
	We seek a nonempty relatively open subset $O\subseteq M$
	such that $f(O)$ is $\eps$-small.
	If $M$ has an isolated point $y$, take $O=\{y\}$.
	Thus we may assume that $M$ has no isolated points;
	in particular, $M$ is infinite.
	
	Since $\mu$ is precompact and convex, choose a finite
	cover $\alpha$ of $Y$ by convex $\eps$-small sets.
	Since $M$ is infinite, some $A\in\alpha$ has the
	property that $C\cap M$ contains at least three
	distinct points, where $C:=f^{-1}(A)$.
	
	Consider first the circularly ordered case.
	By Proposition~\ref{p:Prop-c-ord}(2b), $C$ is convex.
	If $C=X$, the choice $O=M$ works, since $f(M)\subseteq A$.
	Otherwise choose $w\in X\setminus C$ and label three
	distinct points of $C\cap M$ as $x,y,z$ so that
	$x<_w y<_w z$ in the linear order obtained by
	cutting the circular order of $X$ at $w$.
	Then $y\in(x,z)_R$, whereas $w\in(z,x)_R$.
	Since $x,z\in C$ and $C$ is convex, at least one of
	$[x,z]_R$ and $[z,x]_R$ is contained in $C$.
	The latter is impossible because $w\notin C$.
	Hence $[x,z]_R\subseteq C$.
	
	It follows that $O:=M\cap(x,z)_R$ contains $y$ and
	is relatively open in $M$, because
	$\lambda_R\subseteq\tau$.
	Moreover, $f(O)\subseteq A$, so $f(O)$ is
	$\eps$-small.
	
	In the linearly ordered case, $C$ is linearly convex.
	Choose $x<y<z$ in $C\cap M$.
	Then $[x,z]\subseteq C$, and
	$O:=M\cap(x,z)$ is a nonempty relatively open subset
	of $M$ with $f(O)\subseteq A$.
	This proves fragmentability in both cases.
	
	(2) The conclusion follows from~(1) and
	Fact~\ref{f:fr}.
\end{proof}	

This implies that in item (2) the discontinuity set 
$D(f)$ is a meager $F_{\s}$-set in $X$.
One of the immediate important dynamical consequences
of Theorem \ref{monot} is the tameness of any compact COTS $G$-space (with COP action); see Theorem \ref{t:COTStame}.  	

\subsection*{Independent sequences of functions}
\label{s:ind}

Let $f_n: X \to \R$ be a uniformly bounded sequence of functions on a \emph{set} $X$. Following Rosenthal \cite{Ro} we say that this sequence is an \emph{$l_1$-sequence} on $X$ if there exists a constant $a >0$
such that for all $n \in \N$ and choices of real scalars $c_1, \dots, c_n$ we have
$$
a \cdot \sum_{i=1}^n |c_i| \leq ||\sum_{i=1}^n c_i f_i||_{\infty}.
$$

For every $l_1$-sequence $f_n$, its closed linear span in $l_{\infty}(X)$
is linearly homeomorphic to the Banach space $l_1$.
In fact, the map $l_1 \to l_{\infty}(X), \ (c_n) \to \sum_{n \in \N} c_nf_n$ is a linear homeomorphic embedding.

A Banach space $V$ is said to be {\em Rosenthal} if it does
not contain an isomorphic copy of $l_1$, or equivalently, if $V$ does not contain a sequence which is equivalent to an $l_1$-sequence. 
Every Asplund (in particular, every reflexive) space is Rosenthal. 
A sequence $f_n$ of	real valued functions on a set
$X$ is said to be \emph{independent} (see \cite{Ro,Tal,Dulst}) if
there exist real numbers $a < b$ such that
$$
\bigcap_{i \in P} f_i^{-1}(\leftarrow,a) \cap  \bigcap_{j \in M} f_j^{-1}(b,\to) \neq \emptyset
$$
for all finite disjoint subsets $P, M$ of $\N$.

\begin{defin} \label{d:tameF} \cite{GM-tame,GM-tLN}  
	We say that a bounded family $F$ of real valued (not necessarily continuous)  
	functions on a set $X$ is {\it tame} if $F$ does not contain an independent infinite sequence.
\end{defin} 

The following observation is straightforward. Its equivalent form was mentioned in \cite{GM-tLN,Me-b}.  

\begin{f} \label{f:lift-TameFamily} \cite{Me-b} 
	Let $g \colon X_1 \to X_2$ be an onto map. Then $F$ is a tame family of functions from $X_2$ to $\R$ if and only if $F \circ g$ is a tame family of functions from $X_1$ to $\R$. 
\end{f}

\begin{ex} \label{ex:tame2} 
	Let $(X,\leq)$ be a linearly ordered set. Then any family $F$ of order preserving real functions is tame. Moreover there is no independent pair of functions in $F$.  
\end{ex}
\begin{proof}    
	Assuming that $f_1, f_2 \in F$ is an independent pair there exist $a < b$ and $x,y \in X$ such that $x \in f_1^{-1}(\leftarrow,a) \cap f_2^{-1}(b, \to)$ and $y \in f_2^{-1}(\leftarrow,a) \cap f_1^{-1}(b, \to)$. Then $f_1(x) < f_1(y)$ and $f_2(y) < f_2(x)$. Since $f_1$ and $f_2$ are order preserving and $X$ is linearly ordered we obtain that $x <y$ and $y < x$, a contradiction.  
\end{proof}

The following useful result is a reformulation of some known facts. It is based on results of Rosenthal \cite{Ro}, Talagrand \cite[Theorem 14.1.7]{Tal}, van Dulst \cite{Dulst} and also \cite[Section 4]{GM-rose}. 

\begin{f} \label{f:sub-fr} \cite[Theorem 3.11]{Dulst} 
	Let $X$ be a compact space and $F \subset C(X)$ a bounded subset. The following conditions are equivalent:
	\begin{enumerate}
		\item
		$F$ does not contain an $l_1$-sequence. 
		\item $F$ is a tame family (does not contain an independent sequence). 
		\item
		Each sequence in $F$ has a pointwise convergent subsequence in $\R^X$.
		\item 
		The pointwise closure ${\overline{F}}$ of $F$ in $\R^X$ consists of fragmented maps,
		that is,
		${\overline{F}} \subset {\mathcal F}(X).$
	\end{enumerate}
\end{f}

\begin{lem}\label{l:cop-lift-principle}
	Let $f\colon X\to\T$ be a COP map from a circularly
	ordered set $X$.
	For every $c\in X$, put $y_0=f(c)$.
	There exists a LOP map
	$h\colon(X(c),\le_c)\to([0,1],\le)$ such that
	$q_*^{\,y_0}\circ h=f\circ q_c$, where
	$q_*^{\,y_0}\colon[0,1]\to\T$ is defined by
	$q_*^{\,y_0}(t)=y_0e^{2\pi i t}$. 
		$$
	\xymatrix{
		(X(c), \leq_c) \ar[r]^-{h} \ar[d]_-{q_c} & ([0,1], \leq) \ar[d]^-{q_*^{\,y_0}} \\
		X \ar[r]^-{f} & \T
	}
	$$
	If $f$ is nonconstant, this lift is unique and
	satisfies $h(c^-)=0$ and $h(c^+)=1$.
\end{lem}

\begin{proof}
	Put $L=X(c)=[c^-,c^+]$ and $\varphi=f\circ q_c$.
	If $f$ is constant, then $h\equiv0$ is a LOP lift.
	Henceforth assume that $f$ is nonconstant.
	
	The fiber $F=f^{-1}(y_0)$ is circularly convex
	by {\rm(COP2)}, and its complement is circularly
	convex by Lemma~\ref{l:ConvexProprties}(1).
	Since $c\in F$, the set
	$A=q_c^{-1}(X\setminus F)$ is a nonempty linearly
	convex subset of $L$, avoiding both endpoints.
	Indeed, if $u<_c v$ belong to $A$, the circular
	interval $[q_c(v),q_c(u)]_R$ contains $c$.
	Convexity of $X\setminus F$ therefore forces
	$[q_c(u),q_c(v)]_R\subseteq X\setminus F$,
	and hence $[u,v]_{\le_c}\subseteq A$.
	
	Let $I_0$ consist of the points of $L$ lying
	strictly below every point of $A$, and let $I_1$
	consist of those lying strictly above every point
	of $A$.
	Since $A$ is nonempty and linearly convex,
	$L=I_0\sqcup A\sqcup I_1$.
	Thus $I_0$ is an initial segment, $I_1$ is a
	terminal segment, and
	$\varphi^{-1}(\{y_0\})=I_0\cup I_1$.
	Moreover, $c^-\in I_0$ and $c^+\in I_1$.
	
	Let $\psi\colon\T\setminus\{y_0\}\to(0,1)$ be the
	inverse of $q_*^{\,y_0}|_{(0,1)}$. Define
	\[
	h(u)=
	\begin{cases}
		0, & u\in I_0,\\
		\psi(\varphi(u)), & u\in A,\\
		1, & u\in I_1.
	\end{cases}
	\]
	By construction, $q_*^{\,y_0}\circ h=\varphi$.
	
	To verify monotonicity on $A$, take $u<_c v$ in $A$.
	If $\varphi(u)=\varphi(v)$, then $h(u)=h(v)$.
	Otherwise, $[c,q_c(u),q_c(v)]_R$ and {\rm(COP1)}
	give $[y_0,\varphi(u),\varphi(v)]_{\T}$,
	so $h(u)<h(v)$.
	Since $h(A)\subseteq(0,1)$, while $I_0$ precedes
	$A$ and $I_1$ follows $A$, the map $h$ is LOP
	on all of $L$.
	
	Finally, let $h'$ be any LOP lift.
	On $A$, it must agree with $h$, since each point
	of $\T\setminus\{y_0\}$ has exactly one preimage
	under $q_*^{\,y_0}$.
	On $I_0\cup I_1$, its values belong to $\{0,1\}$.
	Fix $a\in A$.
	For $u\in I_0$, monotonicity gives
	$h'(u)\le h'(a)<1$, hence $h'(u)=0$.
	For $u\in I_1$, it gives
	$h'(u)\ge h'(a)>0$, hence $h'(u)=1$.
	Therefore $h'=h$, proving uniqueness.
\end{proof}

\section{Functions of bounded variation on ordered sets} \label{s:BV} 

In \cite{Me-Helly} we proposed a natural generalization of the  
classical definition of bounded variation functions for functions $f \colon X \to Y$ (or into metric spaces $Y$) with $X \subseteq \R, Y \subseteq \R$. Namely, we consider any linearly ordered set in the domain. This can be extended also to circularly ordered sets \cite{GM-c}. 

\begin{defin} \label{d:BV} \ 
	\begin{enumerate} 
		\item \cite{Me-Helly}
		Let $(X,\leq)$ be a linearly ordered set and $(M,d)$ a metric space. 
		We say that a bounded function $f \colon  (X,\leq) \to (M,d)$ has variation not greater than $r$ (notation: $f \in BV_r$) if 
	$$
			\sum_{i=1}^{n-1} d(f(x_i),f(x_{i+1})) \leq r
	$$
		for every choice of $x_1 < x_2 < \cdots < x_n$ in $X$.
		
		\item \cite{GM-c}
		For circularly ordered sets $(X,R)$ (instead of $(X,\leq)$) the definition is similar but we take injective \emph{cycles} $[x_1, x_2, \cdots, x_n]$  in $X$  (Definition \ref{d:cycl}) and require that 
		$$
			\sum_{i=1}^{n} d(f(x_i),f(x_{i+1})) \leq r
		$$ 
		where $x_{n+1}=x_1$.  
	\end{enumerate} 
	
	The least upper bound of all such possible sums is the {\it variation} of $f$; notation (both cases):   $\Upsilon(f)$. If  $\Upsilon(f) \leq r$ then we write $f \in BV_r(X,M)$ or, simply $BV_r(X)$ if $(M,d)=\R$. If $f(X) \subset [c,d]$ for some reals $c \leq d$ then we write also $f \in BV_r(X,[c,d])$. One more notation: $BV(X):=\cup_{r>0} BV_r(X)$. In order to distinguish linear and circular cases, sometimes we use more special notation writing:  $BV^{\prec}_r(X,M)$, $\Upsilon^{\prec}(f)$ and $BV^{\circ}_r(X,M)$, $\Upsilon^{\circ}(f)$ respectively. 
\end{defin}

\begin{remarks} \label{r:FinCol} \ 
	\begin{enumerate} 
		\item  
		$M_+(X,[c,d]) \subset BV^{\prec}_r(X,[c,d])$ holds for every $r \geq d-c$. 
		In particular, $M_+(X,[0,1]) \subset BV_1(X,[0,1])$.  
	
		\item If $f_1 \colon X\to Y$ is LOP between LOTS and $f_2\in BV^{\prec}_r(Y,M)$, then
		$f_2\circ f_1\in BV^{\prec}_r(X,M)$. 
		For every COP map $f_1 \colon X \to Y$ and every $f_2 \in BV^{\circ}_r(Y,M)$ we have $f_2 \circ f_1 \in BV^{\circ}_r(X,M)$.  
		
		\item For every $1$-Lipschitz map $\alpha \colon M_1 \to M_2$ between metric spaces and every $f \in BV^{\prec}_r(X,M_1)$ ($f \in BV^{\circ}_r(X,M_1)$) on a linearly (circularly) ordered set $X$ we have $\alpha \circ f \in BV_r^{\prec}(X,M_2)$ (resp., $BV_r^{\circ}(X,M_2)$).  
	\end{enumerate}
\end{remarks}

$H([0,1]) \subset BV^{\prec}_1([0,1],[0,1])$ because 
every homeomorphism $[0,1] \to [0,1]$ is either order preserving or order reversing.  
For every finite interval partition of a c-ordered set $(X,R)$ (or of a linearly ordered set $(X,\leq)$), every finite coloring 
$f \colon X \to \Delta \subset \R$  of this partition is a function with bounded variation.  
This observation is important, in particular, for Sturmian like systems \cite{GM-c,GM-tLN}.   

\begin{lem} \label{l:BVprop} 
Let $(X,\leq)$ be a linearly ordered set and $(X,R)$ a circularly ordered set.
Suppose that $(M,d)$ is a compact metric space. 
	\begin{enumerate} 
		\item 
		$BV^{\prec}_r(X,M)$ and $BV^{\circ}_r(X,M)$ are pointwise closed (hence, compact) subsets of $M^X$. 
		
		\item $M_+(X,[c,d])$ is a closed subset of $BV^{\prec}_r(X,[c,d])$ for every $r \geq d-c$.
		
		\item 
		(Analog of Jordan's decomposition)  Every function $f \in BV^{\prec}(X)$ on $(X, \leq)$ is a difference $f=u-v$ of two linear order preserving bounded functions $u,v\colon  X \to \R$. 
	\end{enumerate}
\end{lem}
\begin{proof} (1) Is straightforward. 
	
	(2) The set $M_+(X,[c,d])$ is pointwise closed in $[c,d]^X$ (use that the linear order of $[c,d]$ is closed). The set $BV^{\prec}_r(X,[c,d])$ is also pointwise closed by part (1) of this lemma. By Remark \ref{r:FinCol}.1, $M_+(X,[c,d])$ is a subset of $BV^{\prec}_r(X,[c,d])$. 
	
	(3) 
	If in Definition \ref{d:BV}.1 we
	allow only the chains $\{x_i\}_{i=1}^n$ with $x_n \leq c$ for some given $c \in X$ then we obtain a
	variation on the subset $\{x \in X: x \leq c\} \subset X$. Notation: $\Upsilon^{c}(f)$. 
	As in the classical case (as, for example, in \cite{Natanson}) it is easy to see that the functions 
	$u(x):=\Upsilon^{x}(f)$ and $v(x):=u(x)-f(x)$ on $X$ are increasing. These functions are bounded because $|\Upsilon^{x}(f)| \leq \Upsilon(f)$ and $f$ is bounded. 
\end{proof}

\begin{lem} \label{l:FisVectSp} 
	$\mathcal{F}(X)$ is a vector space over $\R$ with respect to the natural operations. 
\end{lem}
\begin{proof} Clearly, $f \in \mathcal{F}(X)$ implies that $cf \in \mathcal{F}(X)$ for every $c \in \R$. 
	Let $f_1, f_2 \in \mathcal{F}(X)$. We have to show that $f_1 + f_2 \in \mathcal{F}(X)$. 
	Let $\emptyset \neq A \subset X$ and $\eps >0$. Since $f_1 \in \mathcal{F}(X)$  there exists an open subset $O_1 \subset X$ such that 
	$A \cap O_1 \neq \emptyset$ and $f_1(A \cap O_1)$ is $\frac{\eps}{2}$-small. Now since $f_2 \in \mathcal{F}(X)$, for $A \cap O_1$ we can choose an open subset $O_2 \subset X$ such that $(A \cap O_1) \cap O_2$ is nonempty and $f_2(A \cap O_1 \cap O_2)$ is $\frac{\eps}{2}$-small. Then $(f_1+f_2)(A \cap (O_1 \cap O_2))$ is $\eps$-small.    
\end{proof}

\begin{cor} \label{c:BVisFr} 
	$BV^{\prec}(X) \subset \mathcal{F}(X)$ for any $(X,\leq) \in $ LOTS$_{\leq}$ .  
\end{cor}
\begin{proof} Any $f \in BV^{\prec}(X)$ is a difference of two increasing functions (Lemma \ref{l:BVprop}.3). Hence we can combine Theorem \ref{monot} and Lemma \ref{l:FisVectSp}. 	
\end{proof}

\begin{thm} \label{newprinciple} \cite{Me-Helly}  
	For every linearly ordered set $(X,\leq)$ the family of functions $BV^{\prec}_r(X,[c,d])$ 
	is tame. In particular, for every $r \geq d-c$, its subfamily $M_+(X,[c,d])$ is also tame. 
\end{thm}
\begin{proof} 
	Let $f_n \colon X \to \R$ be an independent sequence in $BV^{\prec}_r(X,[c,d])$. 
	By Lemma \ref{l:BVprop}.3, for every $n$ we have $f_n=u_n -v_n$, where $u_n(x):=\Upsilon^{x}(f_n)$ and $v_n(x):=u_n(x)-f_n(x)$  are increasing functions on $X$. Moreover, the family $\{u_n, v_n\}_{n \in \N}$ remains bounded because $|\Upsilon^{x}(f_n)| \leq \Upsilon(f_n) \leq r$ for every $x \in X, n \in \N$ and $f_n$ is bounded. 
	Apply Fact \ref{RepLem} to the GLOTS $(X,\leq,\lambda_{disc})$. Then we conclude that there exist two bounded sequences $t_n\colon  Y \to \R$ and $s_n\colon  Y \to \R$ of continuous increasing functions on a compact LOTS $Y$ which extend $u_n$ and $v_n$. 
	Consider $F_n:=t_n -s_n$.  
	First, note that for a sufficiently large $k \in \R$, $F_n \in BV^{\prec}_k(Y,[-k,k])$ holds simultaneously for all $n \in \N$.
	
	Since $F_n|_X=f_n$ we clearly obtain that the sequence $F_n\colon  Y \to \R$ is independent, too. 
	On the other hand we can show that  $\overline{\Gamma} \subset \mathcal{F}(Y)$, where $\Gamma=\{F_n\}_{n \in \N} \subset \R^Y$. 
	Indeed, by	Corollary \ref{c:BVisFr} we know that 
	$BV^{\prec}_k(Y,[-k,k]) \subset \mathcal{F}(Y)$. Using Lemma \ref{l:BVprop}.1 we get 	
$$\overline{\Gamma} \subset \overline{BV^{\prec}_k(Y,[-k,k])} = BV^{\prec}_k(Y,[-k,k]) \subset \mathcal{F}(Y).$$
	Then $\Gamma$ is a tame family 
	by Fact \ref{f:sub-fr}. This contradiction completes the proof. 
\end{proof}

\subsection*{Generalization of Helly's sequential compactness type theorems for linear orders} 
\label{s:H} 

Recall that the Helly's compact space $M_+([0,1],[0,1])$ of all increasing selfmaps on the closed unit interval $[0,1]$ is sequentially compact in the pointwise topology.  
A slightly more general form of this result is the following classical result of Helly.  

\begin{f} \label{t:HellyClassic} 
	{\bf (Helly's selection theorem)} 
	
	\noindent For every sequence of functions from the set $BV^{\prec}_r([a,b],[c,d])$ of all real functions $[a,b] \to [c,d]$ with variation $\leq r$, there exists a pointwise convergent subsequence (which tends to a function with finite variation). 
	Equivalently, $BV^{\prec}_r([a,b],[c,d])$ is sequentially compact.
\end{f}

\begin{thm} \label{t:GenH}   
	Let $(X, \leq)$ be a linearly ordered set. Then $BV^{\prec}_r(X, [c,d])$ is sequentially compact 
	in the pointwise topology. 
	In particular, for every $r \geq d-c$, its closed subspace $M_+(X,[c,d])$ of all increasing functions is also sequentially compact in the pointwise topology.  
\end{thm} 
\begin{proof} 
	Using Lemma \ref{l:BVprop}.3  
	one may reduce the proof to the case where $f_n\colon  X \to \R$ is a bounded sequence in $M_+(X)$. 
	By Fact \ref{RepLem} (for the GLOTS $(X,\leq,\lambda_{disc})$) we have a bounded sequence of {\it continuous} increasing functions $F_n\colon  Y \to \R$ on a compact LOTS $Y$, where $F_n|_X=f_n$. 
	By Theorem \ref{newprinciple} the sequence $F_n$ does not contain an independent subsequence. By Fact  \ref{f:sub-fr} 
	there exists a convergent subsequence $F_{n_k}$. Since the convergence is pointwise and $X$ is a subset of $Y$, we obtain that the corresponding sequence of restrictions $f_{n_k}:=F_{n_k}|_X$ is pointwise convergent on $X$. 
\end{proof} 

The following corollary can be derived also by results of \cite{FP}. 

\begin{cor} \label{increas} 
	Let $(X, \leq)$ be a linearly ordered set. Then the compact space $M_+(X,[c,d])$ of all order preserving maps 
	is sequentially compact. 
\end{cor} 

Using Nachbin's Lemma \ref{l:Nachbin} we give now in Theorem \ref{corOfGenHelly1} a further generalization  replacing $[c,d]$ in Theorem \ref{t:GenH} by partially ordered compact metrizable spaces. This gives a partial generalization of \cite[Theorem 7]{FP}. Some restriction (e.g., metrizability) on a compact ordered space $Y$ is really essential as it follows from \cite[Theorem 9]{FP}.

\begin{thm} \label{corOfGenHelly1} 
	Let $(X, \leq)$ be a linearly ordered set and $(Y, \leq)$ be a compact metrizable partially ordered space. Then the compact space $M_+(X,Y)$ of all LOP maps is sequentially compact. 
\end{thm}
\begin{proof} First of all note that $M_+(X,Y)$ is compact being a closed subset of $Y^X$. Here we have to use the assumption that the given order on $Y$ is closed (Definition \ref{d:ord}). 
	$M_+(X,[0,1])$ is sequentially compact by 
	Theorem \ref{t:GenH}. Therefore, its countable power  $M_+(X,[0,1])^{\N}$ is also sequentially compact. Now observe that $M_+(X,Y)$ is topologically embedded as a closed subset into $M_+(X,[0,1])^{\N}$. 
	Indeed,  by  Nachbin's Lemma \ref{l:Nachbin} continuous increasing maps $Y \to [0,1]$ separate the points. Since $Y$ is a compact metrizable space  one may choose a countable family $h_n$ of increasing continuous maps which separate the points of $Y$. For every $f \in M_+(X,Y)$ define the  function 
	$$u(f) \colon \N \to M_+(X,[0,1]), \ n \mapsto h_n \circ f.$$
	This assignment defines a natural topological embedding of compact Hausdorff spaces (hence, this embedding is closed) 
	$
	u \colon M_+(X,Y) \hookrightarrow M_+(X,[0,1])^{\N}, \ \ f \mapsto u(f)=(h_n \circ f)_{n \in \N}.
	$
\end{proof}

Another Helly type theorem can be obtained for functions of bounded variation with values into a compact metric space. 
In the particular case of $(X,\leq)=[a,b]$ Theorem  \ref{corOfGenHelly2} is well known, \cite{BC,Ch}.      

\begin{thm} \label{corOfGenHelly2}    
	Let $(X, \leq)$ be a linearly ordered set and $(Y,d)$ be a compact metric space. 
	Then the compact space $BV^{\prec}_r(X,Y)$ 
	is sequentially compact in the pointwise topology. 
\end{thm}
\begin{proof} Since $(Y,d)$ is a compact metric space there exist  
	countably many $1$-Lipschitz functions $h_n \colon  (Y,d) \to [0,1]$ which separate the points of $Y$. Indeed, take a countable dense subset $\{y_n: \ n \in \N\}$ in $Y$ and define $h_n(x):=\frac{d(y_n,x)}{\max\{1,\operatorname{diam}Y\}}$. 
	Then $h_n \circ f \in  BV^{\prec}_r(X,[0,1])$ for every $f \in BV^{\prec}_r(X,Y)$. 
	The rest is similar to the proof of Theorem \ref{corOfGenHelly1}.  
\end{proof}
  
An elegant argument was presented by Rosenthal in \cite{Ros1}.  
The set $M_+([0,1],[0,1])$ is a 
compact subset in the space $\mathcal{B}_1[0,1]$ of Baire 1 functions. 
Hence it is sequentially compact because the compactness and sequential compactness are the same for subsets $\mathcal{B}_1(X)$ for any Polish $X$, \cite{Ros1}.  

\subsection*{Circularly ordered sets and BV}

\begin{lem}\label{l:lift-BV}
	Let $f \colon X\to (M,d)$ be a map, where $(M,d)$ is a bounded metric space, and let $q\colon X(c)\to X$ be the canonical quotient from Corollary~\ref{cover}. Then
	\[
	\Upsilon^{\prec}(f\circ q)\ \le\ \Upsilon^{\circ}(f)\ \le\ \Upsilon^{\prec}(f\circ q)+\diam(M).
	\]
	Consequently,
	\begin{align*}
		&\text{(i)}\quad f\in BV^{\circ}_{r}(X,M)\ \Longrightarrow\ f\circ q\in BV^{\prec}_{r}(X(c),M),\\
		&\text{(ii)}\quad f\circ q\in BV^{\prec}_{t}(X(c),M)\ \Longrightarrow\ f\in BV^{\circ}_{\,t+\diam(M)}(X,M).
	\end{align*}
\end{lem}
\begin{proof}
	Put $f_q:=f\circ q$.
	Take a finite chain
	$y_1<_c\cdots<_c y_n$ in $X(c)$ and set
	$x_i:=q(y_i)$.
	
	If the $x_i$ are pairwise distinct, then
	$[x_1,\ldots,x_n]$ is an injective cycle in $X$.
	Hence
	\[
	\sum_{i=1}^{n-1}
	d(f_q(y_i),f_q(y_{i+1}))
	\leq \Upsilon_\circ(f).
	\]
	The only possible failure of injectivity is
	$y_1=c^-$ and $y_n=c^+$, in which case
	$x_1=x_n=c$ and all the remaining $x_i$ are distinct.
	For $n>2$, the linear sum is exactly the circular
	sum corresponding to the cycle
	$[c,x_2,\ldots,x_{n-1}]$; for $n=2$ it is zero.
	Thus in all cases
	$\Upsilon_\prec(f_q)\leq\Upsilon_\circ(f)$.
	
	For the reverse inequality, take an injective cycle
	$[x_1,\ldots,x_n]$ in $X$.
	If $c\notin\{x_1,\ldots,x_n\}$, then, after a cyclic
	reindexing, its unique lifts satisfy
	$y_1<_c\cdots<_c y_n$ in $X(c)$.
	Therefore its circular variation sum is at most
	$\Upsilon_\prec(f_q)+\diam(M)$, the additional term
	being $d(f(x_n),f(x_1))$.
	
	If $c$ belongs to the cycle, cyclically reindex it
	so that $x_n=c$.
	Then
	$c^-<_c x_1<_c\cdots<_c x_{n-1}<_c c^+$,
	and the circular variation sum of the cycle is
	exactly the linear variation sum of $f_q$ on this
	chain. Hence it is at most $\Upsilon_\prec(f_q)$.
	
	Taking suprema gives
	$\Upsilon_\prec(f\circ q)\leq\Upsilon_\circ(f)
	\leq\Upsilon_\prec(f\circ q)+\diam(M)$.
	Assertions (i) and (ii) follow immediately.
\end{proof}

A crucial application of the bounded variation concept, particularly for dynamics, involves c-order preserving maps into the circle, $f \colon X \to \T$. It is a non-trivial fact that any such map from a circularly ordered set $X$ to the circle $\T$ 
is of bounded variation.	

\begin{lem} \label{l:COPtoTisBV} 
	For every circularly ordered set $(X,R)$, one has $M_+(X,\T) \subset BV^\circ(X,\T)$ and $M_+(X,\T)$ is pointwise closed.
\end{lem}

\begin{proof}
	Pointwise closedness follows from Theorem \ref{t:M_Closed}. 
	
	For the BV inclusion: Let $f \in M_+(X,\T)$. Fix $c \in X$ and set $y_0 := f(c)$. 
	By Lemma \ref{l:cop-lift-principle}, the COP map $f$ lifts to a LOP map $h \colon (X(c), \le_c) \to ([0,1], \le)$ such that $f \circ q_c = q_*^{y_0} \circ h$.
	
	Because $h$ is a LOP map into $[0,1]$, we have $h \in BV_1^{\prec}(X(c), [0,1])$ (see Remark \ref{r:FinCol}.1). 
	The standard quotient map $q_*^{y_0} \colon [0,1] \to \T$ is a 1-Lipschitz function 
	for the normalized geodesic metric
	$
	d_{\mathbb T}(p(s),p(t))
	=
	\min_{k\in\mathbb Z}|s-t-k|.
	$
	By Remark \ref{r:FinCol}.3, composing a BV function with a Lipschitz map preserves bounded variation, yielding $\Upsilon^{\prec}(q_*^{y_0} \circ h) \le \Upsilon^{\prec}(h) \le 1$.
	
	Applying Lemma \ref{l:lift-BV} (with $M=\T$), we obtain:
	\[
	\Upsilon^\circ(f) \le \Upsilon^{\prec}(f \circ q_c) + \operatorname{diam}(\T) = \Upsilon^{\prec}(q_*^{y_0} \circ h) + \operatorname{diam}(\T) \le 1 + \operatorname{diam}(\T) < \infty.
	\]
	Thus, $f \in BV^\circ(X,\T)$.
\end{proof}

\begin{thm}[Helly Selection Theorem for \textbf{circular orders}]  \label{GenHellyThm} 
	Let $(X,R)$ be a circularly ordered set and $(Y,d)$ be a compact metric space.
	Then $BV^{\circ}_t(X,Y)$ is sequentially compact in the pointwise topology for any $t>0$. 
	In particular, the space of COP maps $M_+(X,\T)$ is sequentially compact.
\end{thm}   
\begin{proof}
	Fix $c \in X$ and consider the canonical quotient map $q \colon X(c) \to X$. 
	By Lemma \ref{l:lift-BV}(i), the pullback map 
	\[
	q_{\#} \colon BV_{t}^{\circ}(X,Y) \to BV_{t}^{\prec}(X(c),Y), \quad f \mapsto f \circ q
	\]
	is well-defined. This map is trivially continuous with respect to the pointwise topologies.
	
	By Lemma \ref{l:BVprop}(1), the domain $BV_{t}^{\circ}(X,Y)$ is compact. Since $q$ is surjective, $q_{\#}$ is injective. Because a continuous injection from a compact space into a Hausdorff space is a topological embedding, $BV_{t}^{\circ}(X,Y)$ is homeomorphic to its closed image under $q_{\#}$. 
	By Theorem \ref{corOfGenHelly2}, the space $BV_{t}^{\prec}(X(c),Y)$ on the linearly ordered set $X(c)$ is sequentially compact. Because closed subspaces of sequentially compact spaces are sequentially compact, $BV_{t}^{\circ}(X,Y)$ is also sequentially compact.
	
	Finally, by Lemma \ref{l:COPtoTisBV}, $M_+(X,\T)$ is a pointwise closed subspace of $BV^{\circ}_t(X,\T)$ for $t = 1 + \operatorname{diam}(\T)$. Therefore, it inherits sequential compactness.
\end{proof}

\begin{thm} \label{t:BVX(c)} 	
	Let $(X, R)$ be a c-ordered set. Then any family 
	$F \subseteq [c,d]^X$ with uniformly bounded total variation is tame. In particular, $BV^{\circ}_r(X,[c,d])$ is a tame family for every $r>0$.
\end{thm} 

\begin{proof} 
	It suffices to prove that the maximal family $F := BV^{\circ}_r(X,[c,d])$ is tame for any fixed $r>0$. 
	
	Fix an arbitrary $z \in X$ and consider the canonical quotient map $q \colon X(z) \to X$. 
	By Lemma \ref{l:lift-BV}(i), the pullback family $F_q := \{f \circ q \mid f \in F\}$ is entirely contained within $BV^{\prec}_r(X(z),[c,d])$. 
	
	By Theorem \ref{newprinciple}, $BV^{\prec}_r(X(z),[c,d])$ is a tame family on the linearly ordered space $X(z)$, and therefore its subfamily $F_q$ is also tame. Because $q$ is surjective, Fact \ref{f:lift-TameFamily} immediately guarantees that the original family $F$ is tame on $X$.
\end{proof}

\begin{thm} \label{t:BVisFr}  
	Let $X$ be a $\LOTS$ or a $\COTS$, and let $M$ be a compact metric space.  
	Then every bounded variation function $f \colon X \to M$ is fragmented. Furthermore, if $X$ is Polish, then $f$ is of Baire class 1.
\end{thm}  

\begin{proof} 
	\textit{Case 1: The scalar case.} Assume $M$ is a bounded interval $[c,d] \subset \R$. If $X$ is a $\LOTS$, the fragmentability of $f$ follows directly from Corollary \ref{c:BVisFr}. If $X$ is a $\COTS$, fix an arbitrary $z \in X$. By Lemma \ref{l:lift-BV}, the pullback $f \circ q_z$ has bounded variation on the linearly ordered split space $X(z)$. By the LOTS case, $f \circ q_z$ is fragmented, and Lemma \ref{l:lift-fragm} immediately implies that the original function $f$ is fragmented on $X$.
	
	\textit{Case 2: The general case.} Observe that the family of 1-Lipschitz distance functions 
	\[
	\{h_z \colon M \to [0,\operatorname{diam}(M)], \quad h_z(x):=d(z,x) \mid z \in M \}
	\]
	separates the points of $M$. For every $z \in M$, the composition $h_z \circ f \colon X \to [0,\operatorname{diam}(M)]$ has bounded variation by Remark \ref{r:FinCol}.3. By Case 1, each scalar function $h_z \circ f$ is fragmented. Since the family $\{h_z\}_{z \in M}$ separates the points of $M$, applying \cite[Lemma 2.3.3]{GM-rose} guarantees that the original target function $f$ is fragmented. If the topology on $X$ is Polish, then $f \in \mathcal{B}_1(X)$ by Fact \ref{f:fr}.  	
\end{proof}

	In \cite{Me-medianBV} we study functions of bounded variation on \textit{median pretrees}, which is a natural generalization of linearly ordered sets.

\section{Order preserving actions and tameness} \label{s:tame} 

Let $X$ be a compact dynamical $G$-system and let $h \colon G \to \operatorname{H}(X)$ be the induced (continuous) homomorphism. Recall that the pointwise closure $E(X):=\operatorname{cl}_p(h(G))$ of $h(G)$ in $X^X$ is a compact right topological (Ellis) semigroup, which is said to be the \textit{enveloping semigroup} and reflects several important dynamical properties of $(G,X)$, \cite{Ellis,Gl-env}. 

A compact $G$-space $K$ is said to be \textit{tame}  \cite{Ko,KL} if $fG$ is a tame family for every $f \in C(K)$ (\textit{regular}, in terms of K{\"o}hler). By results of \cite{GMU08,GM-rose,GM-tLN} it is equivalent to require that for every element $p \in E(K)$, the map $p \colon K \to K$ is fragmented. 
If, in addition, $K$ is metrizable, then this means that $p$ is a Baire 1 function  
and we can assume that $E(K)$ is a (separable) pointwise compact subset of $\mathcal{B}_1(K,K)$. That is, a \textit{Rosenthal compactum}. 
It is well known that every Rosenthal compactum is sequentially compact (by a result of Bourgain-Fremlin-Talagrand 
\cite[Theorem 3F]{BFT}) and also a \textit{Fr\'echet space} (a space where the closure operator and the sequential closure are the same).   
 
\begin{thm} \label{t:COTStame} 
Every compact $\COTS$ $G$-space $K$ is tame. 	
\end{thm}
\begin{proof} $M_+(K,K)$ is pointwise closed in $K^K$ by Theorem \ref{t:M_Closed}. Since $E(K) \subset M_+(K,K)$, 
for every $p \in E(K)$ the map $p \colon K \to K$ is COP, hence fragmented by Theorem \ref{monot}. 	
\end{proof}

\subsection*{Representations on Rosenthal spaces} 

Let $V$ be a Banach space and let $\operatorname{Iso}(V)$ be the topological group (with the strong operator topology) of all onto linear isometries $V \to V$. For every continuous homomorphism $h \colon G \to \operatorname{Iso}(V)$, we have a canonically induced dual continuous action on the weak-star compact unit ball $B_{V^*}$ of the dual space $V^*$. So, we get a $G$-space $B_{V^*}$. 
A natural question is which continuous actions 
of $G$ on a topological space $X$ can be represented as a $G$-subspace of $B_{V^*}$ for a certain Banach space $V$ from a nice class of (low-complexity) spaces.  
Recall that a dynamical system $(G,X)$ is WRN (\textit{Weakly Radon-Nikodym}) if it is representable on a Rosenthal Banach space \cite{GM-rose}. In particular, this defines the class of WRN compact spaces.  

\begin{f} \label{WRNcriterion}   
	(see \cite[Theorem 6.5]{GM-rose} and with more details \cite{Me-b}) 
	Let $X$ be a compact $G$-space. The following conditions are equivalent:
	\begin{enumerate}
		\item $(G,X)$ is Rosenthal representable (that is, $(G,X)$ is $\mathrm{WRN}$). 
		\item There exists a point separating bounded $G$-invariant family $F$ of continuous real functions on $X$ such that $F$ is a tame family.
	\end{enumerate}
\end{f}

Recall that every Rosenthal representable compact $G$-space is tame by \cite{GM-rose}. If $X$ is metrizable and tame, then it is necessarily Rosenthal representable. The following result was established in \cite{GM-c}. Its proof uses the arguments of BV functions on c-ordered sets.  

\begin{thm} \label{t:CoisWRN} \cite{GM-c} 
	Every c-ordered compact, not necessarily metrizable, $G$-space $X$ is Rosenthal representable (that is, $\mathrm{WRN}$), and therefore tame. 
	Consequently, $\mathrm{CODS} \subset \mathrm{WRN} \subset \mathrm{Tame}$. 
\end{thm}

\begin{proof}
	Let $X$ be a c-ordered compact $G$-system. By Fact \ref{WRNcriterion}, it suffices to find a $G$-invariant, bounded, point-separating family $F \subseteq C(X)$ that is tame. By Theorem \ref{t:BVX(c)}, establishing that $F$ has uniformly bounded total variation is sufficient to guarantee its tameness.

	If $|X|\leq2$, the assertion is immediate: the action
	factors through a finite group of permutations, and a
	finite point-separating $G$-invariant family of continuous
	functions is tame.
	Assume therefore that $|X|\geq3$.
	
	Let $a\neq b$ be points of $X$ and choose
	$c\in X\setminus\{a,b\}$.
	Consider the cut space $X(c)=[c^-,c^+]$ and its
	quotient map $q:X(c)\to X$.
	After interchanging $a$ and $b$ if necessary, we have
	$c^-<_c a<_c b<_c c^+$.
	
	The intervals $[c^-,a]$, $[a,b]$, and $[b,c^+]$
	are compact LOTS.
	On $[c^-,a]$ put $f_1\equiv0$.
	By Nachbin's Lemma~\ref{l:Nachbin}, the endpoint map
	$a\mapsto0$, $b\mapsto1$ extends to a continuous LOP
	map $f_2:[a,b]\to[0,1]$.
	Similarly, choose a continuous LOP map
	$u:[b,c^+]\to[0,1]$ with $u(b)=0$, $u(c^+)=1$,
	and put $f_3:=1-u$.
	Thus $f_3$ is order reversing, with
	$f_3(b)=1$ and $f_3(c^+)=0$.
	
	The maps $f_1,f_2,f_3$ agree at their common endpoints
	and hence glue to a continuous map
	$f:X(c)\to[0,1]$.
	It is constant on $[c^-,a]$, increases from $0$ to $1$
	on $[a,b]$, and decreases from $1$ to $0$ on
	$[b,c^+]$; hence $\Upsilon_\prec(f)=2$.
	
	Since $f(c^-)=f(c^+)=0$, the quotient property of $q$
	gives a unique continuous map $f_0:X\to[0,1]$ such
	that $f_0\circ q=f$.
	Moreover, $f_0(a)=0$ and $f_0(b)=1$, so $f_0$
	separates $a$ and $b$.
	By Lemma~\ref{l:lift-BV},
	$\Upsilon_\circ(f_0)\leq
	\Upsilon_\prec(f)+1=3$.
	
	For every ordered pair of distinct points $a,b\in X$,
	choose one such function $f_0$, and denote the resulting
	point-separating family by $F_0$.
	Put $F:=F_0G=\{f\circ g:f\in F_0,\ g\in G\}$.
	Then $F$ is bounded, $G$-invariant and point separating.
	Since every $g$ is a COP homeomorphism,
	Remark~\ref{r:FinCol}.(2), applied also to $g^{-1}$, gives
	$\Upsilon_\circ(f\circ g)=\Upsilon_\circ(f)$.
	Hence every member of $F$ has circular variation at
	most $3$.
	
	By Theorem~\ref{t:BVX(c)}, $F$ is tame.
	Fact~\ref{WRNcriterion} now implies that $X$ is Rosenthal
	representable.  
\end{proof}

Every GCOTS $G$-space $X$ admits a proper COP
$G$-compactification $X\hookrightarrow K$ by
Corollary~\ref{c:GCOTSareG-compactifiable}.
The compact COTS $G$-space $K$ is Rosenthal representable
by Theorem~\ref{t:CoisWRN}; hence so is its $G$-subspace $X$.

Using Theorem \ref{t:CoisWRN}, one may show that many Sturmian-like multidimensional symbolic $\Z^d$-systems are circularly ordered, hence tame. See \cite{GM-c,GM-tLN} for details. 	

As a purely topological consequence of Theorem \ref{t:CoisWRN}, Corollary \ref{c:COTSisCOMPACTIFIABLE} and Lemma  \ref{l:easyPropGCO}(2), we obtain:  

\begin{prop} \label{p:CorGCOTSisWRN} 
	Every $\GCOTS$ and every $\GLOTS$ is $\mathrm{WRN}$. 
\end{prop}

For instance, the two arrows compact LOTS space $K$ is Rosenthal representable. At the same time, $K$ is not Asplund representable by a result of Namioka \cite[Example 5.9]{N}.  
Note that $\beta \N$ is not WRN, a result of Todor\u{c}evi\'{c} (see \cite{GM-tame}).  

\begin{thm} \label{t:GrRepr} 
	The topological group $H_+(X)$ (with compact-open topology) is Rosenthal representable for every c-ordered compact space $X$. For example, this is the case for $H_+(\T)$.  
\end{thm} 
\begin{proof}
	(See also \cite{GM-tame}.)
	Let $G:=H_+(X)$ with the compact-open topology.
	By Theorem~\ref{t:CoisWRN}, the $G$-system $X$ admits a
	representation $(h,\alpha)$ on a Rosenthal Banach
	space $V$, where
	$
	h:G\to\operatorname{Iso}(V)
	$
	is a continuous anti-homomorphism and
	$\alpha:X\to B_{V^*}$ is a topological embedding. 
	
	Define
	$
	h^*:G\to\operatorname{Iso}(V),\
	h^*(g):=h(g^{-1}).
	$
	Then $h^*$ is a continuous homomorphism. It is
	injective, since the action of $G=H_+(X)$ on $X$ is
	faithful and $\alpha$ is equivariant.
	
	It remains to show that $h^*$ is a topological
	embedding. Suppose that $h^*(g_i)\to h^*(g)$ in the
	strong operator topology. Then
	$h(g_i^{-1})\to h(g^{-1})$ strongly. By equivariance,
	$
	\alpha(g_i^{-1}x)\longrightarrow\alpha(g^{-1}x)
	$ 
	for every $x\in X$ 
	in the weak-star topology. Since $\alpha$ is a
	topological embedding,
	$g_i^{-1}x\to g^{-1}x$ for every $x\in X$.
	Thus $g_i^{-1}\to g^{-1}$ in the pointwise topology
	$\sigma_X^p$.
	
	By Corollary~\ref{c:Arens}, since $X$ is compact,
	$\sigma_X^p$ coincides on $H_+(X)$ with the
	compact-open topology. Hence
	$g_i^{-1}\to g^{-1}$ in the compact-open topology,
	and therefore $g_i\to g$.
	Thus $h^*$ is a topological group embedding.
\end{proof}

The Ellis compactification $j \colon G \to E(G,\T)$ of the Polish group $G=H_+(\T)$ is a topological embedding. Indeed, the compact-open topology on $j(G) \subset C_+(\T, \T)$ coincides with the pointwise topology \cite[Section 8]{GM-tame}. 
This observation implies, by \cite[Remark 4.14]{GM-survey}, that
$\mathrm{Tame}(G)$ separates points and closed subsets. 
Thus, by Theorem~\ref{t:GrRepr}, every c-orderly topological
group (and hence every orderly topological group) is
Rosenthal representable.

For example, $\R$ is orderly as it can be embedded into $H_+([0,1])$, where $[0,1]$ is treated as the two-point compactification of $\R$.
Recall that it is yet unknown (see \cite{GM-survey}) whether every topological group is Rosenthal representable.

\subsection*{When is the universal system $M(G)$ circularly ordered?} 

Recall that for every topological group $G$ there exist the canonically defined 
universal minimal $G$-system $M(G)$
and universal irreducible affine $G$-system $I\!A(G)$. 
In \cite{GM-tame,GM-UltraHom21} we discuss some examples of Polish groups $G$ for which $M(G)$ and $I\!A(G)$ are tame. 
These properties can be viewed as natural generalizations of 
extreme amenability and amenability, respectively. 

Let us say that $G$ is \emph{intrinsically c-ordered} (\emph{intrinsically tame}) if the $G$-system $M(G)$ is c-ordered (respectively, tame). 
In particular, we see that $G=H_+(\T)$ 
is intrinsically c-ordered, using a well known result of Pestov \cite{Pest98} which identifies $M(G)$ as the tautological action of $G$ on the circle $\T$. 
Note also that 
the Polish groups $\Aut(\mathbf{S}(2))$ and 
$\Aut(\mathbf{S}(3))$, of automorphisms 
of the circular directed graphs $\mathbf{S}(2)$ and $\mathbf{S}(3)$, are also intrinsically c-ordered. 
The universal minimal $G$-systems for the groups $\Aut(\mathbf{S}(2))$ and 
$\Aut(\mathbf{S}(3))$ are computed by L. Nguyen van Th\'{e} in \cite{van-the}. One can show that $M(G)$ for these groups are c-ordered $G$-systems, see \cite{GM-tame}.  
It is interesting to find more examples where $M(G)$ is c-ordered. 

The following definition is justified by Todor\u{c}evi\'{c}'s Trichotomy and the dynamical version of the Bourgain-Fremlin-Talagrand dichotomy \cite{GM1}. 

\begin{defin} \label{d:TameClasses} \cite{GM-TC} 
A compact metrizable dynamical $G$-system is said to be:
	\begin{enumerate}
		\item Tame$_\mathbf{1}$ if $E(X)$ is first countable.
		\item Tame$_\mathbf{2}$ if $E(X)$ is hereditarily separable. 
	\end{enumerate}	
	By results of \cite{GM-TC} we know that 
	$
	\mathrm{Tame}_\mathbf{2} \subset \mathrm{Tame}_\mathbf{1} \subset \mathrm{Tame}. 
	$ 
\end{defin}
  
\begin{thm} \label{t:LOT1}
	Every linearly ordered compact metric dynamical system is Tame$_\mathbf{1}$.
\end{thm}
\begin{proof} Let $X$ be a compact metrizable linearly ordered dynamical system. Every element $p \in E(X)$ is a LOP selfmap $X \to X$, because $M_+(X,X)$ is pointwise closed. 
	So, $E(X)$ is a subspace of $M_+(X,X)$ (Generalized Helly space), which is first countable by Theorem \ref{t:lin-Helly}. 
\end{proof}

\begin{ex} \label{ex:Helly} 
	Consider the linearly ordered $H_+([0,1])$-system $[0,1]$. 
	The enveloping semigroup of this order preserving system is a (compact) subspace of the Helly space, which is first countable. So, this system is Tame$_\mathbf{1}$. It is not Tame$_{\bf 2}$. 
	In fact, it is (like the Helly space) not hereditarily separable. 
\end{ex}

\begin{prop} \label{ex:T} \cite[Proposition 8.10]{GM-TC} 
	Let $H_+(\T)$ be the Polish topological group of all c-order preserving homeomorphisms of the circle $\T$. The minimal circularly ordered dynamical system $(H_+(\T), \T)$ is tame but not \rm{Tame}$_\mathbf{1}$. 
\end{prop}
\begin{proof}
	This system is tame being a circularly ordered system (Theorem \ref{t:CoisWRN}). 
	However, the enveloping semigroup of this c-order preserving system is not first countable. 
	Choose any point $a \in \T$. For every $b \neq a$ in $\T$, 
	the pair $(a, b)$ is the target pair of a loxodromic idempotent $p=p_{(a, b)}$ (see Figure 3)  
 	with attracting point $a$ and repelling point $b$.  
	Then, by results of \cite{GM-TC}, 
	the parabolic idempotent $p_a$ defined by $p_ax = a, \ \forall x \in \T$, does not admit a countable basis for its topology. 
\end{proof}  

In a similar way to the proof of Proposition \ref{ex:T}, one shows that the $G$-system $M(G)=\operatorname{Split}(\T,\Q_0)$ from Fact \ref{f:GM}.3 is tame but not Tame$_\mathbf{1}$. 

\begin{remark} \label{r:c-Helly} 
	Theorems \ref{t:lin-Helly} and \ref{t:LOT1} cannot be extended, in general, to circular orders. Indeed, the ``circular analog of Helly's space" $M_+(\T,\T)$ (which is a separable Rosenthal compactum) is not first countable. Also its subspace, the enveloping semigroup $E(\T)$ of the circularly ordered system $(H_+(\T), \T)$
	from Proposition \ref{ex:T}, is not first countable. 
	This result demonstrates once more 
	the relative complexity of circular orders 
	when compared with linear orders.
\end{remark}

\begin{prop}  \label{p:EisCODS} 
	If the enveloping semigroup $E(X)$, as a compact space, is circularly ordered, then the original compact metrizable dynamical system $X$ is \rm{Tame}$_\mathbf{2}$. 
\end{prop}
\begin{proof}  
		Since $X$ is a metrizable compactum, $E(X)$ is separable. Hence, assuming that $E(X)$ is circularly ordered, we obtain that $E(X)$ is hereditarily separable by Corollary \ref{c:normality}(3).    		
	\end{proof}

\begin{cor} \label{cor:St} 
	The Sturmian-like cascades  
	$\operatorname{Split}(\T, R_\al; A)$ are \rm{Tame}$_\mathbf{2}$.  
\end{cor}  
\begin{proof}
	We can apply Proposition \ref{p:EisCODS} because for these systems the enveloping semigroup 
	$E$ becomes (by \cite[Cor. 6.5]{GM-c}) a circularly ordered cascade, where $E=\T_{\T} \cup \Z$ is a c-ordered compact (non-metrizable) subset of 
	the c-ordered lexicographic order $\T \times \{-,0,+\}$. 
	Here, $\T_{\T}$ is homeomorphic to the two arrows space and $\Z$ is a discrete copy of the integers.
\end{proof}

It would be interesting to study for which c-ordered systems the enveloping semigroup is c-ordered. 

\bibliographystyle{amsplain}

\begin{thebibliography}{10}

\bibitem{Arens} 	
	R. Arens, \textit{Topologies for homeomorphism groups}, Amer. J. of Math., \textbf{68} (4) (1946), 593--610
	
	\bibitem{BS} 
	H. Baik and E. Samperton, 
	\textit{Spaces of invariant circular orders of groups}, Groups Geom. Dyn. \textbf{12} (2018), 721–763 
	
	
	\bibitem{BORT}   
	H. Bass, M.V. Otero-Espinar, D. Rockmore, C. Tresser, 
	\emph{Cyclic Renormalization and Automorphism Groups of Rooted Trees}, Lecture Notes in Math. 1621, Springer, 
	1995
	
	\bibitem{BC}
	S.A. Belov and V.V. Chistyakov, {\it A selection principle for mappings of bounded variation},  
	J. Math. Anal. Appl. 249 (2000), 351--366
	
	\bibitem{BenLut} 
	H.R. Bennett, D.J. Lutzer, \textit{Recent Developments in the Topology of Ordered Spaces}, Chapter 3 in: Recent Progress in General Topology II (2002), Ed.: M. Husek and J. van Mill, Elsevier 
	
	\bibitem{Borubaev}
	A.A. Borubaev, \textit{Uniform Spaces and Uniformly Continuous Mappings}, Frunze, Ilim, 1990 (Russian) 
	
	\bibitem{BuhagiarMiwa} 
	D. Buhagiar and T. Miwa, \textit{Ordered uniform completions of GO-spaces}, Topology Proc. \textbf{22} (1997), 59--80
	
	\bibitem{BM}  
	G. Bezhanishvili and P.J. Morandi, 
	\textit{Order-Compactifications of Totally Ordered Spaces: Revisited}, Order \textbf{28} (2011), 577--592
	
	\bibitem{Blatter1975}
	C.~Blatter, 
	\emph{Order compactifications of totally ordered topological spaces}, 
	J. Approx. Theory \textbf{13} (1975), 56--65
	
	\bibitem{Bour}
	J. Bourgain, 
	{\em Some remarks on compact sets of first Baire class}. 
	Bull. Soc. Math. Belg. {\bf 30} (1978), 3--10
	
	\bibitem{BFT} J. Bourgain, D.H. Fremlin and M. Talagrand,
	{\it Pointwise compact sets in Baire-measurable functions}, Amer. J. of Math.,  \textbf{100} (1978), n. 4, 845--886
	
	\bibitem{Br}
	R.B. Brook, {\it A construction of the greatest ambit}, Math. Systems Theory, {\bf 6} (1970), 243--248
	
	\bibitem{Calegari04} 
	D. Calegari, 
	{\it Circular groups, planar groups, and the Euler class}, Geometry \& Topology Monographs, 
	Volume 7: Proceedings of the Casson Fest,  431--491 (2004)
	
	\bibitem{Cech}
	E. \v{C}ech, \emph{Point Sets}, Academia, Prague, 1969
	
	\bibitem{CJ} 
	S. \v{C}ernak, J. \v{J}akub\v{i}k, \emph{Completion of a cyclically ordered group}, Czech. Math. J. \textbf{37} (1987), 157--174 
	
	\bibitem{Ch}
	V.V. Chistyakov, {\it Selections of bounded variation},  
	Journal of Applied Analysis, 
	10 (2004), 1--82
	
	\bibitem{CR-b}
	A. Clay, D. Rolfsen,\textit{ Ordered Groups and Topology}, Graduate Studies in Mathematics \textbf{176}, AMS, 2016  
	
	\bibitem{DNR} 
	B. Deroin, A. Navas and C. Rivas, {\it Groups, Orders and Dynamics,}  
	arXiv:1408.5805 (2016)
	
	\bibitem{Dulst}  
	D. van Dulst, \emph{Characterizations of Banach spaces not containing $l^1$}, 
	Centrum voor Wiskunde en Informatica, Amsterdam, 1989
	
	\bibitem{Ellis}
	R. Ellis,
	{\em Lectures on Topological Dynamics\/},
	W. A. Benjamin, Inc.\ , New York, 1969
	
	\bibitem{Eng89}
	R. Engelking, {\em  General Topology\/} 
	(revised and completed edition), 
	Heldermann Verlag, Berlin, 1989
	
	\bibitem{Fed} 
	V. Fedorchuk, \textit{Ordered spaces,} Soviet Math. Dokl., 7 (1966), 1011--1014
	
	
	\bibitem{FS}
	J.J. Font and M. Sanchis,  
	Sequentially compact subsets and monotone 
	functions: an application to Fuzzy Theory, 
	Topology and its Applications, 
	V. 192, 2015, Pages 113--122
	
	\bibitem{FP}
	S. Fuchino and Sz. Plewik, {\it On a theorem of E. Helly}, Proc. Amer. Math. Soc. \textbf{127} (2) (1999), 491--497
	
		\bibitem{Ghys}
		E. Ghys, \textit{Groups acting on the circle}, L'Enseignement Mathematique \textbf{47} (2001), 329--407  
	
	\bibitem{Gl-env} E. Glasner, {\it Enveloping semigroups in topological dynamics}, Topology Appl. {\bf 154} (2007), 2344--2363  
	
	\bibitem{GM1}
	E. Glasner and M. Megrelishvili, \emph{Linear representations of hereditarily non-sensitive dynamical systems}, Colloq. Math. 
	\textbf{104} (2006), no. 2, 223--283 
	
	
	\bibitem{GM-rose}
	E. Glasner and M. Megrelishvili, {\it Representations of dynamical systems on Banach spaces not containing $l_1$},
	Trans. Amer. Math. Soc.,  \textbf{364} (2012),  6395--6424 

	\bibitem{GM-AffComp}
	E. Glasner and M. Megrelishvili,
	\emph{Banach representations and affine compactifications of dynamical systems},
	in: Fields institute proceedings dedicated to the 2010 thematic program on asymptotic geometric analysis,
	M. Ludwig, V.D. Milman, V. Pestov, N. Tomczak-Jaegermann (Editors), Springer, New-York, 2013  
	
	\bibitem{GM-survey}
	E. Glasner and M. Megrelishvili,
	\emph{Representations of dynamical systems on Banach spaces,} in: Recent Progress in General Topology III, 
	(Eds.: K.P. Hart, J. van Mill, P. Simon),  
	Springer-Verlag, Atlantis Press, 2014, 399-470 
	
	\bibitem{GM-tame}
	E. Glasner and M. Megrelishvili, 
	{\it Eventual nonsensitivity and tame dynamical systems},  
	arXiv:1405.2588, 2014 
	
	\bibitem{GM-tLN}
	E. Glasner and M. Megrelishvili,
	\emph{More on tame dynamical systems}, in: Lecture Notes S. vol. 2013, Ergodic Theory and Dynamical Systems in their Interactions with Arithmetics and Combinatorics, Eds.: S. Ferenczi, J. Kulaga-Przymus, M. Lemanczyk,  Springer, 2018, pp. 351--392  
	
	\bibitem{GM-c}
	E. Glasner and M. Megrelishvili,
	\emph{Circularly ordered dynamical systems,} 
	Monats. Math. \textbf{185} (2018), 415--441
	
	\bibitem{GM-UltraHom21} 
	E. Glasner and M. Megrelishvili, 
	\textit{Circular orders, ultrahomogeneity and topological groups}, ArXiv:1803.06583, 2018. 
	AMS Contemporary Mathematics book series volume \textbf{772} "Topology, Geometry, and Dynamics: Rokhlin-100" (ed.: A.M. Vershik, V.M. Buchstaber, A.V. Malyutin)  2021, pp. 133--154
	
	\bibitem{GM-TC} 
	E. Glasner and M. Megrelishvili, 
	\textit{Todor\u{c}evi\'{c}' Trichotomy and a hierarchy in the class of tame dynamical systems}, Trans. Amer. Math. Soc. \textbf{375}, 4513--4548 (2022)  
		
	\bibitem{GMU08} E. Glasner, M. Megrelishvili and V.V. Uspenskij,
	\emph{On metrizable enveloping semigroups}, Israel Journal of Math., \textbf{164}, 317--332 (2008)
	
	\bibitem{HK} 
	N. Hindman, R.D. Kopperman, \emph{Order compactifications of discrete semigroups}, 
	Topology Proc. \textbf{27}, 479--496 (2003)
	
	\bibitem{HLZ} 
	R.W. Heath, D.J. Lutzer and P.L. Zenor, \textit{Monotonically Normal Spaces}, 
	Trans. Amer. Math. Soc. \textbf{178} (1973), 481--493
	
	\bibitem{Hunt16}
	E.V. Huntington, \textit{A set of independent postulates for cyclic order,} Proceedings of the National Academy of Sciences of the United States of America, \textbf{2} (1916), 630--631
	
	\bibitem{Hunt24}
	E.V. Huntington,  \textit{Sets of completely independent postulates for cyclic order,} Proceedings of the National Academy of Sciences of the United States of America, \textbf{10} (1924), 74--78 

	\bibitem{Isb} J. Isbell, \emph{Uniform spaces}, Providence, 1964 
	
	\bibitem{Kaufman1967}
	R.~Kaufman, 
	\emph{Ordered sets and compact spaces}, 
	Colloq. Math. \textbf{17} (1967), 35--39
	
	\bibitem{Kem}
	N. Kemoto, \textit{The lexicographic ordered products and the usual Tychonoff products,} Topology Appl. \textbf{162} (2014), 20--33 
	
	\bibitem{KentRichmond1988}
	D.C. Kent and T.R. Richmond, 
	\emph{Ordered compactification of totally ordered spaces}, Internat. J. Math. Math. Sci. \textbf{11} (1988), 683--694
	
	\bibitem{KL} D. Kerr and H. Li,
	{\it Independence in topological and $C^*$-dynamics}, Math. Ann.
	{\bf 338} (2007), 869--926
	
	\bibitem{Kok}
	H. Kok, \emph{Connected orderable spaces},
	Math. Centhre Tracts {\bf 49}, Math. Cent. Tracts,   Amsterdam, 1973
	
	\bibitem{Ko}
	A. K\"{o}hler, {\em Enveloping semigroups for flows}, Proc.
	of the Royal Irish Academy, {\bf 95A} (1995), 179--191
	
	\bibitem{KS} 
	K.L. Kozlov and B.V. Sorin, \textit{Enveloping semigroups as compactifications of topological groups},   arXiv:2509.17577, 2025  
	
	\bibitem{LB}  
	D.J. Lutzer, H.R. Bennett, 
	\textit{Separability, the countable chain condition and the 
		Lindelof property in linearly orderable spaces}, Proc. AMS, \textbf{23} (1969), pp. 664--667  
	
	\bibitem{McMullen07} 
	C. McMullen, Ribbon $\R$-trees and holomorphic dynamics on the unit disk, J. of Topology, v. 2, 2009, p. 23--76 
	
	\bibitem{Me-EqComp84}
	M. Megrelishvili,
	{\it Equivariant completions and compact extensions},
	Bull. Ac. Sc. Georgian SSR, {\bf 115:1} (1984), 21--24
	
	\bibitem{Me-Ex88}
	M. Megrelishvili,
	{\it A Tychonoff G-space not admitting a compact Hausdorff G-extension or a G-linearization},
	Russian Math. Surveys {\bf 43:2} (1988), 177--178 

	\bibitem{Me-nz}
	M. Megrelishvili, {\em Fragmentability and representations of flows\/}, Topology Proceedings, {\bfseries 27:2} (2003), 497--544. 
	For the updated version see arXiv:math/0411112 
	
	\bibitem{Me-cs07}
	M. Megrelishvili,
	\emph{Compactifications of semigroups and semigroup actions},
	Topology Proc. \textbf{31:2} (2007), 611--650 
	
	
	\bibitem{Me-Helly}
	M. Megrelishvili, \emph{A note on tameness of families having bounded variation}, 
	Topology Appl. \textbf{217} (2017), 20--30 
	
	\bibitem{Me-medianBV} 
	M. Megrelishvili, 
	\textit{Median pretrees and functions of bounded variation}, 
	Topology Appl. \textbf{285} (2020)
	
	\bibitem{Me-OrdSem} 
	M. Megrelishvili,
	\textit{Orderable groups and semigroup compactifications,} Monatsch. Math. \textbf{200} (2022), 903--932 
	
	\bibitem{Me-MaxEqComp}  
	M. Megrelishvili, 
	\textit{Maximal equivariant compactifications}, Topology Appl. \textbf{329} (2023), 108372 
	
	\bibitem{Me-AutKey} 
	M. Megrelishvili, 
	\textit{Key subgroups in the Polish group of all automorphisms of the rational circle}, arXiv:2410.17905, 2024 
	
	\bibitem{Me-b}
	M. Megrelishvili, 
	\textit{Topological Group Actions and Banach Representations}, unpublished book, Available on author's homepage, 2025
	
	\bibitem{MK} 
	T. Miwa, N. Kemoto, Linearly ordered extensions of GO spaces, Topology Appl. \textbf{54} (1993), 133--140
	
	\bibitem{Nach}
	L. Nachbin, \emph{Topology and order,} Van Nostrand Math. Studies, Princeton, New Jersey, 1965 
	
	\bibitem{Nagata}
	J. Nagata, \textit{Modern General Topology}, North Holland, 1985 
	
	\bibitem{Natanson} 
	I.P. Natanson, {\it Theory of functions of real variable}, v. I, New York, 1964  
	
	\bibitem{N}
	I. Namioka, {\em Radon-Nikod\'ym compact spaces and
		fragmentability\/}, Mathematika 34 (1987), 258--281
	
	\bibitem{van-the} 
	L. Nguyen van Th\'{e}, \emph{More on the Kechris-Pestov-Todorcevic correspondence: precompact expansions}, 
	Fund. Math. \textbf{222} (2013), no. 1, 19--47
	
	\bibitem{Novak-cuts} 
	V. Novak, {\it Cuts in cyclically ordered sets}, Czech.  Math. J. \textbf{34}, 322--333 (1984) 
	
	\bibitem{Ovch}  
	S. Ovchinnikov, \textit{Topological automorphism groups of chains}, Mathware Soft Comput., v. 8, n. 1, 47--60, 2001
	
	\bibitem{Pest98} V. Pestov, \emph{On free actions, minimal flows and a problem by Ellis}, Trans. Amer. Math. Soc. {\bf 350} (1998), 4149--4165 
	
	\bibitem{Pest-Smirnov}
	V. Pestov, \textit{A topological transformation group without  non-trivial equivariant compactifications}, 
	Advances in Math. \textbf{311} (2017), 1--17
	
	\bibitem{Reid} 
	C.D. Reid, 
	\textit{Decomposition of locally compact coset spaces}, 
	J. of the London Math. Soc, DOI: 10.1112/jlms.12688 (2022) 
	
	\bibitem{Ro} H.P. Rosenthal, 
	\emph{A characterization of Banach spaces containing $l^1$}, Proc. Nat. Acad. Sci. U.S.A., 71 (1974), 2411--2413
	
	\bibitem{Ros1} H.P. Rosenthal,
	\emph{Point-wise compact subsets of the first Baire class}, Amer. J. of Math., 99:2 (1977), 362--378  
	
	\bibitem{Sorin-B-22}
	B.V. Sorin, 
	\textit{Compactifications of Homeomorphism Groups
		of Linearly Ordered Compacta}, Math. Notes, \textbf{112} (2022) 
	
	\bibitem{Sorin23-Roelcke}
	G.B. Sorin, 
	The Roelcke Precompactness and Compactifications of Transformations Groups of Discrete Spaces and Homogeneous Chains, 2023, arXiv:2310.18570
	
	\bibitem{Sorin24} 
	G.B. Sorin, 
	\textit{Lattices of Extensions of Cyclically Ordered Sets
		and Compactifications of Generalized Cyclically Ordered Spaces}, Mathematical Notes, 116 (2024), 763--776 
	
	\bibitem{Sorin25-Sbornik} 
	G.B. Sorin, Semigroup compactifications of groups of automorphisms of ultra-homogeneous cyclically ordered sets, Math. Sbornik (in Russian), 217:2 (2026), 154--179  
	
	\bibitem{SS}
	L.A. Steen, J.A. Seebach, \textit{Counterexamples in Topology}, Dover Publications, New York, 1978  
	
	\bibitem{Struve} 
	R. Struve, \textit{Cyclic order: a geometric analysis,}
	 Beiträge zur Algebra und Geometrie,  
	 \textbf{61} (2020), 649--669
	
	\bibitem{Tal}
	M. Talagrand, \emph{Pettis integral and measure theory,} Mem. AMS No. 51, 1984 
	
	
	\bibitem{Vr-Embed77} J. de Vries, 
	\emph{Equivariant embeddings of $G$-spaces}, in: J. Novak (ed.), General Topology and its Relations to Modern Analysis and Algebra IV, Part B, Prague, 1977, 485--493
	
	\bibitem{Zheleva97} 
	S. Zheleva, 
	\textit{Representation of right cyclically ordered groups as groups of automorphisms of a cyclically ordered set}, Math. Balkanica \textbf{11}, (1997), 291--294 
\end{thebibliography}

\end{document}